\documentclass[twoside, 11pt]{amsart}

\usepackage{amsfonts}
\usepackage[utf8]{inputenc}
\usepackage{array}

\usepackage[english]{babel}
\usepackage{latexsym}
\usepackage{amscd}
\usepackage{amsmath}
\usepackage{amssymb}
\usepackage{amsthm}
\usepackage{graphicx}
\usepackage{mathrsfs}
\usepackage{mathtools}
\usepackage{pict2e}

\usepackage{stackrel} 
\usepackage{wasysym}
\usepackage[all]{xy}
\usepackage[dvipsnames]{xcolor}
\usepackage[colorlinks,final,hyperindex]{hyperref}
\usepackage[noabbrev,capitalize]{cleveref}
\usepackage{tikz}
\usepackage{tikz-cd}
\usetikzlibrary{decorations.pathmorphing,decorations.markings,arrows,calc,shapes.geometric,arrows.meta,positioning}

\usepackage{geometry}

\allowdisplaybreaks

\newtheorem{lemma}{Lemma}[section]

\newtheorem{theorem}[lemma]{Theorem}

\newtheorem{corollary}[lemma]{Corollary}
\newtheorem{proposition}[lemma]{Proposition}

\theoremstyle{definition}
\newtheorem{definition}[lemma]{Definition}
\newtheorem{remark}[lemma]{\sc Remark}
\newtheorem{example}[lemma]{\sc Example}
\newtheorem{examples}[lemma]{\sc Examples}

\newtheorem{assumption}{\sc Assumptions}

\newcommand{\nocontentsline}[3]{}
\newcommand{\tocless}[2]{\bgroup\let\addcontentsline=\nocontentsline#1{#2}\egroup}

\hypersetup{citecolor=BleuMinuit, linkcolor=bordeau, filecolor=black, urlcolor=BleuMinuit}

\author{Coline Emprin}
\title{Kaledin classes and formality criteria}
\address{Coline Emprin, D\'epartement de math\'ematiques et applications, \'Ecole normale sup\'erieure,
	\indent  45 rue d’Ulm, 75230 Paris, France. \textcolor{white}{cccccccccccccccccccccccccccccccccccccccccccccccc}	\indent Laboratoire de G\'eom\'etrie, Analyse et Applications, Universit\'e Sorbonne Paris Nord, UMR \indent 7539, 93430, Villetaneuse, France.}

\email{\noindent \href{mailto:coline.emprin@ens.psl.eu}{coline.emprin@ens.psl.eu}}

\date{April 26, 2024}

\thanks{
	\indent The author was supported by the project ANR-20-CE40-0016 HighAGT and \'Ecole Normale Sup\'erieure.}

\begin{document}
\definecolor{red}{RGB}{230,97,0}
\definecolor{blue}{RGB}{93,58,155}
\definecolor{Chocolat}{rgb}{0.36, 0.2, 0.09}
\definecolor{BleuTresFonce}{rgb}{0.215, 0.215, 0.36}
\definecolor{BleuMinuit}{RGB}{0, 51, 102}
\definecolor{bordeau}{rgb}{0.5,0,0}
\definecolor{turquoise}{RGB}{6, 62, 62}
\definecolor{rose}{RGB}{235, 62, 124}

\newcommand{\coline}[1]{\textcolor{rose}{#1}}

\newcommand{\hooklongrightarrow}{\lhook\joinrel\longrightarrow}


\newcommand{\chdiscr}{\mathsf{discr.\ Ch}}
\newcommand{\sSe}{\mathsf{sSet}}
\newcommand{\sLialg}{\ensuremath{\mathcal{sL}_\infty\text{-}\,\mathsf{alg}}}
\newcommand{\isLialg}{\ensuremath{\infty\text{-}\,\mathcal{sL}_\infty\text{-}\,\mathsf{alg}}}
\newcommand{\sLiealg}{\ensuremath{\mathcal{s}\lie\,\text{-}\,\mathsf{alg}}}
\def\cD{\Delta}

\newcommand{\ucomnalg}{\ensuremath{\mathrm{uCom}_{\leslant 0}\text{-}\,\mathsf{alg}}}
\newcommand{\fnQsp}{\ensuremath{\mathsf{ft}\mathbb{Q}\text{-}\mathsf{ho}(\mathsf{sSet})}}
\newcommand{\fnsp}{\ensuremath{\mathsf{fn}\text{-}\mathsf{Sp}}}
\newcommand{\nQsp}{\ensuremath{\mathsf{n}\mathbb{Q}\text{-}\mathsf{Sp}}}
\newcommand{\nsp}{\ensuremath{\mathsf{n}\text{-}\mathsf{Sp}}}
\newcommand{\fQalg}{\ensuremath{\mathsf{ft}\text{-}\mathsf{ho}(\mathrm{uCom}_{\leqslant 0}\text{-}\,\mathsf{alg})}}

\newcommand{\coker}{\operatorname{coker}}
\newcommand{\Ran}{\operatorname{Ran}}
\newcommand{\ho}{\operatorname{ho}}
\newcommand{\Mal}{\operatorname{Mal}}
\newcommand{\Loc}{\operatorname{Loc}}


\newcommand{\sLi}{\ensuremath{\mathcal{sL}_\infty}}
\newcommand{\ccoLi}{\ensuremath{\mathrm{cBCom}}}
\newcommand{\sLie}{\ensuremath{\mathcal{s}\mathrm{Lie}}}
\newcommand{\com}{\ensuremath{\mathrm{Com}}}
\newcommand{\ucom}{\ensuremath{\mathrm{uCom}}}
\newcommand{\Cobar}{\ensuremath{\Omega}}
\newcommand{\hatCobar}{\ensuremath{\widehat{\Omega}}}
\renewcommand{\Bar}{\ensuremath{\mathrm{B}}}
\newcommand{\hatBar}{\ensuremath{\widehat{\mathrm{B}}}}
\newcommand{\lie}{\ensuremath{\mathrm{Lie}}}
\newcommand{\uhocom}{\ensuremath{\mathrm{u}\Omega\mathrm{BCom}}}

\newcommand{\Sh}{\ensuremath{\mathrm{Sh}}}


\newcommand{\g}{\ensuremath{\mathfrak{g}}}
\newcommand{\wsLi}{\ensuremath{\widehat{\mathcal{sL}_\infty}}}


\newcommand{\Rh}{\ensuremath{\mathrm{R^h}}}
\renewcommand{\L}{\ensuremath{\mathscr{L}}}
\newcommand{\Li}{\mathfrak{L}}
\newcommand{\APL}{\mathrm{A_{PL}}}
\newcommand{\CPL}{\mathrm{C_{PL}}}


\newcommand{\antishriek}{\text{\raisebox{\depth}{\textexclamdown}}}
\newcommand{\RT}{\mathrm{RT}}
\newcommand{\LRT}{\mathrm{LRT}}
\newcommand{\Sy}{\mathbb{S}}
\renewcommand{\d}{\ensuremath{\mathrm{d}}}
\newcommand{\PP}{\ensuremath{\mathrm{P}}}
\def\Ho#1#2{\Lambda^{#2}_{#1}}
\def\De#1{\Delta^{#1}}
\newcommand{\NN}{\mathbb{N}}
\newcommand{\RR}{\mathbb{R}}
\def\BCH{\mathrm{BCH}}
\newcommand{\wPT}{\ensuremath{\mathrm{wPT}}}
\newcommand{\oPT}{\ensuremath{\overline{\mathrm{PT}}}}
\newcommand{\PaRT}{\ensuremath{\mathrm{PaRT}}}
\newcommand{\PaPT}{\ensuremath{\mathrm{PaPT}}}
\newcommand{\PaPRT}{\ensuremath{\mathrm{PaPRT}}}
\def\rmC{\mathrm{C}}
\newcommand{\berglund}{\ensuremath{\mathcal{B}}}
\def\hot{\widehat{\otimes}} 
\def\whk{\widehat{k}}

\def\colim{\mathop{\mathrm{colim}}}

\newcommand{\CC}{\ensuremath{\mathrm{CC}_\infty}}


\newcommand{\Lalg}{\ensuremath{\mathscr{L}_\infty\text{-}\mathsf{alg}}}

\newcommand{\F}{\ensuremath{\mathrm{F}}}
\newcommand{\h}{\ensuremath{\mathfrak{h}}}

\newcommand{\QQ}{\ensuremath{\mathbb{Q}}}
\newcommand{\D}{\ensuremath{\mathscr{D}}}
\renewcommand{\P}{\ensuremath{\mathcal{P}}}
\newcommand{\C}{\ensuremath{\mathcal{C}}}
\newcommand{\VdL}{\ensuremath{\mathrm{VdL}}}
\newcommand{\ch}{\ensuremath{\mathrm{Ch}}}
\newcommand{\End}{\ensuremath{\mathrm{End}}}
\newcommand{\Aut}{\ensuremath{\mathrm{Aut}}}
\newcommand{\eend}{\ensuremath{\mathrm{end}}}
\newcommand{\susp}{\ensuremath{\mathscr{S}}}
\newcommand{\T}{\ensuremath{\mathcal{T}}}
\newcommand{\vdl}{\ensuremath{\mathrm{VdL}}}
\newcommand{\Tw}{\ensuremath{\mathrm{Tw}}}
\newcommand{\Hom}{\ensuremath{\mathrm{Hom}}}
\renewcommand{\S}{\ensuremath{\mathbb{S}}}
\renewcommand{\k}{\ensuremath{\mathbb{K}}}
\newcommand{\id}{\ensuremath{\mathrm{id}}}
\newcommand{\MC}{\ensuremath{\mathrm{MC}}}
\newcommand{\mc}{\ensuremath{\mathfrak{mc}}}
\newcommand{\mclie}{\ensuremath{\overline{\mathfrak{mc}}}}
\newcommand{\dgl}{\ensuremath{\mathsf{dgLie}}}

\newcommand{\B}{\mathcal{B}}
\newcommand{\Z}{\mathbb{Z}}

\newcommand{\ad}{\operatorname{ad}}
\newcommand{\PTt}{\ensuremath{\widetilde{\mathrm{PT}}}}
\newcommand{\PT}{\ensuremath{\mathrm{PT}}}

\newcommand{\A}{\ensuremath{\mathrm{A}}}

\makeatletter

	\begin{abstract}
We develop a general obstruction theory to the formality of algebraic structures over any commutative ground ring. It relies on the construction of Kaledin obstruction classes that faithfully detect the formality of differential graded algebras over operads or properads, possibly colored in groupoids. The present treatment generalizes the previous obstruction classes in two directions: outside characteristic zero and including a wider range of algebraic structures. This enables us to establish novel formality criteria, including formality descent with torsion coefficients, formality in families, intrinsic formality, and criteria in terms of chain-level lifts of homology automorphism.  
\end{abstract}

	\maketitle

\makeatletter
\def\@tocline#1#2#3#4#5#6#7{\relax
	\ifnum #1>\c@tocdepth 
	\else
	\par \addpenalty\@secpenalty\addvspace{#2}%
	\begingroup \hyphenpenalty\@M
	\@ifempty{#4}{%
		\@tempdima\csname r@tocindent\number#1\endcsname\relax
	}{%
		\@tempdima#4\relax
	}%
	\parindent\z@ \leftskip#3\relax \advance\leftskip\@tempdima\relax
	\rightskip\@pnumwidth plus4em \parfillskip-\@pnumwidth
	#5\leavevmode\hskip-\@tempdima
	\ifcase #1
	\or\or \hskip 1em \or \hskip 2em \else \hskip 3em \fi%
	#6\nobreak\relax
	\hfill\hbox to\@pnumwidth{\@tocpagenum{#7}}\par
	\nobreak
	\endgroup
	\fi}

\newcommand{\enableopenany}{%
	\@openrightfalse%
}
\makeatother
	
	\setcounter{tocdepth}{1}
	\tableofcontents

\section*{\textcolor{bordeau}{Introduction}}

The notion of formality originated in the field of rational homotopy theory. In this context, a topological space $X$ is \emph{formal} if one can recover all its rational homotopy type from its cohomology ring, i.e. if there exists a zig-zag relating the singular cochains to its cohomology
  \[ C_{\mathrm{sing}}^*(X; \mathbb{Q}) \overset{\sim}{\longleftarrow} \cdot \overset{\sim}{\longrightarrow} \cdots \overset{\sim}{\longleftarrow} \cdot \overset{\sim}{\longrightarrow} H_{\mathrm{sing}}^*(X; \mathbb{Q})\ , \] through quasi-isomorphisms, i.e. morphisms of associative algebras inducing isomorphisms in cohomology. When dealing with a formal space, the singular cochains can be replaced by the cohomology, which is usually a more calculable object. For instance, the formality of spheres allowed Sullivan in \cite{Sul77} to recover the calculation of rational homotopy groups of spheres established by Serre in \cite{Ser53}. A seminal formality result was given by Deligne, Griffiths, Morgan, and Sullivan who proved in \cite{DGMS75}, that any compact Kähler manifold is formal.  The formality of topological spaces appears as a particular case of a much more general notion: the one of algebraic structures. Let $A$ be a chain complex over a commutative ring $R$. Let us consider a certain differential graded (dg) algebra structure $\varphi$ over $A$ (e.g. a dg associative algebra, a dg Hopf algebra, a dg Poisson algebra,etc.). It induces a structure of the same type $\varphi_*$ at the homology level. The same question then arises : does this induced structure keeps all the homotopical information that was at the level of the original algebra? In other words, does there exist a zig-zag of quasi-isomorphisms preserving the algebraic structures and relating the algebra to the induced structure in homology? If so, the algebra is said formal.   \bigskip

\paragraph{\bf Massey products} In \cite{Mas58}, Massey associated to any dg associative algebra, higher order homology operations, known today as the \emph{Massey products}. These products turn out to be obstructions to formality: the non-vanishing of only one of them prevents a dg associative algebra from being formal, see \cite[Theorem~4.1]{DGMS75}. Unfortunately, the converse is false: the vanishing of all the Massey products does not imply formality, see e.g. Halperin--Stasheff \cite{HJ79}. To prove the formality of compact Kähler manifolds, the authors of \cite{DGMS75} mainly use Hodge theory. Nevertheless, their intuition came from Massey products and Weil’s conjectures insights. Let $(A, \varphi)$ be a dg associative algebra and let $\alpha$ be a unit of infinite order in $R$. The grading automorphism $\sigma_{\alpha}$ is the homology automorphism defined as the multiplication by $\alpha^k$ in homological degree $k$. Sullivan proved in \cite{Sul77} that if $(A, \varphi)$ satisfies a purity property, i.e. if $\sigma_{\alpha}$ admits a chain level lift, the algebra is formal. The heuristic is the following one: if there exists such a lift, the Massey products have to be compatible with it. Thus, they are going to intertwine multiplication by different powers of $\alpha$. If the unit has infinite order, all Massey products have to vanish. This is nonetheless not sufficient to ensure formality, and the proof of \cite{Sul77} do not make use of Massey products. Sullivan's result was latter widely generalized to other types algebraic structures and to any coefficient ring in \cite{GNPR05}, \cite{CH20},   \cite{DCH21} and \cite{CH22}. \bigskip

\noindent \textbf{Kaledin classes.} In order to construct a refinement of the Massey products which would characterize formality, Kaledin introduced in \cite{Kal07} new obstruction classes, nowadays called after his name. He associates to any dg associative algebra its \emph{Kaledin class} in Hochschild cohomology, which vanishes if and only if the algebra is formal. He was motivated by establishing formality criteria for families of dg associative algebras, for which the methods of \cite{DGMS75} do not apply. The work of Kaledin was extended by Lunts in \cite{Lun07} for homotopy associative algebras over a $\mathbb{Q}$-algebra and by Melani and Rubió in \cite{MR19} for algebras over a binary Koszul operad in characteristic zero. \medskip

\noindent One can ask for the formality of a wide range of other algebraic structures: operads themselves, structures involving operations with several inputs but also several outputs such as dg Frobenius bialgebras, dg involutive Lie bialgebras etc. Such structures are encoded by generalizations of operads: colored operads and properads. The present article generalizes the Kaledin classes construction to study formality 
\begin{itemize}
		\item[$\centerdot$] of any algebra encoded by a groupoid colored operad or properad;
		\item[$\centerdot$] over any commutative ground ring $R$. 
\end{itemize}

  \noindent We use the resulting Kaledin classes to establish new formality criteria: formality descent, an intrinsic formality criterium, formality in families, and formality criteria in terms of chain level lifts of certain homology automorphisms. On the one hand, this enables us to recover and incorporate previous results into a single theory. On the other hand, this allows us to address new formality problems, such as the formality of algebras over properads and formality results with coefficients in any commutative ring. \bigskip
 
   \noindent \textbf{Formality as a deformation problem.} 
   The Kaledin classes generalization relies on an operadic approach which reduces the formality to a deformation problem. In general, the induced structure forgets a part of the homotopical information: most structures are not formal. Under some hypotheses, for instance if $R$ is a field, there is another way to transfer an algebraic structure to the homology without loss of homotopical information. This is the homotopy transfer theorem, see e.g. \cite[Section~10.3]{LodayVallette12}, \cite{berglund}, \cite[Section~4]{PHC}, \cite[Section~5.5]{LRL23} and references therein. This leads to a transferred structure on the homology  \[\varphi_t \coloneqq \varphi_* + \varphi^{(2)} + \varphi^{(3)} + \cdots\] made up of the induced structure $\varphi_*$ and higher order products. The algebra is then formal if these higher order products are trivial in a suitable sense. Thus, formality boils down to the following deformation problem:
 
 \begin{center}
 	\emph{Is the deformation $\varphi_t$ of the induced structure $\varphi_*$ trivial?}
 \end{center}
 \noindent The philosophy of deformation theory which goes back to Deligne, Gerstenhaber, Goldman, Grothendieck, Millson, Schlessinger, Stasheff, and many others, tells that any deformation problem is encoded by a dg Lie algebra, see \cite{Pri10} and \cite{Lurie11}. This is the case for the formality of algebraic structures. Thanks to the operadic calculs, one can treat algebraic structures of a certain type $\C$ on the homology $H(A)$ as Maurer--Cartan elements of a convolution dg Lie algebra $\mathfrak{g}_{\C ,H(A)}$ see \cite[Section~10]{LodayVallette12} and \cite[Proposition 11]{Merkulov_2009}. This dg Lie algebra is for instance the Hochschild complex in the case of associative structures or the Harrison complex in the case of commutative structures. As a dg Lie algebra, the set of degree zero elements of the convolution dg Lie algebra forms a group under the Backer--Campbell--Hausdorff formula, called the gauge group. Its acts on the set of Maurer--Cartan elements through the gauge action. Two elements in the same orbit are said gauge equivalent. This leads to the following definition. A differential graded algebra structure $(A, \varphi)$ admitting a transferred structure is
 
 \begin{itemize}
 	\item[$\centerdot$] \emph{gauge formal} if the structures $\varphi_t$ and $\varphi_*$ are gauge equivalent;
 	\item[$\centerdot$] \emph{gauge $n$-formal} for $n \geqslant 0$, if $\varphi_t$ is gauge equivalent to a structure of the form \[\psi = \varphi_* + \psi^{(n+2)} + \cdots\]
 \end{itemize}
\noindent Over a characteristic zero field, an algebra is formal if an only if it is gauge formal. In the general case, these two notions are related but none of them imply the other, see Proposition \ref{formality-gauge}. Gauge formality generalizes to any weight-graded dg Lie algebra: a Maurer--Cartan element in a weight-graded dg Lie algebras is gauge formal if it is gauge equivalent to its weight one component.

 \bigskip

\noindent \textbf{Outline of the paper.} In Section \ref{section1}, we study the gauge formality of any Maurer--Cartan element $\varphi$ in a weight-graded dg Lie algebra over a $\mathbb{Q}$-algebra. To so do, we construct a certain homology class, its \emph{Kaledin class} $K_{\varphi}$, which vanishes if and only if $\varphi$ is gauge formal, see Theorem \ref{A}. Over a $\mathbb{Q}$-algebra, the gauge formality of algebraic structures can be addressed using Theorem \ref{A} in $\mathfrak{g}_{\C ,H(A)} $. This particular dg Lie algebra comes equipped with additional structures: a convolution product and a Lie-admissible product whose skew-symmetrization defines the Lie bracket. By reformulating the problem in this setting,
it turns out that the characteristic zero can be dispensed with. In Section \ref{section2}, after having generalized operadic calculus results to any coefficient ring, we establish the following theorem. \bigskip

\noindent \textbf{Theorem \ref{B}.} \textit{Let $R$ be a commutative ring. Let $\mathbb{V}$ be a groupoid and let $\C$ be a $\mathbb{V}$-colored reduced weight-graded dg cooperad over $R$. Let $(A, \varphi)$ be a $\Cobar \C$-algebra structure admitting a transferred structure.} 
	
\begin{enumerate}
	\item \textit{ Let $n \geqslant 1$ be an integer such that $n !$ is a unit in $R$. The algebra $(A,\varphi)$ is gauge $n$-formal if and only if the reduction of Kaledin class modulo $\hbar^n$ is zero.}
	
	\item \textit{If $R$ is a $\mathbb{Q}$-algebra, the algebra $(A,\varphi)$ is gauge formal if and only if its Kaledin class vanishes.}  
\end{enumerate}
\bigskip

\noindent In Section \ref{3}, we establish a properadic version of Theorem \ref{B} by defining Kaledin classes for algebras over properads. Theorem \ref{B} and its properadic version are used in Section \ref{4} to establish new formality criteria and results. Section \ref{4:1.4} establishes formality descent results which apply in particular for colored operads and properads or with torsion coefficients. This allows us to prove new formality results with torsion coefficients, see e.g. Theorem \ref{new} dealing with $\mathbb{Z}_{(p)}$ coefficients. Section \ref{4:3.4} studies the problem motivating the introduction of Kaledin classes in the first place: formality in families. Section \ref{4:2.4} presents an intrinsic formality criterion, allowing us to recover standard intrinsic formality results, such as the one at the heart of Tamarkin's proof of Kontsevich formality theorem, see \cite{Hinich03bis}.  \bigskip

\noindent Section \ref{critere} is dedicated to formality criteria in terms of chain-level lifts of certain homology automorphisms. Using Kaledin classes, we generalize the grading automorphism criteria to any algebra encoded by a groupoid colored operad or a properad. We also address the following question: is the grading automorphism $\sigma_{\alpha}$, the only homology automorphism satisfying this property? In Theorem \ref{C}, we settle conditions on an homology automorphism so that the existence a chain level lift implies formality. This condition is satisfied by the Frobenius action in the $\ell$-adic cohomology of any smooth projective variety thanks to the Weil conjectures and Riemann hypothesis for finite fields. This leads to the following result. \bigskip

\noindent\textbf{Theorem \ref{smooth proj2}.} \textit{Let $\mathbb{V}$ be a groupoid and let $\P$ be a $\mathbb{V}$-colored operad in sets. Let $p$ be a prime number. Let $K$ be a finite extension of $\mathbb{Q}_p$ and let $K \hookrightarrow \mathbb{C}$ be an embedding. Let $X$ be a $\P$-algebra in the category of smooth and proper schemes over $K$ of good reduction, i.e. for which there exists a smooth and proper model $\mathcal{X}$ over the ring of integers $\mathcal{O}_K$. The dg $\P$-algebra of singular chains $C_*(X_{\mathrm{an}}, \mathbb{Q})$ is formal.	} 
\bigskip

\noindent As an application to this theorem, one recovers the formality of the cyclic operad of moduli spaces of stable algebraic curves established in \cite{GNPR05}. This also opens the doors to establishing new formality results for algebraic structures in smooth and proper schemes, which will be the subject of a forthcoming article.

\bigskip

\noindent \textbf{Notations and conventions}. 

\begin{itemize}
	\item[$\centerdot$] We work over a commutative ground ring $R$. All tensor products are taken over $R$ and every morphism is $R$-linear unless otherwise specified.
	\item[$\centerdot$]  We adopt a homological point of view and we work in the symmetric monoidal category of chain complexes over $R$ with the Koszul sign rule. 
	\item[$\centerdot$] If $A$ is a chain complex and $x \in A$ is a homogeneous element, we denote by $|x|$ its homological degree.
	\item[$\centerdot$] The abbreviation ``dg'' stands for the words ``differential graded''.
	\item[$\centerdot$]  We use the notations of \cite{LodayVallette12} for operads and the ones of \cite{PHC} for properads.  \bigskip
\end{itemize}

\noindent \textbf{Acknowledgments}.
I would like to thank my advisors, Geoffroy Horel and Bruno Vallette for their constant support, the many fruitful discussions and their invaluable guidance. I am grateful to Victor Roca i Lucio for useful discussions of results of \cite{LRL23} and to Olivier Benoist and Arnaud Vanhaecke for enlightening conversations around \cite{SGA4}.

\section{\textcolor{bordeau}{Kaledin classes in differential graded Lie algebras}}\label{section1}

In this section, we construct Kaledin classes in any weight-graded dg Lie algebra. This allows us to detect whether a given Maurer--Cartan element is gauge equivalent to its first component. In all this section, the ground ring $R$ is a $\mathbb{Q}$-algebra.

\subsection{Deformation theory with dg Lie algebras}\label{1.1}
We start by recalling the deformation theory controlled by a dg Lie algebra, see e.g. \cite[Chapter~1]{DSV22} for more details.

\begin{definition}[Differential graded Lie algebra]
	A \emph{differential graded (dg) Lie algebra} $\left(\mathfrak{g}, [-,-], d, \right)$ is a graded $R$-module $\mathfrak{g}$ equipped with 
	
	\begin{itemize}
		\item[$\centerdot$] a Lie bracket $[-,-] : \mathfrak{g} \otimes \mathfrak{g} \to \mathfrak{g}$~, i.e. homogeneous map of degree $0$~, satisfying the Jacoby identity and the antisymmetry properties; 
		\item[$\centerdot$] a differential $d : \mathfrak{g} \to \mathfrak{g}$~, i.e. a degree $-1$ derivation that squares to zero.
	\end{itemize}
	A dg Lie algebra $\mathfrak{g}$ said \emph{complete} if it admits a complete descending filtration $\mathcal{F}$~, i.e. a decreasing chain of sub-complexes \[\mathfrak{g} = \mathcal{F}^1 \mathfrak{g} \supset \mathcal{F}^2 \mathfrak{g} \supset \mathcal{F}^3 \mathfrak{g} \supset \cdots  \] compatible with the bracket $\left[\mathcal{F}^n \mathfrak{g}, \mathcal{F}^m \mathfrak{g} \right] \subset \mathcal{F}^{n+m} \mathfrak{g}$ and such that the canonical projections $\pi_n : \mathfrak{g} \twoheadrightarrow  \mathfrak{g} /\mathcal{F}^n \mathfrak{g}$ induce an isomorphism \[\pi : \mathfrak{g} \to \lim_{n \in \mathbb{N}} \mathfrak{g} / \mathcal{F}^n\mathfrak{g} \ . \] In the sequel, we will use the same notation for the underlying chain complex $\mathfrak{g}$ and the full data of a complete dg Lie algebra.
\end{definition}

\begin{example}
	Let $\mathfrak{g}$ be a dg Lie algebra. The set of its derivations forms a dg Lie algebra denoted $\mathrm{Der} (\mathfrak{g})$. Its grading comes from the degree of derivations, the bracket is defined by \[\lbrace \delta, \delta' \rbrace \coloneqq \delta \circ \delta' - (-1)^{|\delta | |\delta '|} \delta ' \circ \delta \ ,\] and the differential is given by $\delta \mapsto \lbrace d , \delta \rbrace$. The dg Lie subalgebra of \emph{inner derivations} $\mathrm{ad}(\mathfrak{g}) \subset \mathrm{Der} (\mathfrak{g})$ is the image of the dg Lie algebras morphism \[\mathrm{ad} : \mathfrak{g} \to \mathrm{Der} (\mathfrak{g}), \quad x \mapsto \mathrm{ad}_{x}(-) \coloneqq [x, -] \ .\] 
\end{example}

\begin{definition}[Maurer--Cartan element]
A \emph{Maurer--Cartan element} of a dg Lie algebra $\mathfrak{g}$ is a degree $-1$ element $\varphi \in \mathfrak{g}_{-1}$ satisfying the Maurer--Cartan equation: \[d \left(\varphi\right) + \tfrac{1}{2}\left[\varphi, \varphi \right] = 0 \ .\] The set of Maurer–Cartan elements in $\mathfrak{g}$ is denoted by $\mathrm{MC}(\mathfrak{g})$~. 
\end{definition}

\begin{proposition}
Let $\varphi \in \mathrm{MC}(\mathfrak{g})$ be a Maurer--Cartan elements in a dg Lie algebra $\mathfrak{g}$.  The map $d^{\psi} \coloneqq d + \mathrm{ad}_{\psi}$ is a differential and defines a \emph{twisted dg Lie algebra} \[\mathfrak{g}^{\psi} \coloneqq \left(\mathfrak{g},  [-,-], d^{\psi} \right) \ .\]  
\end{proposition}

\begin{proof} See \cite[Proposition~1.43]{DSV22}.
\end{proof}

\begin{definition}[Baker--Campbell--Hausdorff formula]
The \emph{Baker--Campbell--Hausdorff} formula is the element in the associative algebra of formal power series on two variables $x$ and $y$ given by $\mathrm{BCH}(x,y)\coloneqq \mathrm{log}(e^xe^y) \ .$
\end{definition}

\begin{proposition}\label{gauge}
The \emph{gauge group} of a complete dg Lie algebra $\mathfrak{g}$ is the group obtained from the set of degree $0$ elements via the Baker--Campbell--Hausdorff formula : \[\left(\mathfrak{g}_0, \mathrm{BCH}, 0\right) \ .\] It acts on the set of Maurer--Cartan elements through the \emph{gauge action} defined by  \[  \lambda \cdot \varphi \coloneqq e^{ \mathrm{ad}_{\lambda}}(\varphi) - \frac{ e^{ \mathrm{ad}_{\lambda}} - \mathrm{id}}{\mathrm{ad}_{\lambda}} (d \lambda)     \ , \] for every Maurer--Cartan element $\varphi \in \mathrm{MC}(\mathfrak{g} )$ and every gauge $\lambda \in \mathfrak{g}_0$.   
\end{proposition}

\begin{proof}
	See \cite[Theorem~1.53]{DSV22}. 
\end{proof}

\begin{remark}
The gauge action's formula comes from the flow of certain vector fields on the set of Maurer--Cartan elements $\mathrm{MC}(\mathfrak{g})$ which is a variety
(an intersection of quadrics) in the finite dimensional case, see \cite[Theorem~1.42]{DSV22}. For every $\lambda \in \mathfrak{g}_0$, one can consider the vector field $\Upsilon_{\lambda}$ characterized for all $\varphi \in \mathrm{MC}(\mathfrak{g} )$ by 
	\begin{equation}\label{eq1}
		\Upsilon_{\lambda} (\varphi) \coloneqq d \lambda + [\varphi, \lambda] \ .
	\end{equation}   
The integration of the flow associated to $\Upsilon_{-\lambda}$, with the initial condition $\gamma_{\lambda}(0)= \varphi $, leads to 
	\[ \gamma_{\lambda}(t) = e^{ t  \mathrm{ad}_{\lambda}}(\varphi ) - \frac{ e^{ t\mathrm{ad}_{\lambda}} - \mathrm{id}}{t\mathrm{ad}_{\lambda}} (td \lambda) \in \mathrm{MC}(\mathfrak{g}) \ .  \] Thus, we have $\lambda \cdot \varphi = \psi $ if and only if the previous integration gives $\psi$ at time $t =1$~.
\end{remark}

\begin{definition}[Gauge equivalence] \label{gauge equivalences}Two Maurer--Cartan elements $\varphi, \psi \in  \mathrm{MC(\mathfrak{g})}$ in a complete dg Lie algebra $\mathfrak{g}$ are \emph{gauge equivalent} if there exists an element $\lambda \in \mathfrak{g}_0$ such that \[ \lambda \cdot \varphi = \psi \ .\] We say that $\lambda$ is a \emph{gauge} between $\varphi$ and $\psi$~. The \emph{moduli space of Maurer–Cartan elements} $\mathcal{MC}(\mathfrak{g})$ is the coset of Maurer–Cartan elements modulo the gauge action. 
\end{definition}

\begin{lemma}\label{technique0}
For every gauge $\lambda \in \mathfrak{g}_0$ in a complete dg Lie algebra $\mathfrak{g}$, the morphism  \[e^{\mathrm{ad}_{\lambda}} : \mathfrak{g}^{\varphi} \rightarrow \ \mathfrak{g}^{\lambda \cdot \varphi}, \quad x \mapsto  e^{\mathrm{ad}_{\lambda}}(x) \ , \] is an isomorphism of dg Lie algebras whose inverse is given by $e^{\mathrm{ad}_{-\lambda}}$. 
\end{lemma}

\begin{proof}
The morphism $e^{\mathrm{ad}_{\lambda}}$ preserves the Lie bracket and is a morphism of Lie algebras as the exponential of a derivation. Furthermore, for all $x \in \mathfrak{g}$~, one has \[ \begin{split}
	d^{\lambda \cdot \varphi} \left(e^{\mathrm{ad}_{\lambda}}(x)\right) &  = \sum_{n \geqslant0} \frac{1}{n!} d\left(\mathrm{ad}_{\lambda}^{n} (x)\right) - \sum_{k, l \geqslant0} \frac{1}{ l !(k +1)!} \left[\mathrm{ad}_{\lambda}^{k}(d \lambda), \mathrm{ad}_{\lambda}^{l}(x)  \right] + e^{\mathrm{ad}_{\lambda}} (\left[ \varphi, x \right] )  . 
\end{split}   \] Since $d$ is a derivation, one can prove by induction on $n \geqslant1$~, that \begin{equation}\label{the induction}
	d \left(\mathrm{ad}_{\lambda}^{n} (x)\right)  = \mathrm{ad}^n_{\lambda}(d x) + \sum_{k + l = n-1} \frac{n!}{l !(k +1)! } \left[\mathrm{ad}_{\lambda}^{k}(d \lambda), \mathrm{ad}_{\lambda}^{l}(x)  \right] 
\end{equation}	for all $x \in \mathfrak{g}$~. This leads to \[ 
d^{\lambda \cdot \varphi} \left(e^{\mathrm{ad}_{\lambda}}(x)\right) 
=  \sum_{k \geqslant0 } \frac{1}{k!} \mathrm{ad}_{\lambda}^{k} (d(x)) + e^{\mathrm{ad}_{\lambda}} \left(\left[ \varphi, x \right] \right)  = e^{\mathrm{ad}_{\lambda}} \left( d^{\varphi}(x) \right) \ ,  \] and $e^{\mathrm{ad}_{\lambda}}$ is an isomorphism of dg Lie algebras whose inverse is given by $e^{\mathrm{ad}_{-\lambda}}$~. 
\end{proof}

\noindent Let $R[\![\hbar]\!]$ be the ring of formal power series in the variable $\hbar$ and let $\mathfrak{g}$ be any dg Lie algebra. We consider the complete dg Lie algebra of formal power series \[\mathfrak{g}[\![\hbar]\!] := \mathfrak{g} \ \widehat{\otimes} \ R[\![\hbar]\!] \ .\] We write any element $x \in \mathfrak{g}[\![\hbar]\!]$ as $x = \sum_{i \geqslant0} x_i \hbar^i \ .$

\begin{definition}[Formal deformations]
Let $\varphi_0$ a Maurer--Cartan element in $\mathfrak{g}$~. A \emph{formal deformation of $\varphi_0$} is a Maurer--Cartan element \[\Phi = \varphi_0 + \varphi_1 \hbar + \varphi_2 \hbar^2 + \cdots  \in \mathrm{MC}(\mathfrak{g}[\![\hbar]\!])\] which reduces to $\varphi_0$ modulo the maximal ideal $(\hbar)$~. We denote by $\mathrm{Def}_{\varphi_0}(R[\![\hbar]\!])$ the set of all formal deformations of $\varphi_0$~. Two formal deformations $\Phi$ and $\Psi$ are \emph{gauge equivalent}, if there exists a gauge $\lambda \in \mathfrak{g}_0 \otimes (\hbar)$ such that \[\Psi = \lambda \cdot \Phi \coloneqq e^{ \mathrm{ad}_{\lambda}}(\Phi) - \frac{e^{  \mathrm{ad}_{\lambda}} - \mathrm{id}}{\mathrm{ad}_{\lambda}} (d \lambda)  \ . \] 
\end{definition}

\begin{example}
Every $\varphi_0 \in \mathrm{MC}(\mathfrak{g})$ admits a trivial formal deformation $\varphi_0 + 0 \in \mathrm{MC}(\mathfrak{g}[\![\hbar]\!])$~. 
\end{example}

\subsection{Kaledin classes for formal deformations}\label{1.2} Let $\varphi_0$ be a Maurer--Cartan element in a dg Lie algebra $\mathfrak{g}$ and let $\Phi$ be a formal deformation of $\varphi_0$. Is $\Phi$ gauge equivalent to the trivial deformation? To answer this question, we generalize the construction of \cite[Section~7]{Lun07}, which addresses the particular case $\varphi_{0} = 0$~. 

\begin{definition}[Differentiation operator]
	For any graded $R$-module $V$~, the \emph{differentiation operator} $\partial_{\hbar} : V[\![\hbar]\!] \longrightarrow V[\![\hbar]\!]$ is defined by \[\sum_{i \geqslant0} x_i \hbar^i \longmapsto \sum_{i \geqslant1} i x_i \hbar^{i-1} \ .\]
\end{definition}

\begin{remark} Let $\mathfrak{g}$ be a dg Lie algebra,
\begin{itemize}
	\item[$\centerdot$] the operator $\partial_{\hbar} : \mathfrak{g}[\![\hbar]\!] \to \mathfrak{g}[\![\hbar]\!]$ is a derivation of degree zero, 
	\item[$\centerdot$] the underlying differential $d$ of $\mathfrak{g}$~, viewed in  $\mathrm{Der}(\mathfrak{g})[\![\hbar]\!]$~, satisfies $\partial_{\hbar}(d) = 0$~. 
\end{itemize}
\end{remark}

\begin{proposition}\label{c'est un cycle}
Any formal deformation $\Phi \in \mathrm{Def}_{\varphi_0}(R[\![\hbar]\!])$ satisfies \[\partial_{\hbar} \Phi \in Z_{-1}\left(\mathfrak{g}[\![\hbar]\!]^{\Phi}\right)\ ,\] i.e. the image of $\Phi$ under the operator $\partial_{\hbar}$ is a cycle in the twisted dg Lie algebra $\mathfrak{g}[\![\hbar]\!]^{\Phi}$~.
\end{proposition}

\begin{proof}
	Let us decompose $\Phi$ as $\varphi_0 + \tilde{\Phi}$ where $\tilde{\Phi} \in \mathrm{MC}(\mathfrak{g}[\![\hbar]\!]^{\varphi_0})$. The result now follows from the following computation: \[	\begin{split}                                  
		d^{\Phi} \left(\partial_{\hbar} \Phi \right)
		 & = d \left(\partial_{\hbar} \tilde{\Phi} \right) + \left[\varphi_0 + \tilde{\Phi}, \partial_{\hbar} \tilde{\Phi} \right] \\
		& = \partial_{\hbar} \left( d \tilde{\Phi} \right) + \left[\varphi_0 , \partial_{\hbar} \tilde{\Phi} \right] + \left[\tilde{\Phi} , \partial_{\hbar} \tilde{\Phi} \right] \\
		& =  \partial_{\hbar} \left( d \tilde{\Phi} \right) + \partial_{\hbar} \left[\varphi_0 ,  \tilde{\Phi} \right] +  \tfrac{1}{2} \partial_{\hbar} \left[\tilde{\Phi} , \tilde{\Phi} \right] \\
		& =   \partial_{\hbar} \left( d \tilde{\Phi} + \left[\varphi_0 ,  \tilde{\Phi} \right] +  \tfrac{1}{2} \left[\tilde{\Phi} , \tilde{\Phi} \right] \right)\\
		& = 0 \ . \qedhere	\end{split}   \]
\end{proof}	

\begin{definition}[Kaledin class]
The \emph{Kaledin class} of a formal deformation $\Phi \in \mathrm{Def}_{\varphi_0}(R[\![\hbar]\!]) $ is the homology class \[K_{\Phi} \coloneqq \left[\partial_{\hbar} \Phi \right] \in H_{-1}\left(\mathfrak{g}[\![\hbar]\!]^{\Phi}\right) \ .\] 
\end{definition}

\begin{example}
The Kaledin class of the trivial deformation is trivial, i.e. $K_{\varphi_0} = 0$~. 
\end{example}

\begin{theorem}\label{Kaledin1}
Let $\Phi \in \mathrm{Def}_{\varphi_0}(R[\![\hbar]\!]) $ be a formal deformation. Its Kaledin class $K_{\Phi}$ is zero if and only if $\Phi$ is gauge equivalent to the trivial deformation $\varphi_0$~, i.e. there exists a gauge $\lambda \in \mathfrak{g}_0 \otimes (\hbar)$ such that \[\lambda \cdot \Phi = \varphi_0 \in \mathfrak{g}[\![\hbar]\!] \ .\] 
\end{theorem}

\begin{remark}
The heuristic behind this obstruction class is the following one. The Kaledin class $K_{\Phi}$ is zero if and only if $\partial_{\hbar} \Phi$ is a boundary, i.e. there exits $\lambda \in \mathfrak{g}_0 \otimes (\hbar)$ such that \begin{equation*}
	\partial_{\hbar} \Phi =  d^{\Phi}(\lambda) = d \lambda + [\Phi, \lambda]  \ .
\end{equation*}
One recognizes the vector field (\ref{eq1}) characterizing the gauge action. 
\end{remark}

\noindent The proof of Theorem \ref{Kaledin1} requires the following lemma on the invariance of the Kaledin class under the gauge group action.

\begin{lemma}\label{technique1}
Every $\lambda \in \mathfrak{g}_0 \otimes (\hbar)$ induces an isomorphism of dg Lie algebras,  \[e^{\mathrm{ad}_{\lambda}} : \mathfrak{g}[\![\hbar]\!]^{\Phi} \longrightarrow \ \mathfrak{g}[\![\hbar]\!]^{\lambda \cdot \Phi}, \quad x \mapsto  e^{\mathrm{ad}_{\lambda}}(x) \ . \]
Furthermore, the classes $K_{\lambda \cdot \Phi}$ and $e^{\mathrm{ad}_{\lambda}}\left(K_{\Phi}\right)$ are equal in \[H_{-1}\left(\mathfrak{g}[\![\hbar]\!]^{\lambda \cdot \Phi}\right) \ ,\] where we still denote by $e^{\mathrm{ad}_{\lambda}}$ the induced isomorphism in homology. 
\end{lemma}

\begin{proof}
By the proof of Lemma \ref{technique0}, the exponential \[e^{\mathrm{ad}_{\lambda}} : \mathfrak{g}[\![\hbar]\!]^{\Phi} \rightarrow \ \mathfrak{g}[\![\hbar]\!]^{\lambda \cdot \Phi}\] is an isomorphism of dg Lie algebras by the proof of which holds \emph{mutatis mutandis} by $\hbar$-linearity. Let us prove that the cycles $e^{\mathrm{ad}_{\lambda}}(\partial_{\hbar} \Phi)$ and $\partial_{\hbar} (\lambda \cdot \Phi)$ are homologous. We have \[ 
		\partial_{\hbar} (\lambda \cdot \Phi) = \partial_{\hbar} \left(e^{ \mathrm{ad}_{\lambda}}(\Phi) - \frac{e^{  \mathrm{ad}_{\lambda}} - \mathrm{id}}{\mathrm{ad}_{\lambda}} (d \lambda)   \right) \ .  \] The operator $\partial_{\hbar}$ is a derivation and it satisfies an equation similar to Equation (\ref{the induction}). In other words, for all $n \geqslant1$ and all $x \in \mathfrak{g}[\![\hbar]\!]$~, one can prove that \begin{equation*}
			\partial_{\hbar} \left(\mathrm{ad}_{\lambda}^{n} (x)\right)  = \mathrm{ad}^n_{\lambda}\left(\partial_{\hbar} x\right) + \sum_{k + l = n-1} \frac{n!}{l !(k +1)! } \left[\mathrm{ad}_{\lambda}^{k}(\partial_{\hbar} \lambda), \mathrm{ad}_{\lambda}^{l}(x)  \right] .
		\end{equation*}	 Using twice this identity, one gets \[\begin{split}
		\partial_{\hbar} \left(e^{ \mathrm{ad}_{\lambda}}(\Phi) \right) - e^{\mathrm{ad}_{\lambda}}(\partial_{\hbar} \Phi) &  = - \left[ e^{\mathrm{ad}_{\lambda}}(\Phi),  \frac{e^{  \mathrm{ad}_{\lambda}} - \mathrm{id}}{\mathrm{ad}_{\lambda}} (\partial_{\hbar} \lambda)   \right] \ , \  \mathrm{and}
		\end{split}\]	\[ \begin{split}
	\partial_{\hbar} \left(\frac{e^{  \mathrm{ad}_{\lambda}} - \mathrm{id}}{\mathrm{ad}_{\lambda}} (d \lambda)   \right) & = \sum_{n \geqslant0} \frac{1}{(n+1)!} \mathrm{ad}^n_{\lambda}(\partial_{\hbar} (d \lambda) ) + \sum_{\substack{n \geqslant 0 \\ k + l = n}} \frac{1}{n (l !)(k +1)! } \left[\mathrm{ad}_{\lambda}^{k}(\partial_{\hbar} \lambda), \mathrm{ad}_{\lambda}^{l}(d \lambda)  \right] 
\end{split}	 \ .\] In the latter equation, the element $\mathrm{ad}^n_{\lambda}(\partial_{\hbar} (d \lambda) )= \mathrm{ad}^n_{\lambda}(d (\partial_{\hbar} \lambda) ) $ can be decomposed using the equation (\ref{the induction}). This leads to 
	\[\begin{split}
		\partial_{\hbar} \left(\frac{e^{  \mathrm{ad}_{\lambda}} - \mathrm{id}}{\mathrm{ad}_{\lambda}} (d \lambda)   \right)  &  = d \left(\frac{e^{  \mathrm{ad}_{\lambda}} - \mathrm{id}}{\mathrm{ad}_{\lambda}} (\partial_{\hbar} \lambda) \right)  - \left[  \frac{e^{  \mathrm{ad}_{\lambda}} - \mathrm{id}}{\mathrm{ad}_{\lambda}} (d \lambda)  ,  \frac{e^{  \mathrm{ad}_{\lambda}} - \mathrm{id}}{\mathrm{ad}_{\lambda}} (\partial_{\hbar} \lambda)   \right] \ .
	\end{split}\] In conclusion, the following equality holds \[\begin{split}
\partial_{\hbar} (\lambda \cdot \Phi) - e^{\mathrm{ad}_{\lambda}}(\partial_{\hbar} \Phi) & 			
= d^{\lambda \cdot \Phi}\left( -\frac{e^{  \mathrm{ad}_{\lambda}} - \mathrm{id}}{\mathrm{ad}_{\lambda}} (\partial_{\hbar} \lambda) \right) \ ,
\end{split}  \] and the classes $K_{\lambda \cdot \Phi}$ and $e^{\mathrm{ad}_{\lambda}}\left(K_{\Phi}\right)$ are equal in homology. \end{proof}

\noindent All the above definitions and properties can be adapted over the truncated ring $R[\![\hbar]\!]/(\hbar^{n})$~, for all $n \geqslant1$~. This leads to the following definition of the $n^{\text{th}}$-truncated Kaledin class. 

\begin{definition}[$n^{\text{th}}$-truncated Kaledin class]
For all $n \geqslant 1$~, the \emph{$n^{\text{th}}$-truncated Kaledin class} of the formal deformation $\Phi \in \mathrm{Def}_{\varphi_0}(R[\![\hbar]\!])$ is the homology class \[K_{\Phi}^n \coloneqq  \left[\varphi_{1} + 2 \varphi_2 \hbar + \cdots + n \varphi_n \hbar^{n-1}\right]   \in H_{-1}\left(\left(\mathfrak{g}[\![\hbar]\!]/ (\hbar^{n}) \right)^{\overline{\Phi}^n} \right) \ , \] where $\overline{\Phi}^n$ denotes the projection in $\mathfrak{g}[\![\hbar]\!]/ (\hbar^{n}) \ .$
\end{definition} 

\begin{remark}\label{toutes nulles0}
Let $\Phi$ be a formal deformation. For all $n \geqslant 1$~, the projection \[\mathfrak{g}[\![\hbar]\!]^{\Phi} 	\twoheadrightarrow \left(\mathfrak{g}[\![\hbar]\!]/ (\hbar^{n}) \right)^{\overline{\Phi}^n}\] is a morphism of dg Lie algebras. In particular, if the Kaledin class $K_{\Phi}$ is zero, then all the truncated Kaledin classes $K^n_{\Phi}$ are zero. Similarly, if the $n^{\text{th}}$-truncated Kaledin class $K^n_{\Phi}$ is zero then, for all $m \leqslant n$~, the class  $K^m_{\Phi}$ is zero.
\end{remark}

\begin{lemma}\label{technique12} 
Every $\lambda \in \mathfrak{g}_0 \otimes (\hbar)$ induces an isomorphism of dg Lie algebras,  \[e^{\mathrm{ad}_{\lambda}} : \left(\mathfrak{g}[\![\hbar]\!] /(\hbar^{n}) \right)^{\overline{\Phi}^n} \longrightarrow \ \left(\mathfrak{g}[\![\hbar]\!] /(\hbar^{n}) \right)^{\overline{\lambda \cdot \Phi}^n} \ . \] The classes $K_{\lambda \cdot \Phi}^n$ and $e^{\mathrm{ad}_{\lambda}}\left(K_{\Phi}^n\right)$ are equal in \[H_{-1}\left(\left(\mathfrak{g}[\![\hbar]\!]/(\hbar^{n})\right)^{\overline{\lambda \cdot \Phi}^n}\right) \ ,\] where we still denote by $e^{\mathrm{ad}_{\lambda}}$ the induced isomorphism on homology. 
\end{lemma}

\begin{proof} 
The proof is similar to the one of Lemma \ref{technique1}.
\end{proof}

\begin{theorem}\label{truncated1}
Let $\Phi \in \mathrm{Def}_{\varphi_0}(R[\![\hbar]\!])$ be a formal deformation. For all $n \geqslant 1$, the following propositions are equivalent.
\begin{enumerate}
	\item The $n^{\text{th}}$-truncated Kaledin class $K_{\Phi}^n$ is zero.
	\item There exists $\lambda \in \mathfrak{g}_0 \otimes (\hbar)$ such that $\lambda \cdot \Phi \equiv \varphi_0  \pmod{\hbar^{n+1}} \ .$
\end{enumerate}
\end{theorem}

\begin{proof} If Point (2) holds, Lemma \ref{technique12} implies that \[K_{\Phi}^n = e^{\mathrm{ad}_{- \lambda}} \left(K_{\lambda \cdot \Phi}^n\right) = e^{\mathrm{ad}_{- \lambda}} \left(K_{\varphi_0}^n\right) = 0 \ . \] Let us prove the converse result by induction on $n$. If the first Kaledin class \[K^1_{\Phi} = \left[ \varphi_1 \right] \in H_{-1}\left(\mathfrak{g}^{\varphi_0}\right) \] is zero, there exists $\phi_0 \in \mathfrak{g}_0$ such that $d^{\varphi_0} (\phi_0)=\varphi_1 \ .$ The element $\lambda \coloneqq \phi_0 h$ satisfies
		\[
		\lambda \cdot \Phi  \equiv \varphi_0 + \varphi_1 h - d^{\varphi_0} (\phi_0) h   \equiv \varphi_0   \pmod{\hbar^{2}} \ .\] Suppose that $n > 1$ and that the result holds for $n-1$. If the class $K_{\Phi}^n$ is zero, so does $K_{\Phi}^{n-1}$. By the induction hypothesis, there exists $\nu \in  \mathfrak{g}_0 \otimes (\hbar)$ such that \[\nu \cdot \Phi \equiv \varphi_0  \pmod{\hbar^{n}} \ .\] Considering $\Psi \coloneqq \nu \cdot \Phi \ ,$ Lemma \ref{technique12} implies that $e^{\mathrm{ad}_{\nu}}(K_{\Phi}^n) = K_{\Psi}^n = 0$~. By construction \[\Psi \equiv \varphi_0 + \psi_n \hbar^n \pmod{\hbar^{n+1}} \quad \mbox{and} \quad K_{\Psi}^{n} = \left[n \psi_n \hbar^{n-1}\right] = 0 \ .\]  Thus, there exists $\phi \coloneqq \phi_0 + \phi_1 \hbar + \cdots + \phi_{n-1}\hbar^{n-1} \in \mathfrak{g}[\![\hbar]\!]$ of degree zero such that \[d^{\varphi_0}(\phi) \equiv n \psi_n \hbar^{n-1} \pmod{\hbar^{n}} \ .\] Looking at the coefficient of $\hbar^{n-1}$ on both sides, the following equality holds \[n \psi_n = d(\phi_{n-1}) + [\varphi_0, \phi_{n-1}] \ . \] The element $\upsilon \coloneqq \frac{1}{n}\phi_{n-1} \hbar^{n} \in \mathfrak{g}_0 \otimes (\hbar)$ satisfies \[\upsilon \cdot \Psi  \equiv \varphi_0 \pmod{\hbar^{n+1}} \ .\] Finally, the gauge $\lambda \coloneqq \mathrm{BCH}(\upsilon,\nu)$ is such that $\lambda \cdot \Phi = \upsilon \cdot (\nu \cdot \Phi)  \equiv \varphi_0 \pmod{\hbar^{n +1}} \ .$
\end{proof}

\begin{proof}[Proof of Theorem \ref{Kaledin1}] If $\Phi$ is gauge equivalent to $\varphi_0$, Lemma \ref{technique1}  implies that \[K_{\Phi} = K_{\varphi_0} = 0 \ . \] Conversely, suppose that $K_{\Phi} = 0$~. If $\Phi = \varphi_0$~, the result is immediate. Otherwise, let $n_1 \geqslant1$ be the smallest integer such that $\varphi_{n_1}$ is not equal to zero. Since $K_{\Phi}$ is zero, so does the class $K_{\Phi}^{n_1}$. By Theorem \ref{truncated1}, there exists  $\lambda^{1}  \in \mathfrak{g}_0 \otimes (\hbar)$ such that \[\lambda^{1} \cdot \Phi \equiv \varphi_0  \pmod{\hbar^{n_1 +1}} \ .\] The proof of Theorem \ref{truncated1} actually shows that one can find a gauge of the form  $ \lambda^{1} =  x_{n_1} \hbar^{n_1}$~. Let $\Phi^2 \coloneqq \lambda^{1} \cdot \Phi$ and let $n_2$ be the smallest integer ${n_2} > {n_1} $ such that $(\Phi^2)_{n_2}$ is not zero. By Lemma \ref{technique12}, we have \[K_{\Phi^2}^{n_2} = e^{\mathrm{ad}_{\lambda^{1}}}(K^{n_2}_{\Phi}) = 0 \ .\] Again, there exists a gauge of the form $\lambda^{2} = x_{n_2} \hbar^{n_2}$ such that \[\lambda^{2} \cdot  \Phi^2 \equiv \varphi_0 \pmod{\hbar^{n_2+1} }\ .\] By repeating the procedure, we obtain a increasing sequence \[n_1 < n_2 < n_3 < \cdots\] of integers and a sequence of gauges of the form $\lambda^{i} = x_{n_i} \hbar^{n_i}$~, for all $i$~, such that \[\lambda^{i} \cdot  \Phi^{i} \equiv \varphi_0 \pmod{\hbar^{n_i+1}}\] where $\Phi^i \coloneqq \lambda^{{i-1}} \cdot \Phi^{i-1}$ for all $n \geqslant 2$. It follows from the construction that \[\lambda = \mathrm{BCH}(\cdots \mathrm{BCH}(\lambda^{3}, \mathrm{BCH}(\lambda^{2}, \lambda^{1})) \cdots ) \ , \] is well defined, since each gauge $\lambda^{i}$ only has components in $\hbar^{n_i}$~, and is a desired gauge.
\end{proof}

\begin{proposition}\label{toutes nulles}
Let $\Phi \in \mathrm{Def}_{\varphi_0}(R[\![\hbar]\!])$ be a formal deformation. The Kaledin class $K_{\Phi}$ is zero if and only if, for all $n \geqslant 1$, the truncation $K^n_{\Phi}$ is zero. 
\end{proposition}

\begin{proof}  If Kaledin class $K_{\Phi}$ is zero then all the truncations $K^n_{\Phi}$ are zero by Remark \ref{toutes nulles0}. Conversely, suppose that for all $n \geqslant 1$~, the truncation $K^n_{\Phi}$ is zero. If $\Phi = \varphi_0$~, the result is immediate. Otherwise, let $n_1$ be the smallest integer such that $\varphi_{n_1} \neq 0$~. As in the proof of Theorem \ref{truncated1}, we can construct a sequence of integers $n_i$ and a sequence of gauges $\lambda^{i} = x_{n_i} \hbar^{n_i}$ for all $i$~, such that \[\lambda^{i} \cdot  \Phi^{i} \equiv \varphi_0 \pmod{\hbar^{n_i+1}}\] where $\Phi^i \coloneqq \lambda^{{i-1}} \cdot \Phi^{i-1}$~for $i \geqslant 2$ and $\Phi^1 \coloneqq \Phi$. The infinite composition \[\lambda = \mathrm{BCH}(\cdots \mathrm{BCH}(\lambda^{3}, \mathrm{BCH}(\lambda^{2}, \lambda^{1})) \cdots ) \ , \] is then such that $\lambda \cdot \Phi = \varphi_0 \in \mathfrak{g}[\![\hbar]\!]$ and $K_{\Phi} = 0$ by Theorem \ref{Kaledin1}.  
\end{proof}

\begin{theorem}[Descent] \label{descent1.0}
Let $S$ be a faithfully flat commutative $R$-algebra. Let $\varphi_0 \in \mathrm{MC}(\mathfrak{g})$ be a Maurer--Cartan element. A formal deformation $\Phi$ of $\varphi_0$ is gauge equivalent to $\varphi_0$ if and only if $\Phi \otimes 1$ is gauge equivalent to the trivial deformation $\varphi_0 \otimes 1$ in $\mathfrak{g}[\![\hbar]\!] \otimes S $. 
\end{theorem}

\begin{proof} If there exists $\lambda \in \mathfrak{g}_0 \otimes (\hbar)$ satisfying $\lambda \cdot \Phi = \varphi_0$~, then \[(\lambda \otimes 1) \cdot (\Phi \otimes 1) = \varphi_0 \otimes 1 \ .\] Conversely, suppose that $\Phi \otimes 1$ is gauge equivalent to the trivial deformation $\varphi_0 \otimes 1$~. For all $n \geqslant 1$~, Theorem \ref{truncated1} implies that the Kaledin classes $K^n_{\Phi \otimes 1}$ are zero. The isomorphism \[(\mathfrak{g} \otimes S)[\![\hbar]\!]/(\hbar^{n}) \cong  \mathfrak{g}[\![\hbar]\!]/(\hbar^{n}) \otimes_{R[\![\hbar]\!]/(\hbar^{n})} S[\![\hbar]\!]/(\hbar^{n})\] leads to an isomorphism \[H_{-1}\left(\left((\mathfrak{g} \otimes S) [\![\hbar]\!]/(\hbar^{n})\right)^{\overline{\Phi}^n \otimes 1} \right) \cong  H_{-1}\left(\left(\mathfrak{g}[\![\hbar]\!] /(\hbar^{n}) \right) ^{\overline{\Phi}^n}\right) \otimes_{R[\![\hbar]\!]} S[\![\hbar]\!] / (\hbar^{n})  \ ,  \] 
since $R[\![\hbar]\!]/(\hbar^{n}) \to S[\![\hbar]\!]/(\hbar^{n})$ is flat. By faithful flatness, the class $K^n_{\Phi \otimes 1}$ is zero on the left hand-side if and only if $K_{\Phi}^n \otimes 1$ is zero on the right hand-side. Thus, the Kaledin class \[K_{\Phi}^n \in H_{-1}\left(\left(\mathfrak{g}[\![\hbar]\!] /(\hbar^{n})\right)^{\overline{\Phi}^n}\right) \] is zero for all $n$~. Proposition \ref{toutes nulles} and Theorem \ref{truncated1} allow us to conclude.
\end{proof}

\noindent The following result is well-known, see e.g. \cite[Theorem~1.66]{DSV22}, we give it a new proof using the Kaledin classes. 

\begin{theorem}[Rigidity]\label{intrinsec0} Let $\mathfrak{g}$ be a dg Lie algebra and let $\varphi_0 \in \mathrm{MC}(\mathfrak{g})$~. If $H_{-1}(\mathfrak{g}^{\varphi_0})=0$~, then $\varphi_0$ is rigid, i.e. every formal deformation $\Phi \in \mathrm{Def}_{\varphi_0}(R[\![\hbar]\!])$ is gauge equivalent to $\varphi_0$~. 
\end{theorem}

\begin{proof}
We prove by induction on $n$ that all the truncated classes $K_{\Phi}^n$ are zero. The class\[K^1_{\Phi} \in H_{-1}\left(\mathfrak{g}^{\varphi_0}\right) \] is necessarily zero. Suppose that $K^{n-1}_{\Phi} = 0$ for $n > 1$. There exists $\nu \in \mathfrak{g}_0 \otimes (\hbar)$ such that \[\nu \cdot \Phi \equiv \varphi_0  \pmod{\hbar^{n}} \ ,\] by Theorem \ref{truncated1}. The element $\Psi \coloneqq \nu \cdot \Phi$ satisfies $d^{\overline{\Psi}^n} = d^{\varphi_0}$~. By assumption, the group \[H_{-1}\left(\left(\mathfrak{g}[\![\hbar]\!]/(\hbar^{n})\right)^{\overline{\Psi}^n}\right) \cong H_{-1}\left(\mathfrak{g}^{\varphi_0}\right)[\![\hbar]\!]/(\hbar^{n})  \] is trivial. Thus the class $K^n_{\Psi}$ is zero and by Lemma \ref{technique12} \[K^n_{\Phi} = e^{\mathrm{ad}_{-\nu}} (K^n_{\Psi}) = 0 \ ~.\] Using Proposition \ref{toutes nulles} and Theorem \ref{Kaledin1}, this implies that $\Phi$ is equivalent to $\varphi_0$~. 	\end{proof}  

\subsection{The case of weight-graded dg Lie algebras}\label{1.3}

This section detects whether a given Maurer--Cartan element is gauge equivalent to its first component in a weight-graded dg Lie algebra. In this setting, the Kaledin classes of Section \ref{1.2} and \cite[Section~7]{Lun07} do not directly apply, see Remark \ref{refine}. Hence, we refine the previous constructions to deal with any $\delta$-weight-graded dg Lie algebra, for $\delta \geqslant 1$ an integer. 

\begin{definition}[$\delta$-weight-graded dg Lie algebra]\label{delta wg}
	A complete dg Lie algebra $\left(\mathfrak{g}, [-,-] , d  \right)$ is a \emph{$\delta$-weight-graded dg Lie algebra} if it has an additional weight grading such that \[\mathfrak{g} \cong \prod_{k \geqslant 1} \mathfrak{g}^{(k)}, \quad d \left(\mathfrak{g}^{(k)} \right) \subset \mathfrak{g}^{(k + \delta)}, \quad \left[ \mathfrak{g}^{(k)},  \mathfrak{g}^{(l)} \right] \subset  \mathfrak{g}^{(k +l)}  \ .  \]   Every element $\varphi$ in $\mathfrak{g}$ decomposes as $\varphi = \varphi^{(1)} + \varphi^{(2)} +  \varphi^{(3)} + \cdots $~, where $\varphi^{(k)} \in \mathfrak{g}^{(k)}$~, for all $k \geqslant 1$~. It is complete for the descending filtration defined by \[\mathcal{F}^n \mathfrak{g} \coloneqq  \prod_{k \geqslant n} \mathfrak{g}^{(k)} \ .  \]
\end{definition}

\begin{definition}[Prismatic decomposition] Let $\mathfrak{g}$ be a $\delta$-weight-graded dg Lie algebra. The \emph{prismatic decomposition} is the morphism of graded vector spaces \[ \mathfrak{D} : \mathfrak{g}_{-1} \oplus \mathfrak{g}_{0}   \to  \mathfrak{g}[\![\hbar]\!] \ , \] defined for all $\varphi \in \mathfrak{g}_{-1}$ and all $\lambda \in \mathfrak{g}_{0}$~, by \[\begin{split}
	& \mathfrak{D}(\varphi) \coloneqq  \varphi^{(1)} + \cdots + \varphi^{(\delta)} + \varphi^{(\delta + 1 )} \hbar + \varphi^{(\delta + 2)} \hbar^2 + \cdots \\
&	\mathfrak{D} (\lambda) \coloneqq   \lambda^{(1)} \hbar + \lambda^{(2)} \hbar^2 + \lambda^{(3)} \hbar^3 + \cdots \ .
	\end{split}
 \]
\end{definition}

\begin{lemma}\label{decomposition prismatic}
Let $\mathfrak{g}$ be a $\delta$-weight-graded dg Lie algebra, let $\varphi \in \mathrm{MC}\left(\mathcal{F}^{\delta}\mathfrak{g}\right)$ and let $\lambda \in \mathfrak{g}_0$.
\begin{enumerate}
	\item The prismatic decomposition $\mathfrak{D}(\varphi)$ is a formal deformation of $\varphi^{(\delta)}$, i.e. \[\mathfrak{D}(\varphi) = \varphi^{(\delta)} + \varphi^{(\delta + 1)} \hbar + \varphi^{(\delta + 2)} \hbar^2 + \cdots \in \mathrm{Def}_{\varphi^{(\delta)}}(R[\![\hbar]\!]) \ .\]	
	\item The prismatic decomposition of $\lambda \cdot \varphi$ is given by the action of $\mathfrak{D}(\lambda)$ on $\mathfrak{D}(\varphi)$ in $\mathfrak{g}[\![\hbar]\!]$: \[\mathfrak{D}(\lambda \cdot \varphi ) = \mathfrak{D}(\lambda) \cdot \mathfrak{D}(\varphi) \ .\] 
\end{enumerate}	
\end{lemma}

\begin{proof} 
	Since $\varphi \in \mathrm{MC}(\mathfrak{g})$ is a Maurer--Cartan element, the component $\varphi^{(\delta)}$ is also a Maurer--Cartan element. Point (1) follows from the fact that the coefficient of $\hbar^n$ in the equation \[d \left(\mathfrak{D}(\varphi) \right) + \tfrac{1}{2} \left[\mathfrak{D}(\varphi), \mathfrak{D}(\varphi)\right]\] is equal to the component of weight $n+2\delta$ of $d(\varphi) + \tfrac{1}{2} \left[\varphi, \varphi \right] $. For Point (2), one can prove that the coefficient of $\hbar^n$ in \[\mathfrak{D}(\lambda) \cdot \mathfrak{D}(\varphi) \] coincides with the component of weight $n + \delta$ of $\lambda \cdot \varphi$. 
\end{proof}

\begin{definition}[Kaledin classes]
Let $\mathfrak{g}$ be a $\delta$-weight-graded dg Lie algebra and let \[\varphi \in \mathrm{MC}\left(\mathcal{F}^{\delta}\mathfrak{g}\right)\] be a Maurer--Cartan element. 
\begin{itemize}
	\item[$\centerdot$] The \emph{Kaledin class} $K_{\varphi}$ is the Kaledin class $K_{\mathfrak{D}(\varphi) }$ of its prismatic decomposition.
	\item[$\centerdot$]
	Its \emph{$n^{\text{th}}$-truncated Kaledin class} $K_{\varphi}^n$ is the $n^{\text{th}}$-truncated Kaledin class $K_{\mathfrak{D}(\varphi) }^n$ of its prismatic decomposition.
\end{itemize}

\end{definition}

\begin{theorem}\label{A}
Let $R$ be a $\mathbb{Q}$-algebra and let $\delta \geqslant 1$ be an integer. Let $\varphi \in \mathrm{MC}\left(\mathcal{F}^{\delta}\mathfrak{g}\right)$ be a Maurer--Cartan element in a $\delta$-weight-graded dg Lie algebra $\mathfrak{g}$ over $R$~. 
	\begin{enumerate}
		\item The Kaledin class $K_{\varphi}$ is zero if and only if there exists $\lambda \in \mathfrak{g}_0$ such that $\lambda \cdot \varphi = \varphi^{(\delta)}$~.
		\item Let $\psi \in \mathrm{MC}(\mathfrak{g})$ be a Maurer--Cartan element concentrated in weight $\delta$. If \[H_{-1}\left(\mathfrak{g}^{\psi}\right) = 0 \ ,\]  then $\psi$ is \emph{rigid}, i.e. every Maurer--Cartan element $\varphi \in \mathrm{MC}\left(\mathcal{F}^{\delta}\mathfrak{g}\right)$ such that $\varphi^{(\delta)} = \psi$~, is gauge equivalent to $\psi$~. \smallskip
		\item Let $S$ be a faithfully flat commutative $R$-algebra. The element $\varphi$ is gauge equivalent to $\varphi^{(\delta)}$ in $\mathfrak{g}$ if and only if $\varphi \otimes 1$ is gauge equivalent to $\varphi^{(\delta)} \otimes 1$ in $ \mathfrak{g} \otimes S$~. 
	\end{enumerate}
\end{theorem}

\begin{remark}\label{refine}
If the Kaledin class $K_{\varphi}$ is zero, then there exists $\lambda \in \mathfrak{g}_0 \otimes  (\hbar) \ ,$ such that $\lambda \cdot \mathcal{D}(\varphi) = \varphi^{(\delta)}$ by Theorem \ref{Kaledin1}. This is not sufficient \emph{a priori} to get a gauge $\nu \in \mathfrak{g}_0$ such that $\nu \cdot \varphi = \varphi^{(\delta)}$. For all $k \geqslant 1$~, each $\lambda_i$ may have a non-trivial component of weight $k$ and nothing asserts that the sum \[\lambda_1^{(k)} + \lambda_2^{(k)} + \lambda_3^{(k)} + \cdots\] is well-defined. Therefore, we need to refine the proof of the Theorem \ref{Kaledin1}. 
\end{remark}

\begin{proposition}\label{passage}
Let $\varphi \in \mathrm{MC}\left(\mathcal{F}^{\delta}\mathfrak{g}\right)$ be a Maurer--Cartan element in a $\delta$-weight-graded dg Lie algebra $\mathfrak{g}$ over $R$~. For all $n \geqslant 1$~, the following propositions are equivalent. 
\begin{enumerate}
	\item  The $n^{\text{th}}$-truncated Kaledin class $K_{\varphi}^n$ is zero.
	\item There exists $\lambda \in \mathfrak{g}_0 $ such that $(\lambda \cdot \varphi)^{(k)} = 0$~, for all $\ \delta + 1 \leqslant  k \leqslant  \delta + n $ .
\end{enumerate}
\end{proposition}

\begin{proof} If Point (2) holds, Lemma \ref{decomposition prismatic} implies that \[\mathfrak{D}(\lambda) \cdot \mathfrak{D}(\varphi)  \equiv \varphi^{(\delta)} \pmod{\hbar^{n+1}} \ , \]  and $K_{\varphi}^n = 0$ by Theorem \ref{truncated1}. Let us prove the converse result by induction on $n \geqslant 1$~. If  \[K^1_{\varphi} = \left[ \varphi^{(\delta +1 )} \right] \in H_{-1}\left(\mathfrak{g}^{\varphi^{(\delta)}}\right)  \] is zero, there exists $\phi \in \mathfrak{g}_0$ such that \[d^{\varphi^{(\delta)}} (\phi) = \varphi^{(\delta + 1)} \ .\] Then, the element $\lambda \coloneqq \phi^{(1)}  \in \mathfrak{g}_0$ is such that
	$(\lambda \cdot \varphi)^{(\delta +1)} = 0 $~. Suppose that $n > 1$ and that the result holds for $n-1$~. If the class $K_{\varphi}^n$ is zero, so does the class $K_{\varphi}^{n-1}$. By the induction hypothesis, there exists $\nu \in \mathfrak{g}_0 $ such that \[ (\nu \cdot \varphi)^{(k)} = 0  ~, \quad \mbox{for all } \  \delta +1  \leqslant  k \leqslant \delta +  n-1 ~. \] Let us denote $\psi \coloneqq \nu \cdot \varphi$. We have  \[\mathfrak{D}\left( \psi \right) \equiv \varphi^{(\delta)} + \psi^{(\delta + n )} \hbar^n \pmod{\hbar^{n+1}} \quad \mbox{and} \quad K_{\psi }^{n} = \left[n \psi^{(\delta + n)} \hbar^{n-1}\right] \ .\] Since $\mathfrak{D}\left( \psi \right) = \mathfrak{D}\left( \nu \right) \cdot \mathfrak{D}\left( \varphi \right) $ by Lemma \ref{decomposition prismatic}, we have \[K_{\psi}^n = e^{\mathrm{ad}_{\mathfrak{D}\left( \nu \right) }}(K_{\varphi }^n) = 0 \ ~, \] by Lemma \ref{technique12}. Thus, there exists $\phi \coloneqq \phi_0 + \phi_1 h + \dots + \phi_{n-1}\hbar^{n-1} \in \mathfrak{g}[\![\hbar]\!]/ (\hbar^{n})$ such that \[d^{\varphi^{(\delta)}}(\phi) \equiv n \psi^{(\delta + n)} \hbar^{n-1} \pmod{\hbar^{n}} \ .\] Looking at the coefficient of $\hbar^{n-1}$ and the component of weight $\delta + n$ on both sides, we have \[d^{\varphi^{(\delta)}} \left( \phi_{n-1}^{(n)}\right) = n \psi^{(\delta + n)}  \ . \] The element $\upsilon \coloneqq \frac{1}{n}\phi_{n-1}^{(n)}$ in $\mathfrak{g}_0^{(n)} $ thus satisfies \[ (\upsilon \cdot \psi)^{(k)} = 0  ~, \quad \mbox{for all } \  \delta +1  \leqslant  k \leqslant \delta +  n-1 ~. \]  Finally, the element $\lambda \coloneqq \mathrm{BCH}(\upsilon,\nu) \in \mathfrak{g}_0$ satisfies Point (2).
\end{proof}

\begin{proof}[Proof of Theorem \ref{A}] Let us prove Point (1). Suppose that there exists $\lambda \in \mathfrak{g}_0$ such that $\lambda \cdot \varphi = \varphi^{(\delta)}$~. By Lemma \ref{decomposition prismatic}, this implies that \[\mathfrak{D}(\lambda) \cdot \mathfrak{D}(\varphi)= \varphi^{(\delta)}  \ , \]  and $K_{\Phi} = 0$ by Theorem \ref{Kaledin1}. Conversely, we proceed as in the proof of Theorem \ref{Kaledin1}. If $K_{\Phi} = 0$~, we can use Proposition \ref{passage} to construct a sequence of integers $(n_i)$ greater that $1$ and a sequence of gauges $(\lambda^i)$ such that such that \[ (\lambda^{i} \cdot  \psi^i)^{(k)} = 0 ~, \quad \mbox{for all } \  \delta +1  \leqslant  k \leqslant \delta +  n_i ~, \] where $\psi^{i+1} \coloneqq \lambda^{{i}} \cdot \psi^{i}$ and $\psi^1 \coloneqq \varphi$~. The proof of Proposition \ref{passage} actually shows that one can find gauges such that $\lambda^{i} \in \mathfrak{g}^{(n_i)}_0$~, for all $i$~.  It follows from the construction that the gauge \[\lambda \coloneqq \mathrm{BCH}(\cdots \mathrm{BCH}(\lambda^{3}, \mathrm{BCH}(\lambda^{2}, \lambda^{1})) \cdots ) \ , \] satisfies $\lambda \cdot \varphi = \varphi^{(\delta)}$~. The proofs of points (2) and (3) make use of the Kaledin classes and Point (1) in a way similar to the proofs of Theorem \ref{intrinsec0} and Corollary \ref{descent1.0} respectively. 
\end{proof}

\noindent Theorem \ref{A} for $\delta = 1$ precisely corresponds to the formality setting addressed in Section \ref{section2}. 

\begin{remark}[Weight zero setting]
One can prove another version of Theorem \ref{A} with non-trivial weight zero component, a differential preserving the weight-grading, and the following definition of the prismatic decomposition of $\varphi \in \mathfrak{g}_{-1}$, \[\mathfrak{D} \left(\varphi\right) \coloneqq \varphi^{(0)} + \varphi^{(1)} \hbar + \varphi^{(2)} \hbar^2 + \cdots \in \mathrm{Def}_{\varphi^{(0)}}(R[\![\hbar]\!]) \ . \] More precisely, let $\mathfrak{g}$ be a differential graded Lie algebra with a compatible extra weight grading such that \[\mathfrak{g} \cong \prod_{k = 0}^{\infty} \mathfrak{g}^{(k)}, \quad d\left(\mathfrak{g}^{(k)} \right) \subset \mathfrak{g}^{(k)}, \quad \left[ \mathfrak{g}^{(k)},  \mathfrak{g}^{(l)} \right] \subset  \mathfrak{g}^{(k +l)} \ .  \] Then for every Maurer--Cartan element $\varphi \in \mathrm{MC}(\mathfrak{g})$~, one can prove that the Kaledin class $K_{\mathfrak{D} \left(\varphi\right)}$ is zero if and only if there exists $\lambda \in \mathcal{F}^1 \mathfrak{g}_0$ such that $\lambda \cdot \varphi = \varphi^{(0)}$~. 
\end{remark}

\begin{remark}[The complete setting]
The Kaledin classes cannot be defined without the weight-grading assumption. Nonetheless, in the spirit of Halperin and Stasheff's approach \cite{HJ79}, one forgo treating everything in a single class. In \cite[Section~1]{CE24b}, sequences of obstructions are constructed in order to detect gauge equivalences between a Maurer--Cartan element and its first component and more generally if two Maurer--Cartan elements are gauge equivalent in a complete dg Lie algebra. 
\end{remark}

\section{\textcolor{bordeau}{Operadic Kaledin classes over a commutative ground ring}}\label{section2}

In this section, we apply the Kaledin classes developed in Section \ref{section1} to study formality of algebras encoded by colored operads. In this particular context, we adapt the construction to deal with any commutative ground ring $R$. This generalizes the article \cite{MR19} of Melani and Rubi\'o to any commutative ring and to algebras over any reduced colored operad. \medskip

\noindent Let us fix $R$ a commutative ground ring and let $A$ be a differential graded $R$-module. Let $\C$ be a reduced weight-graded differential graded cooperad over $R$, i.e. $\C(0) = 0$ and \[\C =   \mathrm{I} \oplus \C^{(1)} \oplus \C^{(2)} \oplus  \cdots \oplus \C^{(n)} \oplus \cdots \quad \mbox{and} \quad d_{\C} \left(\C^{(k)}\right) \subset \C^{(k - 1)} \ .\]  Let us denote the coaugmentation coideal $\overline{\C} \coloneqq \bigoplus_{k \geqslant1}  \C^{(k)}$ so that $\C \cong \mathrm{I} \oplus \overline{\C}$~.

\begin{examples}\label{exemples}
Among possible choices for such cooperad $\C$, we have: 
\begin{enumerate}
	\item The bar construction $\C \coloneqq \Bar \P$ of a reduced operad $\P$, see \cite[Section~6.5]{LodayVallette12}~.
	
 	\item The Koszul dual cooperad $\C := \P^{\antishriek}$ of a homogeneous Koszul operad $\P$, see \cite[Section~7.4]{LodayVallette12}, or an inhomogeneous quadratic Koszul operad $\P$, see \cite[Section~7.6]{LodayVallette12}.
	
	\item The quasi-planar cooperad $\C = \Bar \left(\mathcal{E} \otimes \P\right) $ where $\mathcal{E}$ denotes the Barratt-Eccles operad and $\P$ is any reduced operad, see \cite{LRL23}. 
\end{enumerate}
\end{examples}

\subsection{Deformation complex of algebras over an operad}\label{2.1}

This section recalls the definition of a $\Cobar \C$-algebra structure on $A$~, as a Maurer--Cartan element of its associated convolution dg pre-Lie algebra $\mathfrak{g}_A$. We refer the reader to \cite[Section~5]{DSV16} for more details.

\begin{definition}[dg pre-Lie algebra]
	A \emph{dg pre-Lie algebra} is a chain complex $(\mathfrak{g},d)$ equipped with a linear map of degree zero $\star : \mathfrak{g}\otimes \mathfrak{g} \to \mathfrak{g}$~, the pre-Lie product, satisfying \[\forall x, y,z \in \mathfrak{a}, \, (x \star y) \star z  - x \star ( y \star z ) =  (-1)^{|y||z|} ((x \star z ) \star y  - x \star (z \star y)) \ . \]
Its set of \emph{Maurer--Cartan elements} is defined as \[\mathrm{MC(\mathfrak{g})} \coloneqq \lbrace x \in \mathfrak{g}_{-1} \mid dx + x\star x = 0 \rbrace \ . \]
\end{definition}

\begin{remark}
In a dg pre-Lie algebra, the skew-symmetrized bracket \[[x,y] := x \star y - (-1)^{|x||y|} y \star x \ ,\]  induces a dg Lie algebra structure. If $2$ is invertible in $R$~, we have $\varphi \star \varphi = \frac{1}{2}[\varphi, \varphi]$ and the Maurer--Cartan equation coincides with usual one.
\end{remark}

\begin{definition}[Convolution dg pre-Lie algebra]
The \emph{convolution dg pre-Lie algebra} associated to $\C$ and $A$ is the dg pre-Lie algebra \[\mathfrak{g}_{ A} := \left( \Hom_{\mathbb{S}} \left(\overline{\C}, \mathrm{End}_A\right), \star, d \right) \ , \]
where the underlying space is  \[\Hom_{\mathbb{S}} \left(\overline{\C}, \End_A\right) : = \prod_{p \geqslant0} \Hom_{\mathbb{S}} \left(\overline{\C}(p), \End_A(p)\right)\ ,\] equipped with the pre-Lie product \[\varphi \star \psi \coloneqq \overline{\C} \xrightarrow{\Delta_{(1)}} \overline{\C} \circ_{(1)} \overline{\C} \xrightarrow{\varphi \circ_{(1)} \psi} \End_A \circ_{(1)} \End_A \xrightarrow{\gamma_{(1)}} \End_A \ ,\] and the differential \[d (\varphi) = d_{\End_A} \circ \varphi - (- 1)^{|\varphi|} \varphi \circ d_{\overline{\C}} \ .\] It embeds inside the dg pre-Lie algebra made up of all the maps from $\C \cong \mathrm{I} \oplus \overline{\C}$, i.e. \[\mathfrak{g}_{A} \hookrightarrow \mathfrak{a}_{A}  \coloneqq \left( \Hom_{\mathbb{S}} \left(\C, \mathrm{End}_A\right), \star, d , 1 \right)  \] where the left unit is given by $1 : \mathrm{I} \mapsto \mathrm{id}_A$~. 
\end{definition}

\begin{proposition}
A \emph{$\Cobar \C$-algebra structure} $\varphi$ on $A$ is a Maurer--Cartan element in $\mathfrak{g}_A$~, i.e. \[\varphi \in \mathrm{MC}(\mathfrak{g}_{A}) = \{\varphi \in \mathfrak{g}_{A} \mid |\varphi|=-1, \; d(\varphi) + \varphi \star \varphi = 0\} \ .\] The \emph{deformation complex} of an $\Cobar \C$-algebra $\varphi \in \mathrm{MC}(\mathfrak{g}_{A})$ is the twisted dg pre-Lie algebra \[\mathfrak{g}_{A}^{\varphi} := \left( \Hom_{\mathbb{S}} \left(\overline{\C}, \mathrm{End}_A\right), \star \ , d^{\varphi} \coloneqq d + \mathrm{ad}_{\varphi} \right) \ .\] 
\end{proposition}

\begin{proof}
	See \cite[Proposition~10.1.1]{LodayVallette12}. 
\end{proof}

\begin{examples} The homology groups of the complex $\mathfrak{g}_{A}^{\varphi}$ is often called the \emph{tangent homology.}
	\begin{enumerate}
		\item In the case of the cooperad $\C = A^{\antishriek}$ and an associative algebra structure $\varphi$, the tangent homology is isomorphic to the Hochschild cohomology with coefficients into itself (up to a suspension) with Gerstenhaber bracket, see \cite[Section~9.1.6]{LodayVallette12}.
		
		\item  In the case of the cooperad $Com^{\antishriek}$ and a commutative algebra structure $\varphi$, one recovers the Harrison cohomology of commutative algebras, see \cite[Section~13.1.7]{LodayVallette12}.
		
		\item  For the operad $Lie^{\antishriek}$ and a Lie algebra structure $\varphi$, one gets the Chevalley--Eilenberg homology of Lie algebras, see \cite[Section~12.3]{LodayVallette12}. 
	\end{enumerate}
\end{examples}

\begin{proposition}\label{associative}
The algebra $\mathfrak{a}_A$ is an associative algebra with the
product defined by \[\varphi \circledcirc \psi \coloneqq \C \xrightarrow{\Delta} \C \circ \C \xrightarrow{\varphi \  \circ \ \psi} \End_A \circ \End_A \xrightarrow{\gamma} \End_A \ .\] 
\end{proposition}

\begin{proof}
See \cite[Section~10.2.3]{LodayVallette12}. 
\end{proof}

\begin{definition}[$\infty$-morphism]
Let $\varphi$ and $\psi$ be two $\Cobar \C$-algebra structures on $A$~. An \emph{$\infty$-morphism} $\varphi \rightsquigarrow \psi$ is a map $f \in \mathfrak{a}_A$ of degree $0$ satisfying \[f \star \varphi - \psi \circledcirc f = d(f) \ . \] The composite of two $\infty$-morphisms $f : \varphi \rightsquigarrow \psi$ and $g : \psi \rightsquigarrow \phi$ is given by $f \circledcirc  g$. 
\end{definition}

\begin{remark}
Let $\varphi \in  \mathrm{MC}(\mathfrak{g}_{A})$ and $\psi \in  \mathrm{MC}(\mathfrak{g}_{B})$ be two $\Cobar \C$-algebra structures. In this context, an $\infty$-morphism $\varphi \rightsquigarrow \psi$ can be defined as a morphism \[f \in \Hom_{\mathbb{S}} \left(\C, \End^A_B\right) \] of degree $0$ satisfying a similar identity, see \cite[Theorem~10.2.3]{LodayVallette12} for more details. 
\end{remark}

\noindent Since the cooperad $\C$ is weight-graded, the pre-Lie algebra $\mathfrak{a}_A$ admits a weight-grading, \[\Hom_{\mathbb{S}}\left(\C, \mathrm{End}_A\right) \cong \prod_{n \geqslant 0} \Hom_{\mathbb{S}} \left(\C^{(n)}, \mathrm{End}_A\right) \ . \] Every element $f \in \Hom_{\mathbb{S}} \left(\C, \mathrm{End}_A \right)$ decomposes as \[f = f^{(0)} + f^{(1)} + f^{(2)} + \cdots \] where $f^{(i)}$ is the restriction of $f$ to $\C^{(i)}$. The first component \[f^{(0)} : \mathrm{I} \to \End_A \] of $\infty$-morphism is equivalent to
a chain map in $\End_A$ that we still denote by $f^{(0)}$. The convolution dg pre-Lie algebra $\mathfrak{g}_A$ has a similar underlying weight grading and every element $\varphi$ decomposes as \[\varphi = \varphi^{(1)} + \varphi^{(2)} + \varphi^{(3)} + \cdots \ .\]

\begin{definition}
An $\infty$-morphism $f : \varphi \rightsquigarrow \psi$ between two $\Cobar \C$-algebra structures on $A$ is  
	\begin{itemize}
		 \item[$\centerdot$] an \emph{$\infty$-quasi-isomorphism} if $f^{(0)} \in \End_A$ is a quasi-isomorphism; 
		 \item[$\centerdot$] an \emph{$\infty$-isomorphism} if $f^{(0)} \in \End_A$ is an isomorphism; 
		 \item[$\centerdot$] an \emph{$\infty$-isotopy} if $f^{(0)} = \mathrm{id}_A$~. 
	\end{itemize}

\end{definition}

\begin{lemma}\label{invertibility} The $\infty$-isomorphisms are the isomorphisms in the category of $\Cobar \C$-algebra structures on $A$. We will denote by $f^{-1}$ the unique inverse of an $\infty$-isomorphism $f$ with respect to the product $\circledcirc$.
\end{lemma}

\begin{proof}
	See \cite[Theorem~10.4.1]{LodayVallette12}. 
\end{proof}

\begin{theorem}[{\cite[Theorem~3]{DSV16}}]\label{parallele}
Let $R$ be a $\mathbb{Q}$-algebra and let $A$ be a chain complex over $R$. The pre-Lie exponential map induces an isomorphism \[\mathrm{exp} : ((\mathfrak{g}_A)_0, \mathrm{BCH}, 0) \overset{\cong}{\longrightarrow} (\infty-\mathrm{iso}, \circledcirc, 1) \] between the gauge group and the group of $\infty$-isotopies. 
\end{theorem}

\subsection{Gauge formality}\label{2.2}

This section introduces the gauge formality notion studied in this article. It coincides with the standard formality definition in most cases, see Proposition \ref{formality-gauge}. Let us consider the operad 
\begin{equation}\label{operad P}
\P \coloneqq \mathcal{T}\left(s^{-1}\C^{(1)}\right) / \left(d_{\Cobar \C} \left(s^{-1}\C^{(2)}\right)\right) \ , 
\end{equation} which comes with a twisting morphism $\C \to \P$.

\begin{proposition}\label{induite}
Any $\Cobar \C$-algebra structure $(A, \varphi)$ induces a canonical $\P$-algebra structure \[(H(A),\varphi_*) \ . \]
\end{proposition}

\begin{proof}
	The homology of $(A, \varphi)$ carries a natural $\Cobar \C$-algebra structure \[ \psi : \C \longrightarrow \End_{H(A)} \ ,\] see \cite[Proposition~6.3.5]{LodayVallette12}. Since the homology has no differential, the component \[\varphi_* \coloneqq \psi^{(1)} \] is a $\Cobar \C$-algebra structure which appears as a $\P$-algebra over the operad $\P$ defined by (\ref{operad P}). 
\end{proof}

\begin{examples}
	Examples \ref{exemples} all comes equipped with a Koszul morphism $\kappa : \C \to \P$. 	
	\begin{enumerate}
		\item The bar construction $\C \coloneqq \Bar \P$ of a reduced operad $\P$ in $\mathrm{gr Mod}_R$ has an associated universal twisting morphism $\pi : \Bar \P \to \P$~, see \cite[Section~6.5]{LodayVallette12}~.
		
		\item The Koszul dual cooperad $\C := \P^{\antishriek}$ has the canonical twisting morphism $\kappa : \P^{\antishriek} \to \P$ .
		
		\item The quasi-planar cooperad $\C = \Bar \left(\mathcal{E} \otimes \P\right) $ where $\P$ is any reduced operad has an associated canonical twisting morphism $ \C \to \P$, see \cite{LRL23}. 
	\end{enumerate}
	
\end{examples}

\begin{definition}[Formality and gauge formality]
	A $\Cobar \C$-algebra structure $(A, \varphi)$ is
	
	\begin{enumerate}
		\item[$\centerdot$] \emph{formal} if there exists a zig-zag of quasi-isomorphisms of $\Cobar \C$-algebras \[(A, \varphi) \; \overset{\sim}{\longleftarrow} \;  \cdot \;  \overset{\sim}{\longrightarrow} \;  \cdots \;  \overset{\sim}{\longleftarrow} \; \cdot \;  \overset{\sim}{\longrightarrow} \; (H(A), \varphi_*) \ ;\] \smallskip
		
		\item[$\centerdot$] \emph{gauge formal} if there exists an $\infty$-quasi-isomorphism of $\Cobar \C$-algebras \begin{center}
			\begin{tikzcd}[column sep=normal]
				(A,\varphi) \ar[r,squiggly,"\sim"] 
				& (H(A), \varphi_*) \ .
			\end{tikzcd} 
		\end{center}
	\end{enumerate} 
\end{definition}

\noindent The formality and the gauge formality coincide in most examples, see Proposition \ref{formality-gauge}.

\begin{definition}[Transferred structure]
 A $\Cobar \C$-algebra $(A, \varphi)$ admits a \emph{transferred structure} if there exist a $\Cobar \C$-algebra structure $(H(A), \varphi_t)$, extending the $\P$-algebra structure $\varphi_*$ \[\varphi_t \coloneqq \varphi_* + \varphi_t^{(2)} + \varphi_t^{(3)} + \cdots \in \mathrm{MC}\left(\mathfrak{g}_{H(A)}\right) \ , \] and two $\infty$-quasi-isomorphisms 
 \[\hbox{
 	\begin{tikzpicture}
 		
 		\def\upshift{0.075}
 		\def\downshift{0.075}
 		\pgfmathsetmacro{\midshift}{0.005}
 		
 		\node[left] (x) at (0, 0) {$(A,\varphi)$};
 		\node[right=1.5 cm of x] (y) {$(H(A),\varphi_t) \ .$};
 		
 		\draw[->] ($(x.east) + (0.1, \upshift)$) -- node[above]{\mbox{\tiny{$p_{\infty}$}}} ($(y.west) + (-0.1, \upshift)$);
 		\draw[->] ($(y.west) + (-0.1, -\downshift)$) -- node[below]{\mbox{\tiny{$i_{\infty}$}}} ($(x.east) + (0.1, -\downshift)$);
 		
 \end{tikzpicture}}
 \] 
\end{definition}

\begin{proposition}\label{gauge formality} Let $(A, \varphi)$ be $\Cobar \C$-algebra structure. If there exists a transferred structure $(H(A), \varphi_t)$, the algebra $(A, \varphi)$ is gauge formal if and only if there exists an $\infty$-isotopy  \begin{center}
		\begin{tikzcd}[column sep=normal]
			f : (H(A), \varphi_t) \ar[r,squiggly,"="] 
			& (H(A), \varphi_*) \ .
		\end{tikzcd} 
	\end{center}
\end{proposition}

\begin{proof}
	The transferred structure comes equipped with two $\infty$-quasi-isomorphisms
	\[\hbox{
		\begin{tikzpicture}
			
			\def\upshift{0.075}
			\def\downshift{0.075}
			\pgfmathsetmacro{\midshift}{0.005}
			
			\node[left] (x) at (0, 0) {$(A, \varphi)$};
			\node[right=1.5 cm of x] (y) {$(H(A),\varphi_t) \ .$};
			
			\draw[->] ($(x.east) + (0.1, \upshift)$) -- node[above]{\mbox{\tiny{$p_{\infty}$}}} ($(y.west) + (-0.1, \upshift)$);
			\draw[->] ($(y.west) + (-0.1, -\downshift)$) -- node[below]{\mbox{\tiny{$i_{\infty}$}}} ($(x.east) + (0.1, -\downshift)$);
			
	\end{tikzpicture}}
	\] The algebra $(A,\varphi)$ is gauge formal if and only if there exists an $\infty$-quasi-isomorphism  \begin{center}
		\begin{tikzcd}[column sep=normal]
			f : (H(A), \varphi_t) \ar[r,squiggly,"\sim"] 
			& (H(A), \varphi_*) \ .
		\end{tikzcd} 
	\end{center} It is an $\infty$-isomorphism since the homology $H(A)$ has no differential. By definition of an $\infty$-morphism, the equation $f \star \varphi_t - \varphi_*  \circledcirc  f = d(f)$ holds and gives in weight 1, \[ f^{(0)} \star \left(\varphi_t\right)^{(1)} -  \varphi_*\circledcirc  f^{(0)} = - f^{(0)} \circ d_{\C} \ . \] Since $\varphi_* = \left(\varphi_t\right)^{(1)}$~, the component $f^{(0)}$ is an $\infty$-automorphism of $(H(A), \varphi_*)$ and the element \[ \left(f^{(0)}\right)^{-1}  \circledcirc f\] is an $\infty$-isotopy between $(H(A), \varphi_t)$ and $(H(A), \varphi_*)$. 
\end{proof}

\begin{definition}
	A chain complex $(A,d)$ is related to $H(A)$ via a \emph{contraction} if there exists \[
	\hbox{
		\begin{tikzpicture}
			
			\def\upshift{0.075}
			\def\downshift{0.075}
			\pgfmathsetmacro{\midshift}{0.005}
			
			\node[left] (x) at (0, 0) {$(A,d)$};
			\node[right=1.5 cm of x] (y) {$(H(A),0)$};
			
			\draw[->] ($(x.east) + (0.1, \upshift)$) -- node[above]{\mbox{\tiny{$p$}}} ($(y.west) + (-0.1, \upshift)$);
			\draw[->] ($(y.west) + (-0.1, -\downshift)$) -- node[below]{\mbox{\tiny{$i$}}} ($(x.east) + (0.1, -\downshift)$);
			
			\draw[->] ($(x.south west) + (0, 0.1)$) to [out=-160,in=160,looseness=5] node[left]{\mbox{\tiny{$h$}}} ($(x.north west) - (0, 0.1)$);
	\end{tikzpicture}} \] such that $ip - \mathrm{id}_A = d_A h + h d_A\ ,$  $pi = \mathrm{id}_{H(A)} \ ,$ $h^2 =0\ , $ $  ph = 0\ ,$ $hi =0 \ .   $
\end{definition}

\begin{example}
	If $R$ is a field, any chain complex $A$ is related to its homology $H(A)$ via a contraction. This is also the case when $A$ and $H(A)$ are degreewise projective and $R$ is an hereditary ring.
\end{example}

\begin{assumption}\label{assumptions 1}
	Let $R$ be a commutative ground ring and let $A$ be a chain complex related to $H(A)$ via a contraction. Let $\C$ be a reduced weight-graded dg cooperad over $R$. 
	\begin{itemize}
		\item[$\centerdot$] If $\C$ is a symmetric cooperad, we assume that $R$ is a $\mathbb{Q}$-algebra.
		\item[$\centerdot$]  If $\C$ is a symmetric quasi-planar cooperad, we assume that $R$ is a field. 
	\end{itemize}
\end{assumption}

\begin{theorem}[Homotopy transfer theorem] \label{HTT}  
\noindent Under Assumptions \ref{assumptions 1}, for any $\Cobar \C$-algebra structure, there exists a transferred structure $(H(A), \varphi_t)$  such that the embedding $i : H(A) \to A$ and the projection $p : A \to H(A)$ extend to $\infty$-quasi-isomorphisms 	\[\hbox{
	\begin{tikzpicture}
		
		\def\upshift{0.075}
		\def\downshift{0.075}
		\pgfmathsetmacro{\midshift}{0.005}
		
		\node[left] (x) at (0, 0) {$(A, \varphi)$};
		\node[right=1.5 cm of x] (y) {$(H(A),\varphi_t) \ .$};
		
		\draw[->] ($(x.east) + (0.1, \upshift)$) -- node[above]{\mbox{\tiny{$p_{\infty}$}}} ($(y.west) + (-0.1, \upshift)$);
		\draw[->] ($(y.west) + (-0.1, -\downshift)$) -- node[below]{\mbox{\tiny{$i_{\infty}$}}} ($(x.east) + (0.1, -\downshift)$);
		
\end{tikzpicture}}
\]  The transferred structure is independent of the choice of contraction in the following sense: any two such transferred structures are related by an $\infty$-isotopy. 
\end{theorem}

\begin{proof} We refer the reader to \cite{berglund},  \cite[Section~10.3]{LodayVallette12} and references therein. In the case of a symmetric quasi-planar cooperad, we refer the reader to \cite[Section~5.5]{LRL23}, 
\end{proof}

\begin{proposition} \label{formality-gauge} 
	Let $\C$ be a reduced weight-graded dg cooperad over $R$. 
	\begin{enumerate}
		\item If $R$ is a characteristic zero field, an $\Cobar \C$-algebra structure is formal if and only if it is gauge formal. 
						
		\item Suppose that $\C$ is non-symmetric. 
		\begin{enumerate}
			\item[i.] If $R$ is a field, then a formal $\Cobar \C$-algebra structure on $A$ is also gauge formal.

			\item[ii.] If $\C$ and $\Cobar \C$ are aritywise and degreewise flat, then a gauge formal $\Cobar \C$-algebra is formal. 
		\end{enumerate}
		
		\item   \cite[Proposition~43]{LRL23} Suppose that $\C$ is a symmetric quasi-planar cooperad. If $R$ is a field, then an $\Cobar \C$-algebra is formal if and only if it is gauge formal.   

	\end{enumerate}
\end{proposition}

\begin{proof}
Let $(A, \varphi)$ be a formal $\Omega \C$-algebra structure. If $R$ is a field and the hypotheses of Theorem \ref{HTT} are satisfied, we claim that it is also gauge formal. It suffices to prove that given a quasi-isomorphism of $\Omega \C$-algebras \[(A, \varphi) \longleftarrow (B, \psi) \ ,  \] there exists an $\infty$-quasi-isomorphism in the other direction. By Theorem \ref{HTT}, there exists $\infty$-quasi-isomorphisms
\begin{center}
	\begin{tikzcd}[column sep=normal]
	f : (H(B),\psi_t) \ar[r,squiggly,"i_{\infty} "]  & 	(B,\psi) \ar[r,squiggly,"\sim"] 
		& (A, \varphi)   \ar[r,squiggly," p_{\infty}"] 		 
		& (H(A), \varphi_t) \ .
	\end{tikzcd} 
\end{center} The $\infty$-isomorphism $f$ admits an inverse $f^{-1}$ by Lemma \ref{invertibility}. The composite 
\begin{center}
	\begin{tikzcd}[column sep=normal]
		 (A,\phi) \ar[r,squiggly,"p_{\infty}\ "]  &  (H(A), \varphi_t) \ar[r,squiggly,"f^{-1}"]  &
		(H(B), \psi_t) \ar[r,squiggly," i_{\infty}"] 		 
		& 	(B,\psi)  \ .
	\end{tikzcd} 
\end{center}
produces the required $\infty$-quasi-isomorphism. This proves (2).i. and the direct implications of (1) and (3). Conversely, let $(A, \varphi)$ be a gauge formal $\Omega \C$-algebra structure and let  \begin{center}
	\begin{tikzcd}[column sep=normal]
		f : (A,\varphi) \ar[r,squiggly,"\sim"] 
		& (H(A), \varphi_*) \ .
	\end{tikzcd} 
\end{center} be an $\infty$-quasi-isomorphism. By \cite[Proposition~11.3.1]{LodayVallette12}, the universal Koszul morphism  $\iota : \C \to \Cobar \C$ induces an adjunction \[
\hbox{
	\begin{tikzpicture}
		\def\upshift{0.185}
		\def\downshift{0.15}
		\pgfmathsetmacro{\midshift}{0.005}
		
		\node[left] (x) at (0, 0) {$\Cobar_{\iota} : \lbrace \mbox{dg } \C \mbox{-coalgebras}\rbrace$};
		\node[right] (y) at (2, 0) {$\lbrace \mbox{dg } \Cobar \C \mbox{-algebras}\rbrace : \Bar_{\iota} \ .$};
		
		\draw[-{To[left]}] ($(x.east) + (0.1, \upshift)$) -- ($(y.west) + (-0.1, \upshift)$);
		\draw[-{To[left]}] ($(y.west) + (-0.1, -\downshift)$) -- ($(x.east) + (0.1, -\downshift)$);
		
		\node at ($(x.east)!0.5!(y.west) + (0, \midshift)$) {\scalebox{0.8}{$\perp$}};
\end{tikzpicture}} 
\]  If $R$ is a characteristic zero field, this leads to a zig-zag of $\Cobar \C$-algebras quasi-isomorphisms \begin{equation}\label{recti}
(A,\varphi)  \underset{\sim}{\xleftarrow{\epsilon_{\iota}(A,\varphi)}  }  \Cobar_{\iota} \Bar_{\iota} (A,\varphi)  \underset{\sim}{\xrightarrow{\Cobar_{\iota} \Bar_{\iota} (f)}  }          \Cobar_{\iota} \Bar_{\iota} (H(A), \varphi_*) \underset{\sim}{\xrightarrow{\epsilon_{\iota}(H(A), \varphi_*)}  }  (H(A), \varphi_*) \ , 
\end{equation}
where the first and the third quasi-isomorphisms are given by the counit $\epsilon_{\iota}$ of the adjunction, see 
\cite[Theorem~11.3.3]{LodayVallette12} and $\Cobar_{\iota} \Bar_{\iota} (f) $ is a quasi-isomorphism by \cite[Theorem~11.4.7]{LodayVallette12}. This proves the converse implication of (1).

\noindent The ground ring $R$ is assumed to be a characteristic zero field in \cite[Section~11]{LodayVallette12} so that the 
operadic K\"unneth formula \cite[Proposition~6.2.3]{LodayVallette12} holds true. If $\C$ and $\Cobar \C$ are non-symmetric and aritywise and degreewise flat, there is a version of this proposition over any ring $R$, as mentioned in the proof of \cite[Proposition~1.18]{DCH21}. Thus, under the hypotheses of (2).ii., the zig-zag (\ref{recti}) also holds and $(A, \varphi)$ is formal.
\end{proof}

\begin{definition}[Gauge $n$-formality]\label{definition formality}
Let $(A, \varphi)$ be $\Cobar \C$-algebra such that there exists a transferred structure $(H(A), \varphi_t)$. It is called \emph{gauge $n$-formal} if there exist an $\Cobar \C$-algebra $(H(A),\psi)$ such that $\psi^{(k)} = 0$~, for all $2 \leqslant k \leqslant n + 1$ and an $\infty$-isotopy 	\begin{center}
		\begin{tikzcd}[column sep=normal]
			(H(A), \varphi_t) \ar[r,squiggly,"="] 
			& (H(A), \psi) \ .
		\end{tikzcd} 
	\end{center}
\end{definition}

\begin{remark}
The notion of gauge $n$-formality depends on the choice of weight-grading on $\C$. If the weight-grading is the homological degree of $\C$, it coincides with \cite[Definition~1.22]{DCH21}. If the weight-grading is given by the arity, one recovers the approach taken for instance in \cite{Lun07}, \cite{Sal17}, \cite{MR19}, etc. 
\end{remark}

\subsection{Pre-Lie calculus in operadic convolution dg algebras}\label{2.3}
As explained in Section \ref{2.1}, the pre-Lie algebra $\mathfrak{a}_A$ is equipped with a pre-Lie product $\star$ and an associative product $\circledcirc$. These two operations satisfy compatibility relations, some of them are proved over a characteristic zero field in \cite[Proposition~5]{DSV16}. In this section, we establish these relations over any commutative ground ring $R$~.

\begin{definition} For all $f,g, h\in \mathfrak{a}_A$~, we define $f\circledcirc \left(g ; h    \right) \in \mathfrak{a}_A$ as the following composition
\[f\circledcirc \left(g ; h    \right) \coloneqq \C \xrightarrow{\Delta} \C \circ \C \xrightarrow{\mathrm{id}_{\C} \circ'  \mathrm{id}_{\C}}  \C \circ \left(\C; \C \right)  \xrightarrow{f  \circ  (g;h)} \End_A \circ \left( \End_A ; \End_A \right) \xrightarrow{\pi}  \End_A \ ,\] where the map $\circ'$ is the infinitesimal composition of morphisms and where the map $\pi$ is \[ \End_A \circ \left( \End_A ; \End_A \right) \rightarrowtail \End_A \circ \left( \End_A \oplus \End_A \right) \twoheadrightarrow \End_A \circ  \End_A  \xrightarrow{\gamma} \End_A \ ,\]  see \cite[Section~6.1.1]{LodayVallette12} for more details.
\end{definition}

\begin{lemma}\label{compatibility}
For all $f,g, h, k\in \mathfrak{a}_A$~, the following identities hold
\begin{enumerate}
\item[(a)] $  \left(f \circledcirc g \right) \star h = f \circledcirc \left(g; g \star h \right) \ ,$
\item[(b)] $\left(f\circledcirc \left(g ; h    \right) \right) \circledcirc k = f\circledcirc \left(g \circledcirc  k  ; h  \circledcirc  k   \right) \ , $
	\item[(c)] $f \star h =  f \circledcirc \left(1;  h \right) \ .$
\end{enumerate}
If furthermore, the element $g$ is invertible with respect to $\circledcirc$\ , then  
\begin{enumerate}
	\item[(d)]  $ f \circledcirc \left(g;h\right) = \left( f \star \left(h \circledcirc g^{-1}\right)\right) \circledcirc g \ ,$
	\item[(e)] $\left(	\left( f \circledcirc g \right)  \star h \right) \circledcirc g^{-1} = f \star 	\left(	\left(    g \star h \right) \circledcirc g^{-1}  \right) \ . $
\end{enumerate}
\end{lemma}

\begin{proof} By coassociativity of the decomposition map $\Delta$~, the following diagram is commutative 
\[\begin{tikzcd}[description]
&\C \arrow{d}{ \Delta }   &\\
&\C \circ \C \arrow{d}{ \mathrm{id}_{\C} \circ'  \mathrm{id}_{\C}}  & \\
& \C \circ \left(\C; \C \right) \arrow{ddr}[above right]{\mathrm{id}_{\C} \circ \left(\mathrm{id}_{\C}, \Delta_{(1)} \right)} \arrow{ddl}[above left]{\Delta \circ \left(\eta , \mathrm{id}_{\C} \right)}& \\
&& \\
(\C \circ \C) \circ (\mathrm{I};\C) \arrow{rr}[above]{\cong} \arrow{dd}[left]{(f \circ g) \circ (\mathrm{id}_\mathrm{I}; h) }  & &   \C \circ \left(\C; \C \circ \left(\mathrm{I}; \C \right) \right)\arrow{dd}{f \circ \left(g; g \circ (\mathrm{id}; h) \right)} \\
&& \\
(\End_A \circ \End_A) \circ (\mathrm{I};\End_A) \arrow{rr}[above]{\cong} \arrow{dd}[left]{ \gamma \circ (\mathrm{id}_\mathrm{I} ; \mathrm{id}_{\End_A})  }  & &   \End_A\circ \left(\End_A; \End_A \circ \left(\mathrm{I}; \End_A \right) \right) \arrow{dd}[right]{\mathrm{id}_{\End_A} \circ \left(\mathrm{id}_{\End_A}; \gamma_{(1)}\right)} \\
&& \\
\End_A  \circ (\mathrm{I};\End_A) \arrow[dr, twoheadrightarrow]   & &   \End_A\circ \left(\End_A; \End_A \right) \arrow[dl, twoheadrightarrow]    \\
&\End_A \circ \End_A \arrow[d, "\gamma"]& \\
&\End_A & 
\end{tikzcd} 
\] \normalsize  The left-hand side is equal to $\left(f \circledcirc g\right) \star h$ and the right-hand side is equal to $f \circledcirc \left(g; g \star h \right)$, which gives Formula (a). The proof of Formula (b) is similar to the one of (a) and relies on the isomorphism \[ \left( \C \circ \left(\C; \C \right) \right) \circ \C \cong \C \circ \left( \C \circ \left( \C \circ \C; \C \circ \C \right)\right) \ . \] Formula (c) follows from (a) with $g= 1$~. Considering $k \coloneqq g^{-1}$ in Formula (b) and using Formula (c), we obtain Formula (d) :
\begin{equation*}
	f \circledcirc \left(g;h\right) \overset{(b)}{=}  \left(f\circledcirc \left(1  ; h  \circledcirc  g^{-1}   \right) \right) \circledcirc  g \overset{(c)}{=} \left( f \star \left(h \circledcirc g^{-1}\right)\right) \circledcirc g  \ .
\end{equation*} Formulas (a) and (d) give Formula (e) :
\begin{equation*}
	\left(\left(f \circledcirc g \right) \star h \right) \circledcirc g^{-1} \overset{(a)}{= } \left(f \circledcirc \left(g; g \star h \right) \right) \circledcirc g^{-1} \overset{(d)}{= }  f \star \left( \left(g \star h\right) \circledcirc g^{-1}\right) \ . \qedhere
\end{equation*}
\end{proof}

\begin{lemma}\label{differentielles} For all $f,g \in \mathfrak{a}_A$~, the following identities hold	
	\begin{enumerate}
		\item[(i)] 	$d \left(f \star g \right)  =  d(f)  \star g +   (-1)^{|f|} f \star  d(g)  \ ,   $
		\item[(j)] If $g$ is of degree zero, then $d \left(f \circledcirc g \right) =  d\left(f  \right) \circledcirc g + (-1)^{|f|} f \circledcirc  \left(g; d(g )  \right) \ , $ 
		\item[(k)]  $ d (1) = 0 \ .$
\end{enumerate} 
If furthermore, the element $g$ is invertible with respect to $\circledcirc$ , then  
	\begin{enumerate}
		\item[(l)] $d \left(g^{-1} \right) =  - g^{-1} \star \left( d(g ) \circledcirc g^{-1}  \right) \ .   $ 
	\end{enumerate}
\end{lemma}

\begin{proof} For points (i) and (j), we refer the reader to \cite[Section~6.4]{LodayVallette12}\ . Formula $(k)$ follows from  Formula (i) applied to $f = g = 1$. By Formula (j) applied to $\left(g^{-1} \circledcirc g \right) \circ d_{\C} = 1 \circ d_{\C} = 0$~, the following equality holds \[g^{-1} \circ d_{\C} = - \left( g^{-1} \circledcirc  \left(g; g \circ d_{\C}  \right)  \right) \circledcirc g^{-1} \ . \] Formula (d) of Lemma \ref{compatibility} implies formula (l). 
\end{proof}

\begin{proposition}\label{Adf}
 Let $\varphi$ and $\psi$ be two $\Cobar \C$-algebra structures on $A$. Every $\infty$-isomorphism $f : \varphi \rightsquigarrow \psi$ induces an isomorphism of dg Lie algebras  \[\mathrm{Ad}_f : \mathfrak{g}_A^{\varphi} \to \mathfrak{g}_A^{\psi}, \quad x \mapsto (f \star x) \circledcirc f^{-1} \ ,\] whose inverse is given by $\mathrm{Ad}_{f^{-1}}$~. 
\end{proposition}

\begin{proof}
The bracket is defined by the skew-symmetrization of the pre-Lie product, i.e. for all $x,y \in  \mathfrak{g}_H$~, we have \[\mathrm{Ad}_f \left([x, y]\right) =    \left(f \star \left(x \star y\right) \right)  \circledcirc f^{-1}  - (-1)^{|x||y|}  \left(f \star \left(y \star x \right) \right)  \circledcirc f^{-1}  \ . \] The definition of a pre-Lie product gives \[ f \star \left(x \star y\right)   = \left( f \star x \right) \star y - (-1)^{|x||y|} \left(  \left( f \star y \right) \star x  \right)    + (-1)^{|x||y|} \left( f \star \left(y \star  x \right)   \right) \ . \] This shows that  \[ \mathrm{Ad}_f \left([x, y]\right)  =  \left( \left( \mathrm{Ad}_f (x) \circledcirc f \right)  \star y \right)  \circledcirc f^{-1}  - (-1)^{|x||y|} \left(  \left( \mathrm{Ad}_f (y) \circledcirc f \right) \star x  \right) \circledcirc f^{-1} \ .  \] Formula (e) of Lemma \ref{compatibility} implies that \[ \mathrm{Ad}_f \left([x, y]\right) = [\mathrm{Ad}_f (x) , \mathrm{Ad}_f (y)] \ . \] Formula (i) of Lemma \ref{differentielles} implies that \[ \mathrm{Ad}_f \left( dx \right) \overset{(i)}{=}   d \left(f \star x \right)   \circledcirc f^{-1}  -   \left(  d \left(f  \right) \star x  \right) \circledcirc f^{-1} \ .   \] By Formula (j) of Lemma \ref{compatibility}, we have \[d \left(f \star x \right)   \circledcirc f^{-1}  \overset{(j)}{=} d \left(\mathrm{Ad}_f \left(x \right) \right) - (-1)^{|x|}\left(f \star x \right) \circledcirc  \left(f^{-1}; d \left(f^{-1} \right)  \right) \ . \] By Formula (l) and (a), we get \[ \left(f \star x \right) \circledcirc  \left(f^{-1}; d \left(f^{-1} \right)  \right) = - \mathrm{Ad}_f \left(x \right) \star  \left( d \left( f  \right) \circledcirc f^{-1}   \right) \] Formula (e) of Lemma \ref{compatibility} implies \[ \left(   d(f) \star x  \right) \circledcirc f^{-1} =  \left( d \left(f \right) \circledcirc f^{-1} \right)  \star  \mathrm{Ad}_f \left(x \right)   \ .  \] This leads to \[\mathrm{Ad}_f \left(d(x) \right) = d\left(\mathrm{Ad}_f \left(x \right) \right) -  \left[ d\left(f \right) \circledcirc f^{-1} , \mathrm{Ad}_f \left(x \right) \right]  \ .  \] Thus, for all $x \in \mathfrak{g}_H$~, we have \[\begin{split}
\mathrm{Ad}_{f}\left(d^{\varphi} (x)\right) & =  \mathrm{Ad}_{f}\left(  d(x) \right) +  \mathrm{Ad}_{f}\left( \left[\varphi,x\right]\right) \\
& =  d \left( \mathrm{Ad}_f \left(x \right) \right) -  \left[ d \left(f \right) \circledcirc f^{-1} , \mathrm{Ad}_f \left(x \right) \right] +  \left[  \mathrm{Ad}_{f}\left(\varphi \right), \mathrm{Ad}_{f}\left(x \right)\right]  \ .
\end{split}  \] By definition of an $\infty$-morphism, we have $f \star \varphi - \psi \circledcirc f = d(f)$ and thus
\[\psi = \mathrm{Ad}_{f}\left(\varphi\right) - d(f ) \circledcirc f^{-1} \ . \] We get $\mathrm{Ad}_{f}\left(d^{\varphi} (x)\right) = d^{\psi} \left(\mathrm{Ad}_{f}(x) \right) \ . $ Finally, Formula (e) of Lemma \ref{compatibility} implies that \[ \mathrm{Ad}_{f^{-1}} \left(\mathrm{Ad}_{f} \left(x\right) \right) =  \left(f^{-1} \star \left( (f \star x) \circledcirc f^{-1}\right) \right) \circledcirc f  \overset{(e)}{=}    \left( \left(	\left( f^{-1} \circledcirc f \right)  \star x \right) \circledcirc f^{-1}  \right) \circledcirc f = x \ .  \] Thus $\mathrm{Ad}_{f}$ is an isomorphism of dg Lie algebras whose inverse is given by $\mathrm{Ad}_{f^{-1}}$~.
\end{proof}

\begin{theorem}\label{action}$\leavevmode$
\begin{enumerate}
\item The set of degree zero elements $f \in \mathfrak{a}_A$ such that $f^{(0)} \in \End_A$ is an isomorphism forms a group with respect to the product $\circledcirc$. 
\item It acts on the set of Maurer--Cartan elements $\mathrm{MC}(\mathfrak{g}_A)$  under \[  f \cdot \varphi = \mathrm{Ad}_{f}\left(\varphi\right) - d(f) \circledcirc f^{-1} \ . \]
\end{enumerate}
\end{theorem}

\begin{proof}
Point (1) follows from Lemma \ref{invertibility}. We have $ 1 \cdot \varphi = \varphi$. Let $f$ and $g$ in $\mathfrak{a}_A$ whose weight-zero components are isomorphisms. Then, we have \[f \cdot (g \cdot \varphi) = \mathrm{Ad}_{f} \circ \mathrm{Ad}_{g}\left(\varphi\right) - \mathrm{Ad}_{f} \left( d(g) \circledcirc g^{-1} \right) - d(f) \circledcirc f^{-1} \ . \] Formula (e) of Lemma \ref{compatibility} implies that \[\mathrm{Ad}_{f} \circ \mathrm{Ad}_{g}\left(\varphi\right) = \mathrm{Ad}_{f \circledcirc g} \left(\varphi\right) \ . \] Formula (j) of Lemma \ref{differentielles} gives \[d \left(f \circledcirc g \right) \circledcirc g^{-1} \circledcirc f^{-1} =   f \circledcirc  \left(g; d(g )  \right)  \circledcirc g^{-1} \circledcirc f^{-1} + d\left(f  \right) \circledcirc f^{-1} \ .  \] By Formula (b) and (c) of Lemma \ref{compatibility}, we have \[f \circledcirc  \left(g; d(g )  \right)  \circledcirc g^{-1} \circledcirc f^{-1}  \overset{(b)}{=}  f \circledcirc  \left( 1; d(g )  \circledcirc g^{-1} \right)   \circledcirc f^{-1}  \overset{(c)}{=}  \mathrm{Ad}_{f} \left( d(g) \circledcirc g^{-1} \right) \ .   \] This proves that $f \cdot (g \cdot \varphi) = (f \circledcirc g) \cdot \varphi  \ .$
\end{proof}

\begin{lemma}\label{key lemma}
Let $H$ be a graded $R$-module and let $\varphi$ and $\psi$ be two $\Omega \C$-algebra structures. Let $n \geqslant1$ be an integer and suppose that there exists an $\infty$-isomorphism $f : \varphi \rightsquigarrow \psi $  such that
	
	\begin{itemize}
		\item[$\centerdot$] $\varphi^{(k)} = 0$~, for $2 \leqslant k \leqslant n$ and
		\item[$\centerdot$]  $f^{(k)} = 0$~, for $1 \leqslant k \leqslant n - 1$ . 
	\end{itemize}
	The weight components of $\psi$ are then given by 
	\begin{enumerate}
		\item[$\centerdot$]  $\psi^{(1)} =  \mathrm{Ad}_{f^{(0)}} \left( \varphi^{(1)} \right) + \left(f^{(0)} \circ d_{\C}\right) \circledcirc \left(f^{(0)}\right)^{-1}  $ ,
		\item[$\centerdot$] $\psi^{(k)} = 0$~,  for $2 \leqslant k \leqslant n$~, and
		\item[$\centerdot$]  $\psi^{(n+1)} = \mathrm{Ad}_{f^{(0)}} \left(\varphi^{(n+1)} \right) - d^{\psi^{(1)}} \left(f^{(n)} \circledcirc \left(f^{(0)}\right)^{-1}\right)$ .	
	\end{enumerate}
	In particular, if $f$ is an $\infty$-isotopy, then $\psi$ decomposes as
	\[ \psi =   \varphi^{(1)} + \varphi^{(n+1)} - d^{ \varphi^{(1)}} \left( f^{(n)} \right) + \psi^{(n+2)} + \cdots  \ .\]
\end{lemma}

\begin{proof}
	By definition of an $\infty$-morphism, we have \begin{equation}\label{infini}
 \psi \circledcirc f = 	f \star \varphi + f \circ d_{\C} \ .
	\end{equation}
	The assumptions on $\varphi$ and on $f$ imply that 	\[f \star \varphi \equiv f^{(0)} \star \varphi^{(1)} + f^{(0)} \star \varphi^{(n+1)} + f^{(n)} \star \varphi^{(1)} \pmod{\mathcal{F}^{n+2} \mathfrak{g}_H} \ ,\]\[ \psi \circledcirc f  \equiv \psi \circledcirc f^{(0)}  + \psi^{(1)} \circledcirc \left( f^{(0)} ; f^{(n)} \right)   \pmod{\mathcal{F}^{n+2} \mathfrak{g}_H} \ ,\] \[ \mbox{and} \quad \quad f \circ d_{\C} \equiv f^{(0)} \circ d_{\C} + f^{(n)} \circ d_{\C} \pmod{\mathcal{F}^{n+2} \mathfrak{g}_H} \ .\]	Equation (\ref{infini}) implies that \[\begin{split}
	\psi \  \equiv \  & \mathrm{Ad}_{f^{(0)}} \left( \varphi^{(1)} \right) + \mathrm{Ad}_{f^{(0)}} \left( \varphi^{(n+1)} \right) + \left(f^{(n)} \star \varphi^{(1)} \right) \circledcirc \left(f^{(0)} \right)^{-1}+ \left(f^{(0)} \circ d_{\C}\right) \circledcirc \left(f^{(0)} \right)^{-1} \\
	&+ \left(f^{(n)} \circ d_{\C}\right) \circledcirc \left(f^{(0)} \right)^{-1}  - \left(\psi^{(1)} \circledcirc \left( f^{(0)} ; f^{(n)} \right)    \right) \circledcirc \left(f^{(0)} \right)^{-1}
	 \pmod{\mathcal{F}^{n+2} \mathfrak{g}_H} \ .
	\end{split} \]  This gives the desired formulas for the components $\psi^{(k)}$ for $1 \leqslant k \leqslant n$. Furthermore, by Formula (e) of Lemma \ref{compatibility}, we have \[\left( f^{(n)} \star \varphi^{(1)} \right) \circledcirc \left(f^{(0)}\right)^{-1} \overset{(e)}{=} \left(f^{(n)} \circledcirc \left(f^{(0)}\right)^{-1} \right) \star \mathrm{Ad}_{f^{(0)}} \left(\varphi^{(1)} \right)    \] and it follows from Formulas (b) and (c) of Lemma \ref{compatibility}, that \[\left(\psi^{(1)} \circledcirc \left( f^{(0)} ; f^{(n)} \right)    \right) \circledcirc \left(f^{(0)} \right)^{-1} = \psi^{(1)} \star \left(f^{(n)} \circledcirc \left(f^{(0)}\right)^{-1} \right) \ .  \] Formula (j) of Lemma \ref{differentielles} implies that \[\left(f^{(n)} \circ d_{\C} \right)  \circledcirc \left(f^{(0)}\right)^{-1} \overset{(j)}{=} \left(f^{(n)} \circledcirc \left(f^{(0)}\right)^{-1} \right) \circ d_{\C} - f^{(n)}  \circledcirc \left(\left(f^{(0)}\right)^{-1} ; \left(f^{(0)}\right)^{-1} \circ d_{\C} \right)  \ .  \] By Formula (l) of Lemma \ref{differentielles} and Formula (a) of Lemma \ref{compatibility}, we have \[ - f^{(n)}  \circledcirc \left(\left(f^{(0)}\right)^{-1} ; \left(f^{(0)}\right)^{-1} \circ d_{\C} \right) \overset{(a)}{=} \left(f^{(n)} \circledcirc \left(f^{(0)}\right)^{-1} \right)  \star \left( \left( f^{(0)} \circ d_{\C}  \right) \circledcirc \left(f^{(0)}\right)^{-1}  \right) \ . \] We obtain that \[ \psi^{(n+1)} = \mathrm{Ad}_{f^{(0)}} \left( \varphi^{(n+1)} \right) - \left[ \psi^{(1)} ,   f^{(n)}  \circledcirc \left(f^{(0)}\right)^{-1}   \right] +  \left(f^{(n)} \circledcirc \left(f^{(0)}\right)^{-1} \right) \circ d_{\C} \ . \qedhere \]
\end{proof}

\subsection{Operadic Kaledin classes}\label{2.4}

Let $(H, \varphi)$ be an $\Omega \C$-algebra structure on a graded $R$-module $H$. In the spirit of gauge formality and Proposition \ref{gauge formality}, this section addresses the following problem: does there exist an $\infty$-isotopy between $\varphi$ and $\varphi^{(1)}$ ? If $R$ is a $\mathbb{Q}$-algebra, this is equivalent to the existence of a gauge $\lambda \in \mathfrak{g}_{H}$ such that \[\lambda \cdot \varphi = \varphi^{(1)} \ ,\] through the isomorphism given by Theorem \ref{parallele}. This last problem is fully characterized by the Kaledin classes, applying Theorem \ref{A} to the $1$-weight-graded dg Lie algebra $\mathfrak{g}_{H}$~. Nonetheless, we aim at studying this question for any base rings $R$. If $R$ that is not a $\mathbb{Q}$-algebra, the gauge group do not even exist, neither the isomorphism of Theorem \ref{parallele}. However, we can still ask for the existence of an $\infty$-isotopy $\varphi \rightsquigarrow \varphi^{(1)}$. The goal of this section is to adapt the methods of Section \ref{1.3} to deal with any commutative ground ring.

\begin{remark}
For the sake of clarity, we deal with the case $\delta = 1$. But if $d_{\C} = 0$ then $\mathfrak{g}_H$ is also $\delta$-weight-graded dg Lie algebra for all $\delta \geqslant 1$. All what follows adapts directly to detect if an $\Omega \C$-algebra structure $\varphi \in  \mathcal{F}^{\delta}(\mathfrak{g}_H)$ is related to $\varphi^{\delta}$ by an $\infty$-isotopy. 
\end{remark}

\begin{definition}[Operadic Kaledin classes]
	Let $\varphi \in \mathrm{MC}(\mathfrak{g}_H)$ be an $\Cobar \C$-algebra structure on a graded $R$-module $H$ and let $\mathfrak{D}(\varphi)$ be its prismatic decomposition. The \emph{Kaledin class} $K_{\varphi}$ is the Kaledin class of $\mathfrak{D}(\varphi)$, i.e. \[K_{\varphi} =  \left[\partial_{\hbar}   \mathfrak{D}(\varphi) \right] \in H_{-1}\left(\mathfrak{g}_H[\![\hbar]\!]^{\Phi}\right) \ . \] Its \emph{$n^{\text{th}}$-truncated Kaledin class} $K_{\varphi}^n$ is the $n^{\text{th}}$-truncated Kaledin class of $\mathfrak{D}(\varphi)$, i.e. \[ K^n_{\varphi} = \left[\varphi^{(2)} + 2 \varphi^{(3)} \hbar + \dots + n \varphi^{(n+1)} \hbar^{n-1} \right]  \in H_{-1}\left(\left(\mathfrak{g}_H[\![\hbar]\!]/(\hbar^{n })\right)^{\overline{\Phi}^n}\right) \ .  \] 
\end{definition}

\begin{proposition}\label{Kaledin} 	Let $R$ be a commutative ground ring and let $\C$ be a reduced weight-graded dg cooperad over $R$. Let $\varphi \in \mathrm{MC}(\mathfrak{g}_H)$ be an $\Cobar \C$-algebra structure on a graded $R$-module $H$~. 
\begin{enumerate}
	\item  Let $n \geqslant 1$ be an integer such that $n !$ is a unit in $R$. The $n^{\text{th}}$-truncated Kaledin class $K^n_{\varphi}$  is zero if and only if there exists an $\infty$-isotopy $f : \varphi \rightsquigarrow  \psi $ where $\psi \in \mathrm{MC}(\mathfrak{g}_H)$ is such that $\psi^{(k)} = 0$~, for $2 \leqslant k \leqslant n + 1 $. 
	
	\item If $R$ is a $\QQ$-algebra, the Kaledin class  $K_{\varphi}$ vanishes if and only if there exists an $\infty$-isotopy $f : \varphi \rightsquigarrow \varphi^{(1)} $~.
\end{enumerate}
\end{proposition}

\begin{remark}
If $R$ is a $\QQ$-algebra, Proposition \ref{Kaledin} is a direct corollary of Proposition \ref{passage} and Theorem \ref{A} through the isomorphism of Theorem \ref{parallele}. In order to deal with any commutative ground ring, we are going to adapt the proofs, mainly by using $\infty$-isotopies of the form $1 + \lambda$~, instead of exponential maps $e^{\lambda}$~.  
\end{remark}

\begin{definition}
The \emph{prismatic decomposition} of $f  \in \mathfrak{a}_H$ in degree zero is defined by \[\mathfrak{D} \left(f\right) \coloneqq  f^{(0)} + f^{(1)}\hbar +  f^{(2)}\hbar^2 + \cdots \in \mathfrak{a}_H[\![\hbar]\!] \ .\] 
\end{definition}

\begin{remark}
The pre-Lie product $\star$ and the circle product $\circledcirc$ extend by $\hbar$-linearity to $\mathfrak{a}_H[\![\hbar]\!]$. For all $f,g \in \mathfrak{a}_H$ of degree zero, we have \[\mathfrak{D} \left(f \star g \right) = \mathfrak{D} \left(f\right) \star \mathfrak{D} \left(g\right) \quad \mbox{and} \quad \mathfrak{D} \left(f \circledcirc g \right) = \mathfrak{D} \left(f\right) \circledcirc \mathfrak{D} \left(g\right)  \ . \] Furthermore, if $f : \varphi \rightsquigarrow \psi$ is an $\infty$-isomorphism then $\mathfrak{D} \left(f\right)$ satisfies \[\mathfrak{D} \left( \psi \right) = \left( \mathfrak{D} \left(f\right) \star \mathfrak{D} \left( \varphi \right)  \right)\circledcirc \mathfrak{D} \left(f\right)^{-1} + \left(\mathfrak{D} \left(f\right) \circ d_{\C}\right) \circledcirc \mathfrak{D} \left(f\right)^{-1} \ .\]
\end{remark}

\begin{lemma}\label{compatibility2} For all $F, G \in \mathfrak{a}_H[\![\hbar]\!]$~, the following identities hold:
	\begin{enumerate}
		\item $\partial_{\hbar} (F \star G) = \partial_{\hbar} F \star G + F \star  \partial_{\hbar} G \ ,$
		\item $\partial_{\hbar} (F \circledcirc G) = \partial_{\hbar} F \circledcirc  G + F \circledcirc (G; \partial_{\hbar} G)  \ ,$ 
\end{enumerate}
If furthermore, the element $F$ is invertible with respect to $\circledcirc$~, then  
\begin{enumerate}
		\item[(3)]  $\partial_{\hbar} (F^{-1}) = - F^{-1} \star \left( \partial_{\hbar} F \circledcirc  F^{-1}  \right)\ .$ 
	\end{enumerate}
\end{lemma}

\begin{proof} Points (1) and (2) follow from the $\hbar$-linearity of the pre-Lie product. 
Furthermore, by Point (2), applied to $\partial_{\hbar} (F^{-1} \circledcirc F) = \partial_{\hbar} \left(1 \right) = 0  \ ,$ the following equality holds \[ \partial_{\hbar} \left(F^{-1} \right)  = - \left( F^{-1} \circledcirc (F; \partial_{\hbar} F)  \right) \circledcirc  F^{-1} \ . \] Formula (d) of Lemma \ref{compatibility} then implies \[ \partial_{\hbar} \left(F^{-1} \right)  \overset{(d)}{=} -  F^{-1} \star \left( \partial_{\hbar} F \circledcirc F^{-1}\right)  \ . \qedhere \]
\end{proof}

\begin{lemma} Let $n \geqslant 1$ and let $f : \varphi \rightsquigarrow \psi$ be an $\infty$-isomorphism. Let us denote $\Phi \coloneqq \mathfrak{D}(\varphi)$~, $\Psi \coloneqq \mathfrak{D}(\psi)$ and $F \coloneqq \mathfrak{D}(f)$. There are isomorphisms of dg Lie algebras \[\mathrm{Ad}_F : \mathfrak{g}_H[\![\hbar]\!]^{\Phi} \to \mathfrak{g}_H[\![\hbar]\!]^{\Psi} \quad \mbox{and} \quad \mathrm{Ad}_{F}^n :  \left(\mathfrak{g}_H[\![\hbar]\!] /(\hbar^{n}) \right)^{\overline{\Phi}^n} \to  \left(\mathfrak{g}_H[\![\hbar]\!] /(\hbar^{n}) \right)^{\overline{\Psi}^n}   \] induced by the mapping $x \mapsto (F \star x) \circledcirc F^{-1}$.
\end{lemma}

\begin{proof}
The proof of Proposition \ref{Adf} holds \emph{mutatis mutandis} by $\hbar$-linearity. 
\end{proof}

\begin{lemma}[Invariance of Kaledin classes under $\infty$-isomorphisms]\label{piquant} Let $n \geqslant 1$ and let $f : \varphi \rightsquigarrow \psi$ be an $\infty$-isomorphism. Let us denote $\Phi \coloneqq \mathfrak{D}(\varphi)$~, $\Psi \coloneqq \mathfrak{D}(\psi)$ and $F \coloneqq \mathfrak{D}(f)$.
	\begin{enumerate}
  	\item The classes $K_{\Psi}$ and $\mathrm{Ad}_{F}\left(K_{\Phi}\right)$ are equal in \[H_{-1}\left(\mathfrak{g}_H[\![\hbar]\!]^{\Psi}\right) \ ,\] where we still denote by $\mathrm{Ad}_{F}$ the induced isomorphism on homology. 
		\item The classes $K_{\Psi}^n$ and $\mathrm{Ad}_{F}^n\left(K_{\Phi}^n\right)$ are equal in \[H_{-1}\left(\left(\mathfrak{g}_H[\![\hbar]\!] /(\hbar^{n}) \right)^{\overline{\Psi}^n}\right) \ ,\] where we still denote by $\mathrm{Ad}_{F}^n$ the induced isomorphism on homology. 
	\end{enumerate}
\end{lemma}

\begin{proof} Let us prove point $(1)$~. By definition of an $\infty$-isomorphism, we have \[\psi = \left( f \star \varphi \right)\circledcirc f^{-1} + \left(f \circ d_{\C}\right) \circledcirc f^{-1} \] and thus $ \Psi = \mathrm{Ad}_F(\Phi) + \left(F \circ d_{\C}\right) \circledcirc F^{-1} $~. Let us prove that the cycles $\mathrm{Ad}_{F} (\partial_{\hbar} \Phi)$ and $\partial_{\hbar} \Psi$ are homologous. The relations of Lemma \ref{compatibility2} imply that \[
\partial_{\hbar} \left( \mathrm{Ad}_{F} (\Phi) \right)  = \left(\partial_{\hbar} F  \star \Phi \right) \circledcirc F^{-1} + \left(  F \star \partial_{\hbar}  \Phi   \right) \circledcirc F^{-1} +  (F \star \Phi) \circledcirc \left(F^{-1}; \partial_{\hbar} (F^{-1}) \right) \ .   \] By Formula (e) of Lemma \ref{compatibility}, the following equality holds \[\left(\partial_{\hbar} F  \star \Phi \right) \circledcirc F^{-1} =   \left(\partial_{\hbar} F \circledcirc F^{-1} \right)\star  \mathrm{Ad}_{F} (\Phi) \ . \] Using the formula of $\partial_{\hbar} (F^{-1})$ given by Lemma \ref{compatibility2} and Formula (a) of Lemma \ref{compatibility}, we have  \[ \begin{split}
	(F \star \Phi) \circledcirc \left(F^{-1}; \partial_{\hbar} (F^{-1}) \right) &  = -  (F \star \Phi) \circledcirc \left(F^{-1}; F^{-1} \star \left( \partial_{\hbar} F \circledcirc  F^{-1}  \right) \right) \\
	& \overset{(a)}{=} - \mathrm{Ad}_{F} (\Phi)  \star \left(\partial_{\hbar} F \circledcirc F^{-1} \right) \ .
\end{split} \]
This shows that $ \partial_{\hbar} \left( \mathrm{Ad}_{F} (\Phi) \right) = \mathrm{Ad}_{F} (\partial_{\hbar} \Phi) -  \left[\mathrm{Ad}_{F} (\Phi), \partial_{\hbar} F \circledcirc F^{-1} \right] \ .$ Formula (2) of Lemma \ref{compatibility2} implies that \[\partial_{\hbar} \left( \left(F \circ d_{\C}\right) \circledcirc F^{-1}\right) =  \left( \partial_{\hbar} F \circ d_{\C}\right) \circledcirc F^{-1} +  \left(  F \circ d_{\C}\right) \circledcirc  \left(F^{-1}; \partial_{\hbar} (F^{-1}) \right) \ . \]  Using the expression of $\partial_{\hbar} (F^{-1})$ and Formula (a) of Lemma \ref{compatibility}, we have  \[
\left(  F \circ d_{\C}\right) \circledcirc  \left(F^{-1}; \partial_{\hbar} (F^{-1}) \right) \overset{(a)}{=} - \left(\left(  F \circ d_{\C}\right)  \circledcirc  F^{-1} \right) \star \left(\partial_{\hbar} F \circledcirc F^{-1} \right) \ .\] Formula (j) of Lemma \ref{differentielles} implies that \[\left( \partial_{\hbar} F \circ d_{\C}\right) \circledcirc F^{-1} = \left(\partial_{\hbar} F \circledcirc F^{-1} \right) \circ d_{\C} - \partial_{\hbar} F \circledcirc \left( F^{-1}; F^{-1} \circ d_{\C} \right) \ . \]  Using the expression of $F^{-1} \circ d_{\C}$ given by Lemma \ref{differentielles} and Formula (a) of Lemma \ref{compatibility}, we have \[\partial_{\hbar} F \circledcirc \left( F^{-1}; F^{-1} \circ d_{\C} \right) \overset{(a)}{=} -  \left(\partial_{\hbar} F \circledcirc F^{-1} \right) \star \left(\left(  F \circ d_{\C}\right)  \circledcirc  F^{-1} \right) \ . \] This leads to \[\partial_{\hbar} \left( \left(F \circ d_{\C}\right) \circledcirc F^{-1}\right) = \left(\partial_{\hbar} F \circledcirc F^{-1} \right) \circ d_{\C} - \left[ \left(  F \circ d_{\C}\right)  \circledcirc  F^{-1}  , \partial_{\hbar} F \circledcirc F^{-1} \right]~, \quad \mbox{and} \] \begin{equation}\label{homologues}
\mathrm{Ad}_{F} (\partial_{\hbar} \Phi) - \partial_{\hbar} \Psi  = d^{\Psi} \left(\partial_{\hbar} F \circledcirc F^{-1}  \right) \ .
\end{equation}
The proof of Point (2) is similar.
\end{proof}

\begin{proof}[Proof of Proposition \ref{Kaledin}] 
	
Let us prove Point (1). Let $\psi \in \mathrm{MC}(\mathfrak{g}_H)$ be such that $\psi^{(k)} = 0$~, for $2 \leqslant k \leqslant n+1$~, and let $\varphi \rightsquigarrow \psi$ be an $\infty$-isotopy. Point (2) of Lemma \ref{piquant} implies that \[K_{\Phi}^n = \left(\mathrm{Ad}_{F}^n \right)^{-1} \left(K_{\Psi}^n\right) = \left(\mathrm{Ad}_{F}^n \right)^{-1} \left(K_{\varphi^{(1)}}^n\right) = 0 \ . \] Let us prove the converse result by induction on $n \geqslant 1$. Suppose that \[ K^1_{\Phi} = 0 \in H_{-1}\left(\mathfrak{g}^{\varphi^{(1)}}_H\right) \ . \] There exists $\upsilon \in \mathfrak{g}_H$ of degree zero such that $ d^{\varphi^{(1)}} (\upsilon) =\varphi^{(2)} \ .$ Let us set $f \coloneqq 1 + \upsilon^{(1)}$ and  \[\psi \coloneqq (f \star \varphi) \circledcirc f^{-1} + \left(f \circ d_{\C}\right) \circledcirc f^{-1} \ . \] By construction, $\psi \in \mathrm{MC}(\mathfrak{g}_H)$ and $f$ determines an $\infty$-isotopy $ \varphi \rightsquigarrow  \psi \ .$ Lemma \ref{key lemma} gives \[
		\psi^{(2)}  = \varphi^{(2)} - d^{\varphi^{(1)}} \left(\upsilon^{(1)}\right) = 0 \ .  \] Suppose that $n > 1$ and that the result holds for $n-1$~. If the class $K_{\Phi}^n$ is zero, so does the class $K_{\Phi}^{n-1}$. By the induction hypothesis, there exist $\psi \in \mathrm{MC}(\mathfrak{g}_H)$ such that $\psi^{(k)} = 0$ for $2 \leqslant k \leqslant n$ and an $\infty$-isotopy $g : \varphi \rightsquigarrow \psi$~. Point (2) of Lemma \ref{piquant} implies that \[\mathrm{Ad}_{\mathfrak{D}(g)}^n(K^{n}_{\Phi})= K^{n}_{\mathfrak{D}(\psi)} = [n \psi^{(n+1)} \hbar^{n-1}] = 0 \ .\] There exists $\upsilon \coloneqq \upsilon_0 + \upsilon_1 h + \dots + \upsilon_{n-1}\hbar^{n-1} \in \mathfrak{g}_H[\![\hbar]\!]$ of degree zero such that \[d^{\varphi^{(1)}}(\upsilon) \equiv n \psi^{(n+1)} \hbar^{n-1} \pmod{\hbar^{n}} \ .\] Looking at the coefficient of $\hbar^{n-1}$ and in weight $n +1$ on both sides, we have\[ d^{\varphi^{(1)}} \left( \upsilon_{n-1}^{(n)} \right) = n \psi^{(n+1)}  \ . \]  Let us consider $\lambda \coloneqq \frac{1}{n} \upsilon_{n-1}^{(n)}$ and $f \coloneqq 1 + \lambda$~. We set \[\psi' \coloneqq (f \star \psi) \circledcirc f^{-1} + \left(f \circ d_{\C}\right) \circledcirc f^{-1} \in \mathrm{MC}(\mathfrak{g}_H) \ .\] By construction, $f \circledcirc g$ determines an $\infty$-isotopy $ \varphi \rightsquigarrow  \psi' \ .$ Lemma \ref{key lemma} implies that \[\left(\psi'\right)^{(k)} = 0~, \quad \mbox{for} \; 2 \leqslant k \leqslant n~, \quad \mbox{and} \quad
		(\psi')^{(n+1)}  = \psi^{(n+1)} - d^{\varphi^{(1)}} \left( \lambda\right) = 0 \ , \] and we get the desired result. Let us prove Point (2). If there exists an $\infty$-isotopy $f: \varphi \rightsquigarrow \varphi^{(1)}$~, Point (1) of Lemma \ref{piquant} implies that \[K_{\Phi} =  \mathrm{Ad}_{F}^{-1} \left(K_{\varphi^{(1)}}\right) = 0 \ . \]  Conversely, suppose that $K_{\Phi} = 0$~. If $\Phi = \varphi^{(1)}$~, the result is immediate. Otherwise, let $n_1 > 1$ be the smallest integer such that $\varphi^{(n_1)}$ is not zero. Since $K_{\Phi} = 0$~, the class $K_{\Phi}^{n_1}$ is also zero. By the proof of Point (1), there exist $\psi \in \mathrm{MC}(\mathfrak{g}_H)$ such that $\psi^{(k)}= 0$ for $2  \leqslant k \leqslant n_1 +1 $ and an $\infty$-isotopy between $\varphi$ and $\psi$ of the form \[f_{n_1} \coloneqq  1 + \lambda_1 \ , \quad \mbox{where } \lambda_1 \in \mathfrak{g}_H^{(n_1)} \ . \] We can repeat the procedure with $\psi_{n_1} \coloneqq \psi$~. If $R$ is a $\mathbb{Q}$-algebra, we obtain a sequence  \[n_1 < n_2 < n_3 < \cdots\] of integers and a sequence of $\Cobar \C$-algebra structures $\psi_{n_i}$ such that $\psi_{n_i}$ is zero in weight $k$ for $2 \leqslant k \leqslant n_i +1$ associated to a sequence of $\infty$-isotopies 
\begin{center}
	\begin{tikzcd}[column sep=normal]
		\varphi \ar[r,squiggly,"f_{n_1}"] 
		& \psi_{n_1} \ar[r,squiggly,"f_{n_2}"] 
		& \psi_{n_2}  \ar[r,squiggly,"f_{n_3}"] 
		& \cdots \ .
	\end{tikzcd} 
\end{center}  
where $f_{n_i}$ is of the form $1 + \lambda_i$ with $\lambda_i$ concentrated in weight $n_i$. The composite \[f \coloneqq  \cdots f_{n_3} \circledcirc f_{n_2} \circledcirc f_{n_1} \] is well-defined by construction and induces an $\infty$-isotopy $f : \varphi \rightsquigarrow \varphi^{(1)}$~. \qedhere
\end{proof}

\begin{proposition}\label{toutes nulles 2}
Let $H$ be a graded $R$-module and let $(H,\varphi)$ be an $\Cobar \C$-algebra structure. If $R$ is a $\mathbb{Q}$-algebra, the Kaledin class $K_{\varphi}$ is zero if and only if, for all $n \geqslant 1$, the $n^{\text{th}}$-truncated Kaledin class $K^n_{\varphi}$ is zero. 
\end{proposition}

\begin{proof}
The proof makes use of the Kaledin classes and Proposition \ref{Kaledin} in a way similar to the proof of Proposition \ref{toutes nulles}. 
\end{proof}

\subsection{Kaledin classes in the colored setting}\label{2.5}

All the previous results apply \emph{mutatis mutandis} to operads colored in groupoids. The notion of a colored operad is a generalization of the notion of operads in which each input or output comes with a color chosen in a given set. A composition is possible whenever the colors of the corresponding input and output involved match. While an operad acts on a single chain complex, colored operads encode operations acting on various chain complexes. In order to encode symmetries, one can also consider groupoid colored operads where colors are chosen in a given groupoid $\mathbb{V}$. The associated Koszul duality theory was developed by Ward in \cite{War19}. We refer the reader to \cite[Chapter 2, Section 6.4]{RL22} for more details.

\begin{example}[Endomorphism operad] Let $A$ be a $\mathbb{V}$-module, i.e. a functor $A : \mathbb{V} \to \mathrm{dgMod_R}$. The associated endomorphism $\mathbb{V}$-colored operad is given by \[\End_A \left(v_0; v_1, \dots , v_r\right) \coloneqq \Hom\left( A(v_0); A(v_1) \otimes \cdots \otimes A(v_r)\right)\] with the $\mathbb{S}_r$ action given by permuting the inputs. 
\end{example}

\noindent We fix $\mathbb{V}$ a groupoid and $\C$ a reduced $\mathbb{V}$-colored weight-graded dg cooperad over $R$. For instance, let $\C = \P^{\antishriek}$ be the dual $\mathbb{V}$-colored cooperad of a Koszul $\mathbb{V}$-colored operad $\P$, see \cite[Definition~2.45]{War19}.

\begin{example}
	We denote by $\mathbb{S}$ the groupoid defined by $\mathbb{S} \coloneqq \bigsqcup_{n \geqslant0} \mathrm{B} \mathbb{S}_n$ where $\mathrm{B} \mathbb{S}_n$ is the category with one object $\{*\}$ and $\mathrm{Aut}(*) = \mathbb{S}_n$~. As proved by Ward in \cite{War19}, there exists an $\mathbb{S}$-colored Koszul operad encoding symmetric operads. 
\end{example}

\noindent The previous sections generalize directly to this setting by working in the following colored version of the convolution dg pre-Lie algebra.

\begin{definition}
[Colored convolution dg pre-Lie algebra]
The \emph{convolution dg pre-Lie algebra} associated to $\C$ and to a $\mathbb{V}$-module $A$ is the dg pre-Lie algebra \[\mathfrak{g}_{A} := \left( \Hom_{\mathbb{S}} \left(\overline{\C}, \mathrm{End}_A\right), \star, d \right) \ , \]
where the underlying space is  \[\Hom_{ \mathcal{S}_{\mathbb{V}}} \left(\overline{\C}, \End_A\right) : = \prod_{ \mathcal{S}_{\mathbb{V}}} \Hom_{\mathbb{S}} \left(\overline{\C}\left(v_0;v_1, \dots , v_r\right), \End_A\left(v_0; v_1, \dots , v_r\right)\right)\ ,\] where $\mathcal{S}_{\mathbb{V}} \coloneqq \lbrace \left(v_1, \dots , v_r; v_0\right) \in \mathrm{ob} \left(\mathbb{V}\right)^{r+1} \rbrace$. It is equipped with the pre-Lie product \[\varphi \star \psi \coloneqq \overline{\C} \xrightarrow{\Delta_{(1)}} \overline{\C} \circ_{(1)} \overline{\C} \xrightarrow{\varphi \circ_{(1)} \psi} \End_A \circ_{(1)} \End_A \xrightarrow{\gamma_{(1)}} \End_A \ ,\] and the differential \[d (\varphi) = d_{\End_A} \circ \varphi - (- 1)^{|\varphi|} \varphi \circ d_{\overline{\C}} \ .\] 
\end{definition}

\begin{assumption}\label{assumptions 2}
	Let $R$ be a commutative ground ring and let $A$ be a chain complex related to $H(A)$ via a contraction. Let $\mathbb{V}$ be a groupoid and let $\C$ be a reduced weight-graded $\mathbb{V}$-colored dg cooperad where $R$ is either 
	\begin{itemize}
		\item[$\centerdot$] any commutative ground ring if $\mathbb{V}$ is a set and $\C$ is a non-symmetric; 
		\item[$\centerdot$] a characteristic zero field if  $\mathbb{V}$ is a groupoid.  
	\end{itemize} 
\end{assumption}

\begin{theorem}[Homotopy transfer theorem]  
Under Assumptions \ref{assumptions 2}, for any $\Cobar \C$-algebra structure, there exists a transferred structure $(H(A), \varphi_t)$. Ihis is independent of the choice of contraction, i.e. any two such
	transferred structures are related by an $\infty$-isotopy. 
\end{theorem}

\begin{proof} The groupoid colored version of the homotopy transfer theorem is given in \cite{War19} for a characteristic zero field. In the case of set colored non-symmetric operads, one can deal with any coefficient ring see e.g. \cite[Theorem~5]{DV15}. 
\end{proof}

\begin{theorem}\label{B}
Let $R$ be a commutative ring. Let $\mathbb{V}$ be a groupoid and let $\C$ be a $\mathbb{V}$-colored reduced weight-graded dg cooperad over $R$. Let $(A, \varphi)$ be a $\Cobar \C$-algebra structure such that there exists a transferred structure $(H(A),\varphi_t)$.
	\begin{enumerate}
		\item Let $n \geqslant 1$ be an integer such that $n !$ is a unit in $R$. The algebra $(A, \varphi)$ is gauge $n $-formal if and only if the truncated class $K^n_{\varphi_t}$ is zero. 
		
		\item If $R$ is a $\mathbb{Q}$-algebra, the algebra $(A, \varphi)$ is gauge formal if and only if the class $K_{\varphi_t}$ is zero.
	\end{enumerate}
\end{theorem}

\begin{proof}
	This follows from the straightforward colored counterpart of Proposition \ref{Kaledin}, by using Definition \ref{definition formality} and Proposition \ref{gauge formality}.
\end{proof}

\begin{remark}[Naturality in the algebra] Let us consider $\infty$-quasi-isomorphic $\Cobar \C$-algebras $(A, \varphi)$ and $(B, \psi)$, admitting transferred structures. This leads to an $\infty$-isomorphism \begin{center}
		\begin{tikzcd}[column sep=normal]
			f : (H(A), \varphi_t)  \ar[r,squiggly]  & 	(A, \varphi)  \ar[r,squiggly] 
			&   (B,\psi)  \ar[r,squiggly] 		 
			& (H(B),\psi_t) \ .
		\end{tikzcd} 
	\end{center} By setting $\psi_t' =  \left(f^{(0)}\right)^{-1} \cdot  \psi_t $, we have an $\infty$-isotopy \begin{center}
	\begin{tikzcd}[column sep=normal]
g : (H(A), \varphi_t)  \ar[r,squiggly, "f"]  		 
		& (H(B),\psi_t) \ar[r,squiggly, "\left(f^{(0)}\right)^{-1}"]  		 
		& (H(A),\psi_t') \ ,
	\end{tikzcd} 
	\end{center}  see Theorem \ref{action}. By Lemma \ref{piquant}, their respective Kaledin classes are related by \[K_{\psi_t'} = \mathrm{Ad}_{\mathfrak{D}(g)} (K_{\varphi_t}) \quad \mbox{and} \quad K_{\psi_t'}^n = \mathrm{Ad}_{\mathfrak{D}(g)}^n (K_{\varphi_t}^n) \ . \]
\end{remark}

\section{\textcolor{bordeau}{Kaledin classes in the properadic setting }} \label{3}

The formality question also arises for algebraic structures involving operations with several inputs but also several outputs such as dg Frobenius bialgebras, dg involutive Lie bialgebras etc. Such structures are encoded by a generalization of operads introduced in \cite{Vallette_2007} called properads. The purpose of this section is to recall the associated deformation theory and apply the methods of Section \ref{section1} in order to study the formality of algebras over a properad. \medskip

\noindent In all this section, $R$ denotes a characteristic zero field. Let $A$ be a chain complex and let $\C$ be a reduced weight-graded dg coproperad over $R$, i.e. \[\C =   \mathcal{I} \oplus \C^{(1)} \oplus \C^{(2)} \oplus  \cdots \oplus \C^{(n)} \oplus \cdots \quad \ , \quad d_{\C} \left(\C^{(k)}\right) \subset \C^{(k - 1)} \] and $\C(0,0) = 0$. Let us denote the coaugmentation coideal $\overline{\C} \coloneqq \bigoplus_{k \geqslant1}  \C^{(k)}$ so that $\C \cong \mathrm{I} \oplus \overline{\C}$~.

\begin{example}
	Among possible choices for such a coproperad $\C$, we have: 
	
	\begin{enumerate}
		\item The bar construction $\C \coloneqq \Bar \P$ of a reduced properad $\P$ in $\mathrm{gr Mod}_R$, see \cite{Vallette_2007}~.

		\item The Koszul dual coproperad $\C := \P^{\antishriek}$, if $\P$ is a homogeneous Koszul properad, see \cite{Vallette_2007}, or an inhomogeneous quadratic Koszul properad, see \cite[Appendix~A]{GCTV12}.
	\end{enumerate}
\end{example}

\subsection{Recollections on properadic homological algebra}

The Koszul duality theory of properads, the associated deformation theory and the homotopy theory of algebras over a properad were respectively developed in \cite{Vallette_2007},  \cite{Merkulov_2009, Merkulov_2009bis} and \cite{PHC}. In this section, we recall the definition of $\Cobar \C$-algebra structures on $A$, as Maurer--Cartan elements of the associated convolution algebra $\mathfrak{g}_{A}$. We refer the reader to \cite{PHC} for more details.

\begin{definition}[dg Lie-admissible algebra]
	A \emph{dg Lie-admissible algebra} is a chain complex $\mathfrak{g}$ equipped with a linear map of degree zero $\star : \mathfrak{g} \otimes \mathfrak{g} \to \mathfrak{g}$ such that $d$ is a derivation and such that the skew-symmetrized bracket induces a dg Lie algebra structure. 
\end{definition}

\begin{proposition}[{\cite[Proposition 11]{Merkulov_2009}}]
	The \emph{convolution dg Lie-admissible algebra} associated to the coproperad $\C$ and the chain complex $A$ is the Lie-admissible algebra \[\mathfrak{g}_A \coloneqq \left(\Hom_{\mathbb{S}}\left(\overline{\C}, \End_A\right), \partial, \star\right) \ , \] where the underlying space is \[\Hom_{\mathbb{S}}\left(\overline{\C}, \End_A \right) \coloneqq \prod_{m,n \geqslant0} \Hom_{\mathbb{S}_m^{op} \times \mathbb{S}_n}\left(\overline{\C}(m,n),\End_A(m,n) \right)\ ,\] equipped with the Lie-admissible product \[\varphi \star \psi \coloneqq \overline{\C} \xrightarrow{\Delta_{(1,1)}} \overline{\C} \underset{(1,1)}{\boxtimes} \overline{\C} \xrightarrow{\varphi \underset{(1,1)}{\boxtimes}  \psi} \End_A \underset{(1,1)}{\boxtimes} \End_A \xrightarrow{\gamma_{(1,1)}} \End_A \ , \] and the differential \[d (\varphi) \coloneqq d_{\End_A} \circ \varphi - (-1)^{|\varphi|} \varphi \circ d_{\overline{\C}} \ .\] 
\end{proposition}

\begin{remark}
 It embeds into the dg Lie-admissible algebra made up of maps from $\C$~, i.e. \[\mathfrak{g}_A \hookrightarrow \mathfrak{a}_A  \coloneqq \left( \Hom_{\mathbb{S}} \left(\C, \mathrm{End}_A\right), \star, d  \right) \ . \] 
\end{remark}

\begin{proposition}[{\cite[Proposition~3.19]{PHC}}]
	The space $\Hom_{\mathbb{S}} \left(\C, \End_A\right)$ is equipped with the following associative and unital product \[\varphi \circledcirc \psi \coloneqq \C \xrightarrow{\Delta} \C \boxtimes \C \xrightarrow{\varphi \boxtimes  \psi} \End_A \boxtimes \End_A \xrightarrow{\gamma} \End_A \ , \] where the unit is given by $1 : \mathcal{I} \mapsto \mathrm{id}_A$ and $\Delta$ and $\gamma$ stand respectively for the decomposition map of the coproperad $\C$ and the 
	composition map of the endomorphism properad $\End_A$~. 
\end{proposition}

\begin{definition}
		A \emph{$\Cobar \C$-algebra structure} on $A$ is a Maurer--Cartan element $\varphi \in \mathrm{MC}(\mathfrak{g}_{A})$.
An \emph{$\infty$-morphism} $\varphi \rightsquigarrow \psi$ between two $\Cobar \C$-algebra structures $\varphi$ and $\psi$ on $A$ is a map $f \in \mathfrak{a}_A$ of degree $0$ satisfying \[f 	\rhd \varphi -  \psi	\lhd f = d(f) \ \ ,\] where the left and right actions are respectively defined by
	\[\begin{array}{ccccccccc}
		\psi 	\lhd f & \coloneqq & \overline{\C} & \xrightarrow{\Delta_{(1,*)}} & \overline{\C} \underset{(1,*)}{\boxtimes} \C &  \xrightarrow{\psi \underset{(1,*)}{\boxtimes} f} &  \End_A \underset{(1,*)}{\boxtimes} \End_A & \longrightarrow & \End_A \ ,  \\
		& & \mathcal{I} & \overset{\cong}{ \longrightarrow} & \mathcal{I} \boxtimes \mathcal{I} &  \xrightarrow{\varphi \; \boxtimes \; f} & \End_A \boxtimes \End_A & \longrightarrow & \End_A \ , \\ 
		& & & & & & & & \\ f 	\rhd \varphi & \coloneqq & \overline{\C} & \xrightarrow{\Delta_{(*,1)}} & \C \underset{(*,1)}{\boxtimes} \overline{\C}  &  \xrightarrow{f \underset{(*,1)}{\boxtimes} \varphi} &  \End_A \underset{(*,1)}{\boxtimes} \End_A & \longrightarrow & \End_A \ ,  \\
		& & \mathcal{I} & \overset{\cong}{ \longrightarrow} & \mathcal{I} \boxtimes \mathcal{I} &  \xrightarrow{f \; \boxtimes \; \varphi} & \End_A \boxtimes \End_A & \longrightarrow & \End_A \ .
	\end{array}\]
\end{definition}

\begin{remark}
	Let $\varphi \in  \mathrm{MC}(\mathfrak{g}_{A})$ and $\psi \in  \mathrm{MC}(\mathfrak{g}_{B})$ be two $\Cobar \C$-algebra structures. In this context, an $\infty$-morphism $\varphi \rightsquigarrow \psi$ can be defined as a morphism  \[f \in \Hom_{\mathbb{S}} \left(\C, \End^A_B\right)\] satisfying a similar identity, see \cite[Proposition~3.16]{PHC} for more details. 
\end{remark}

\noindent Since the coproperad $\C$ is weight-graded, the Lie-admissible algebra $\mathfrak{a}_A$ admits a weight-grading. Every element $f \in \mathfrak{a}_A$ decomposes as \[f = f^{(0)} + f^{(1)} + f^{(2)} + \cdots \] where $f^{(i)}$ is the restriction of $f$ to $\C^{(i)}$~. The first component \[f^{(0)} : \mathcal{I} \to \Hom(A,A)\] is equivalent to the data of an endomorphism $f^{(0)}(1) \in \End_A$ that we simply denote by $f^{(0)}$~. The $f$ is an $\infty$-morphism, this element $f^{(0)}$ is a chain map. The convolution dg Lie-admissible algebra $\mathfrak{g}_A$ has the same underlying weight grading and every $\varphi \in  \mathfrak{g}_A$ decomposes as $\varphi = \varphi^{(1)} + \varphi^{(2)} + \cdots \ .$

\begin{definition} An $\infty$-morphism $f : \varphi \rightsquigarrow \psi$ between two $\Cobar \C$-algebra structures on $A$ is
	\begin{itemize}
		\item[$\centerdot$] an \emph{$\infty$-quasi-isomorphism} if $f^{(0)} : A \to A$ is a quasi-isomorphism; 
		\item[$\centerdot$] an \emph{$\infty$-isomorphism} if $f^{(0)} : A \to A$ is an isomorphism; 
		\item[$\centerdot$] an \emph{$\infty$-isotopy} if $f^{(0)} = \mathrm{id}_A$~. 
	\end{itemize}
	The set of all the $\infty$-isotopies is denoted by $\infty-\mathrm{iso}$~.
\end{definition}

\begin{theorem}[{\cite[Theorem~3.21]{PHC}}]\label{invertibility prop}
	The $\infty$-isomorphisms are the isomorphisms in the category of $\Cobar \C$-algebra structures on $A$~, i.e. every $\infty$-isomorphism $f : \varphi \rightsquigarrow \psi$ admits a unique inverse, denoted $f^{-1}$, with respect to the product $\circledcirc$. 
\end{theorem}

\begin{theorem}[{\cite[Theorem~2.16]{graphexp}}]\label{graphexp2}
	The gauge group associated to the convolution dg Lie admissible algebra $\mathfrak{g}_A$ and the group of $\infty$-isotopies are isomorphic though the \emph{graph exponential/logarithm maps}, \[\mathrm{exp} : ((\mathfrak{g}_A)_0, \mathrm{BCH}, 0) \cong (\infty-\mathrm{iso}, \circledcirc, 1) : \mathrm{log}\ . \] 
\end{theorem}

\subsection{Gauge formality of algebras and properadic Kaledin classes} In this section, we define gauge formality for algebras over a properad and we introduce Kaledin classes characterizing it. Let us consider the properad 
\[\P \coloneqq \mathcal{T}\left(s^{-1}\C^{(1)}\right) / \left(d_{\Cobar \C} \left(s^{-1}\C^{(2)}\right)\right) \ , \] which comes with a twisting morphism $\C \to \P$.

\begin{proposition}
Any $\Cobar \C$-algebra structure $(A, \varphi)$ induces a canonical $\P$-algebra structure \[(H(A),\varphi_*) \ . \]
\end{proposition}

\begin{proof} 
Similar to the one of Proposition \ref{induite}. 
\end{proof}

\begin{definition}
	A $\Cobar \C$-algebra structure $(A, \varphi)$ is
	
	\begin{enumerate}
		\item[$\centerdot$] \emph{formal} if there exists a zig-zag of quasi-isomorphisms of $\Cobar \C$-algebras \[(A, \varphi) \; \overset{\sim}{\longleftarrow} \;  \cdot \;  \overset{\sim}{\longrightarrow} \;  \cdots \;  \overset{\sim}{\longleftarrow} \; \cdot \;  \overset{\sim}{\longrightarrow} \; (H(A), \varphi_*) \ ;\]
		\item[$\centerdot$] \emph{gauge formal} if there exists an $\infty$-quasi-isomorphism \begin{center}
			\begin{tikzcd}[column sep=normal]
				(A,\varphi) \ar[r,squiggly,"\sim"] 
				& (H(A), \varphi_*) \ .
			\end{tikzcd} 
		\end{center}
	\end{enumerate} 
\end{definition}

\begin{theorem}[{\cite[Theorem 1.11]{SPH}}] \label{formal - gauge}
	An $\Cobar \C$-algebra structure $(A, \varphi)$ is formal if and only if it is gauge formal.  
\end{theorem}

\begin{proof}
	As for operads, the existence of a zig-zag of quasi-isomorphisms implies an $\infty$-quasi-isomorphism relies on the invertibility of $\infty$-quasi-isomorphisms, see \cite[Corollary 4.19]{PHC}. Since the rectification method do not admit direct generalization to the context of $\Cobar \C $-algebras over a properad, the converse is nonetheless more subtle.  
\end{proof}

\begin{theorem}[{\cite[Theorem~4.14]{PHC}}]\label{HTTprop} Let $(A, \varphi)$ be an $\Cobar \C$-algebra structure. There exists a transferred structure on the homology, i.e. a $\Cobar \C$-algebra structure $(H(A), \varphi_t)$ extending the induced structure $\varphi_*$ such that the embedding $i : H(A) \to A$ and the projection $p : A \to H(A)$ extend to $\infty$-quasi-isomorphisms. The transferred structure is independent of the choice of contraction in the following sense: any two such transferred structures are related by an $\infty$-isotopy. 
\end{theorem}

\begin{proposition} \label{gauge formality 2} Let $(A, \varphi)$ be $\Cobar \C$-algebra structure and let $(H(A), \varphi_t)$ be a transferred structure. The algebra $(A, \varphi)$ is formal if and only if there exists an $\infty$-isotopy \begin{center}
		\begin{tikzcd}[column sep=normal]
			(H(A), \varphi_t) \ar[r,squiggly,"="] 
			& (H(A), \varphi_*) \ .
		\end{tikzcd} 
	\end{center}
\end{proposition}

\begin{proof}
	By Theorem \ref{formal - gauge}, the algebra is formal if and only if it is gauge formal. The proof is then similar to the one of Proposition \ref{gauge formality}.  
\end{proof}

\begin{definition}
Let $(A, \varphi)$ be $\Cobar \C$-algebra such that there exists a transferred structure $(H(A), \varphi_t)$. It is called \emph{gauge $n$-formal} if there exist an $\Cobar \C$-algebra $(H(A),\psi)$ such that $\psi^{(k)} = 0$~, for all $2 \leqslant k \leqslant n + 1$ and an $\infty$-isotopy 	\begin{center}
	\begin{tikzcd}[column sep=normal]
		(H(A), \varphi_t) \ar[r,squiggly,"="] 
		& (H(A), \psi) \ .
	\end{tikzcd} 
\end{center}
\end{definition}

\begin{definition}[Properadic Kaledin classes]
Let $\varphi \in \mathrm{MC}(\mathfrak{g}_H)$ be an $\Cobar \C$-algebra structure on a graded $R$-module $H$ and let $\mathfrak{D}(\varphi)$ be its prismatic decomposition. The \emph{Kaledin class} $K_{\varphi}$ is the Kaledin class of $\mathfrak{D}(\varphi)$, i.e. \[K_{\varphi} =  \left[\partial_{\hbar}   \mathfrak{D}(\varphi) \right] \in H_{-1}\left(\mathfrak{g}_H[\![\hbar]\!]^{\Phi}\right) \ . \] Its \emph{$n^{\text{th}}$-truncated Kaledin class} $K_{\varphi}^n$ is the $n^{\text{th}}$-truncated Kaledin class of $\mathfrak{D}(\varphi)$, i.e. \[ K^n_{\varphi} = \left[\varphi^{(2)} + 2 \varphi^{(3)} \hbar + \dots + n \varphi^{(n+1)} \hbar^{n-1} \right]  \in H_{-1}\left(\left(\mathfrak{g}_H[\![\hbar]\!]/(\hbar^{n })\right)^{\overline{\Phi}^n}\right) \ .  \] 
\end{definition}

\begin{theorem}\label{Kaledin properadic}
Let $R$ be a characteristic zero field. Let $\C$ be a reduced weight-graded differential graded coproperad over $R$. Let $(A, \varphi)$ be a $\Cobar \C$-algebra structure and let $(H(A), \varphi_t)$ be a transferred structure.
	\begin{enumerate}
		\item  The algebra $(A, \varphi)$ is gauge $n $-formal if and only if the truncated class $K^n_{\varphi_t}$ is zero. 
		
		\item  The algebra $(A, \varphi)$ is formal if and only if the Kaledin class $K_{\varphi_t}$ is zero.
	\end{enumerate}
\end{theorem}

\begin{proof} Point (1) is an immediate corollary of Proposition \ref{passage}, by Proposition \ref{gauge formality 2} and the isomorphism of Theorem \ref{graphexp2}. Similarly, Point (2) follows from Point (1) of Theorem \ref{A}. 
\end{proof}

\begin{remark}
	If $R$ is not a characteristic zero field, the properadic Kaledin classes also adapts using the same arguments as in Section \ref{section2}. One can also prefer to work
	with the obstruction sequences developed in \cite[Section~2]{CE24b}, without a bound on the characteristic.
\end{remark}

\subsection{Properadic calculus}
This section generalizes some lemmas established in Section \ref{2.3} to the properadic case. These results will be used in Section \ref{critere} to prove the vanishing of properadic Kaledin classes. 

\begin{lemma}\label{compatibility prop}
	Let $g \in \mathfrak{a}_H$ be an element concentrated in weight $0$ and invertible with respect to $\circledcirc$. Let $f$ and $h$ in $\mathfrak{a}_A$ and let $(f;h)$ be the standard notation of \cite{PHC} for $f$ applied everywhere except at one place where $h$ is applied. The following identities holds 
	\begin{enumerate}
		\item[(a)] $\left(f \lhd g \right) \star h = f \lhd \left( g ; g \star h \right) $;
		\item[(b)] $f \star h = f \lhd (1 ; h) $; 
		\item[(c)] $\left(g ; f \lhd  g \right) \rhd h = f \lhd \left(g ; g  \rhd h\right)  $;
		\end{enumerate}
		If furthermore, the element $f^{(0)}$ is invertible with respect to $\circledcirc$ , then
	\begin{enumerate}
		
		\item[(d)] $ \left( f \lhd \left( g ;  h \right) \right) \lhd g^{-1} = f \star \left(h \lhd g^{-1}\right)$;
		\item[(e)] $ \left(f \lhd g  \right) \lhd  g^{-1} = f $;
		\item[(f)] $\left( \left(g ; f\right) \rhd h \right) \lhd g^{-1} = \left(f \lhd g^{-1}\right) \star \left( \left(g \rhd h \right) \lhd g^{-1}\right) $ \ .
	\end{enumerate}
\end{lemma}

\begin{proof}
	The proof is similar to the one of Lemma \ref{compatibility} and relies on the fact that $g$ is zero outside of $\C^{(0)} = \mathcal{I}$. Formula (a) follows from the isomorphism \[\left( \overline{\C} \underset{(1,*)}{\boxtimes} \mathcal{I} \right)  \underset{(1,1)}{\boxtimes} \overline{\C} \cong \overline{\C} \underset{(1,*)}{\boxtimes} \left( \mathcal{I} ; \mathcal{I} \underset{(1,1)}{\boxtimes} \overline{\C} \right)  \ . \]  Formula (b) follows from (a) with $g = 1$. Formula (c) boils down to the isomorphism \[\left( \mathcal{I};  \overline{\C} \underset{(1,*)}{\boxtimes} \mathcal{I} \right)  \underset{(*,1)}{\boxtimes} \overline{\C} \cong \overline{\C} \underset{(1,*)}{\boxtimes} \left( \mathcal{I} ; \mathcal{I} \underset{(*,1)}{\boxtimes} \overline{\C} \right)  \ . \] We have that $ \left( f \lhd \left( g ;  h \right) \right) \lhd g^{-1} =  f \lhd \left( g \lhd g^{-1} ;  h \lhd g^{-1} \right)   $ by the isomorphism \[ \left( \overline{\C} \underset{(1,*)}{\boxtimes} \left(\mathcal{I};  \overline{\C} \right)  \right)   \underset{(1,*)}{\boxtimes} \mathcal{I} \cong  \overline{\C} \underset{(1,*)}{\boxtimes}
	 \left(\mathcal{I}  \underset{(1,*)}{\boxtimes}  \mathcal{I}  ;  \overline{\C} \underset{(1,*)}{\boxtimes}  \mathcal{I}   \right)  	 \ . \] Using the fact that $g \lhd g^{-1} =g \circledcirc g^{-1} = 1$, Formula (d) then follows from (b). Formula (e) is a particular case of (d) with $h = g$. Using Formulas (e) and (c), we have \[\left(g ; f\right) \rhd h = \left(g ; \left(f \lhd g^{-1}  \right) \lhd  g \right) \rhd h = \left(f \lhd g^{-1}  \right) \lhd \left(g ; g  \rhd h\right) \ . \] Formula (f) then follows from Formula (d) since	 
	 \[\left( \left(g ; f\right) \rhd h \right) \lhd g^{-1} = \left( \left(f \lhd g^{-1}  \right) \lhd \left(g ; g  \rhd h\right) \right) \lhd g^{-1} \overset{(d)}{=} \left(f \lhd g^{-1}\right) \star \left( \left(g \rhd h \right) \lhd g^{-1}\right)  \ . \qedhere \]
\end{proof}

\begin{lemma}\label{diff prop}
For all $f$ and $g$ in $\mathfrak{a}_A$ of degree $0$, the following identities hold   
	\begin{enumerate}
		\item[(i)] $(f \rhd g) \circ d_{\C} = f \rhd (g \circ d_{\C}) + (f; f \circ d_{\C}) \rhd g$;
		\item[(j)] $\left(f \lhd g \right) \circ d_{\C} =  \left( f \circ d_{\C} \right) \lhd g +  f  \lhd \left(g; g \circ d_{\C}   \right)  \ . $
	\end{enumerate}
If furthermore, the element $f^{(0)}$ is invertible with respect to $\circledcirc$ , then for all $x \in \mathfrak{a}_H$
	\begin{enumerate}
		\item[(k)] $  \left(f^{(0)} \right)^{-1} \circ d_{\C} = -  \left(f^{(0)} \right)^{-1} \star  \left( \left( f^{(0)} \circ d_{\C} \right)  \lhd  \left(f^{(0)}\right) ^{-1} \right) \ .  $ 
	\end{enumerate}
\end{lemma}

\begin{proof}
	Points (i) and (j) follows from the fact that $\C$ is a dg coproperad. By Formula (j) and Formula (e) of Lemma \ref{compatibility prop},  we have \[ \left(f^{(0)} \right)^{-1} \circ d_{\C}  = - \left( \left(f^{(0)} \right)^{-1}  \lhd \left(f^{(0)}; f^{(0)} \circ d_{\C}   \right)\right) \lhd \left(f^{(0)} \right)^{-1} \ . \] Point (k) then follows from Point (d) of Lemma \ref{compatibility prop}. 
\end{proof}

\begin{lemma}\label{key lemma 2}
	Let $n \geqslant1$ be an integer and let $f : \varphi \rightsquigarrow \psi $ be an $\infty$-isomorphism such that	
	\begin{itemize}
		\item[$\centerdot$] $\varphi^{(k)} = 0$~, for $2 \leqslant k \leqslant n$ and
		\item[$\centerdot$]  $f^{(k)} = 0$~, for $1 \leqslant k \leqslant n - 1$ . 
	\end{itemize}
	Then, the weight components of $\psi$ are given by 
	\begin{itemize}
		\item[$\centerdot$]  $\psi^{(1)} =  \mathrm{Ad}_{f^{(0)}} \left( \varphi^{(1)} \right) + \left(f^{(0)} \circ d_{\C}\right) \lhd \left(f^{(0)}\right)^{-1}  $ ,
		\item[$\centerdot$] $\psi^{(k)} = 0$~,  for $2 \leqslant k \leqslant n$~, and
		\item[$\centerdot$]  $\psi^{(n+1)} = \mathrm{Ad}_{f^{(0)}} \left(\varphi^{(n+1)} \right) - d^{\psi^{(1)}} \left(f^{(n)} \lhd \left(f^{(0)}\right)^{-1}\right)$ ,	
	\end{itemize}
	where $\mathrm{Ad}_{f^{(0)}}(x) \coloneqq \left(f^{(0)} \rhd x \right) \lhd \left(f^{(0)} \right)^{-1}$. 
	In particular, if $f$ is an $\infty$-isotopy, we have
	\[ \psi =   \varphi^{(1)} + \varphi^{(n+1)} - d^{ \varphi^{(1)}} \left( f^{(n)} \right) + \psi^{(n+2)} + \cdots  \ .\]
\end{lemma}

\begin{proof}
 By definition of an $\infty$-morphism, we have \begin{equation}\label{infini2}
 \psi \lhd f = f \rhd \varphi + f \circ d_{\C} \ .
 \end{equation} The assumptions on $\varphi$ and on $f$ imply that \[\begin{split}
f \rhd \varphi + f \circ d_{\C}\  \equiv \ &  f^{(0)} \rhd \varphi^{(1)} + f^{(0)} \rhd \varphi^{(n+1)} + \left(f^{(0)}; f^{(n)} \right) \rhd \varphi^{(1)}  \\ & +  f^{(0)} \circ d_{\C} + f^{(n)} \circ d_{\C} \pmod{\mathcal{F}^{n+2} \mathfrak{g}_H} \ .
 \end{split}\]   Similarly, we have \[\psi \lhd f \equiv \psi \lhd f^{(0)} + \psi^{(1)} \lhd \left(f^{(0)} ; f^{(n)} \right) \pmod{\mathcal{F}^{n+2} \mathfrak{g}_H} \ . \]  Equation (\ref{infini2}) implies that \[\begin{split}
 	\psi \  \equiv \  & \mathrm{Ad}_{f^{(0)}} \left( \varphi^{(1)} \right) +  \left(\left(f^{(0)}; f^{(n)} \right) \rhd \varphi^{(1)}\right) \lhd  \left(f^{(0)}\right) ^{-1}+ \left(f^{(0)} \circ d_{\C} + f^{(n)} \circ d_{\C}  \right) \lhd  \left(f^{(0)}\right) ^{-1}  \\ &  +  \mathrm{Ad}_{f^{(0)}} \left( \varphi^{(n+1)} \right)   - \left(\psi^{(1)} \lhd \left(f^{(0)} ; f^{(n)} \right) \right) \lhd  \left(f^{(0)}\right) ^{-1}    
 	\pmod{\mathcal{F}^{n+2} \mathfrak{g}_H} \ .
 \end{split} \] This gives the desired formulas for the components $\psi^{(k)}$ for $1 \leqslant k \leqslant n$. Furthermore, by Point (d) of Lemma \ref{compatibility prop}, we have \[ \left(\psi^{(1)} \lhd \left(f^{(0)} ; f^{(n)} \right) \right) \lhd  \left(f^{(0)}\right) ^{-1}   =  \psi^{(1)} \star  \left(f^{(n)} \lhd \left(f^{(0)}\right) ^{-1} \right) \ . \] By Point (f) of Lemma \ref{compatibility prop}, we have\[\left(\left(f^{(0)}; f^{(n)} \right) \rhd \varphi^{(1)}\right) \lhd  \left(f^{(0)}\right) ^{-1} = \left(f^{(n)} \lhd \left(f^{(0)}\right) ^{-1} \right) \star \mathrm{Ad}_{f^{(0)}} \left( \varphi^{(1)} \right) \ . \] Point (j) of Lemma \ref{diff prop} leads to \[\left(f^{(n)} \circ d_{\C} \right) \lhd  \left(f^{(0)}\right) ^{-1}  = \left(f^{(n)} \lhd \left(f^{(0)}\right) ^{-1} \right) \circ d_{\C} -  f^{(n)} \lhd \left(\left(f^{(0)}\right) ^{-1}; \left(f^{(0)}\right) ^{-1} \circ d_{\C}   \right)  \] while Point (k) of Lemma \ref{diff prop} and Point (a) of Lemma \ref{compatibility prop} imply that \[ - f^{(n)} \lhd \left(\left(f^{(0)}\right) ^{-1}; \left(f^{(0)}\right) ^{-1} \circ d_{\C}   \right)  =  \left(f^{(n)} \lhd \left(f^{(0)}\right) ^{-1} \right) \star \left( \left( f^{(0)} \circ d_{\C} \right)  \lhd  \left(f^{(0)}\right) ^{-1}  \right) \ . \] We obtain that \[ \psi^{(n+1)} = \mathrm{Ad}_{f^{(0)}} \left( \varphi^{(n+1)} \right) - \left[ \psi^{(1)} ,   f^{(n)}  \lhd \left(f^{(0)}\right)^{-1}   \right] +  \left(f^{(n)} \lhd \left(f^{(0)}\right)^{-1} \right) \circ d_{\C} \ .  \]
\end{proof}

\section{\textcolor{bordeau}{Formality criteria}} \label{4}

This section exploits the obstruction theory set out above in order to produce explicit formality criteria. 
We establish formality descent results in Section \ref{4:1.4} and study formality of families in Section \ref{4:3.4}. In Section \ref{4:2.4}, we establish a criterion of intrinsic formality. Section \ref{critere} concludes with criteria in terms of chain level lifts of certain homology automorphisms. This generalizes the previously known results to any coefficient ring and to algebras over a colored operad or a properad. 

\begin{assumption}\label{assump}
	\noindent Let $A$ be a chain complex over a commutative ring $R$. Let $n \geqslant 1$ be an integer such that $n !$ is a unit in $R$. Let $(A, \varphi)$ be a $\Cobar \C$-algebra structure where $\C$ is either 
	\begin{itemize}
		\item[$\centerdot$] a reduced weight-graded dg cooperad over $R$;

		\item[$\centerdot$]  a $\mathbb{V}$-colored reduced weight-graded dg cooperad over $R$, where $\mathbb{V}$ is a given groupoid;
		
		\item[$\centerdot$]  a reduced weight-graded dg coproperad over a characteristic zero field.	
	\end{itemize}
	Suppose that there exists a transferred structure $(H(A), \varphi_t)$, e.g. Assumptions \ref{assumptions 1} are satisfied in the case of a cooperad or Assumptions \ref{assumptions 2} are satisfied in the colored case.
\end{assumption}

\noindent We denote by $\mathfrak{g} \coloneqq \mathfrak{g}_{H(A)}$ the convolution dg Lie algebra associated to $\C$ and $H(A)$ and by $\tilde{\mathfrak{g}} $ is the subcomplex of $\mathfrak{g}$ whose underlying space is given by \[\sum_{k \geqslant 1} \mathfrak{g}^{(k)} \ .\]

\noindent The constraint on the integer $n$ in positive characteristic and the requirement of $R$ being a characteristic zero field for properads come from the use of Kaledin classes to establish the following criteria. In the forthcoming paper \cite{CE24b}, we refine the Kaledin classes construction in order to remove these assumptions.

\subsection{Formality descent}\label{4:1.4}

Let $R \to S$ be any commutative ring morphism. If the $\Cobar \C$-algebra $(A, \varphi)$ is formal, then the algebra $(A \otimes_{R} S,\varphi \otimes 1)$ stays formal. This section gives sufficient conditions under which the converse holds true.

\begin{theorem}[Formality descent]\label{formality descent}
Under Assumptions \ref{assump}, let $S$ be a faithfully flat commutative $R$-algebra. Suppose that $H(A)$ is an $R$-module of finite presentation. Let us denote \[A_{S} \coloneqq A \otimes_{R} S \ .\] 
	\begin{enumerate}
		\item The $\Cobar \C$-algebra $(A, \varphi)$ is gauge $n$-formal if and only if $(A_{S},\varphi \otimes 1)$ is gauge $n$-formal.
		\item If $R$ is a $\mathbb{Q}$-algebra, the $\Cobar \C$-algebra $(A, \varphi)$ is gauge formal if and only if $(A_{S},\varphi \otimes 1)$ is gauge formal.
	\end{enumerate}
\end{theorem}

\begin{proof}
In the properadic setting, this is a direct corollary of Point (3) of Theorem \ref{A}, using in Theorem \ref{Kaledin properadic}. Let us prove the result in the operadic case. For both points, the direct implications are immediate. Conversely, suppose that $(A_{S},\varphi \otimes 1)$ is gauge $n$-formal. By Theorem \ref{B}, the $k^{\text{th}}$-truncated Kaledin class of $\varphi_t \otimes 1$ are zero for all $k \leqslant n$. Let us prove by induction that $K^n_{\varphi_t}$ is also zero. The canonical map \[
H_{-1}\left(\tilde{\mathfrak{g}}  ^{\varphi_*}\right) \hookrightarrow H_{-1}\left(\mathfrak{g} ^{\varphi_*}\right) \]
is injective and its image contains the class $K^1_{\varphi_t }$. The same applies for $\tilde{\mathfrak{g}}_{H(A_S)}$ and $K^1_{\varphi_t \otimes 1 }$. Since $S$ is flat, we have an isomorphism. $H(A_S) \cong H(A) \otimes S \ . $  Since the homology $H(A)$ is of finite presentation, it induces an isomorphism, \[\tilde{\mathfrak{g}}_{H(A_S)} \cong  \tilde{\mathfrak{g}} \otimes S\ , \] which leads to an homology isomorphism \begin{equation}\label{iso descente}
H_{-1} \left(\tilde{\mathfrak{g}}_{H(A_S)}^{ \varphi_* \otimes 1  }  \right) \cong H_{-1} \left( \tilde{\mathfrak{g}}^{ \varphi_*   }  \right) \otimes S \ .
\end{equation} By faithful flatness, since $K^1_{\varphi_t \otimes 1}$ is zero on the left hand-side, so does $K_{\varphi_t}^1$. Suppose that \[K^{k-1}_{\varphi_t}=0 \quad \mbox{for}\quad  2\leqslant k \leqslant n \ .\] By Theorem \ref{B}, there exists an $\Cobar \C$-algebra $(H(A),\psi)$ such that \[\psi^{(1)} = \varphi_* \quad \mbox{and} \quad \psi^{(i)} = 0 \quad \mbox{for all } 2 \leqslant i \leqslant k\] and an $\infty$-isotopy $f : \varphi_t \rightsquigarrow \psi$.  Point (2) of Lemma \ref{piquant} implies that \[K_{\psi}^k = \mathrm{Ad}_{\mathfrak{D}(f)}^k \left(K_{\varphi_t}^k\right) \quad \mbox{and} \quad  K_{\psi \otimes 1}^k = \mathrm{Ad}_{\mathfrak{D}(f \otimes 1)}^k \left(K_{\varphi_t \otimes 1}^k\right) =0\ . \] By construction, we have \[K_{\psi}^k = \left[k \psi^{(k+1)} \hbar^{k-1} \right]  \in H_{-1}\left(\tilde{\mathfrak{g}}^{\varphi_*}[\![\hbar]\!] /(\hbar^{k})\right)  \quad \mbox{and} \quad K_{\psi \otimes 1}^k = 0  \in  H_{-1} \left( \tilde{\mathfrak{g}}^{ \varphi_*   } [\![\hbar]\!] /(\hbar^{k}) \right) \otimes S  \ .   \] The isomorphism (\ref{iso descente}) again implies that $K_{\psi}^k$ is zero, so does $K_{\varphi_t}^k$. By induction, the $n^{\text{th}}$-truncated Kaledin class is zero. Finally, the $\Cobar \C$-algebra $(A, \varphi)$ is gauge $n$-formal by Theorem \ref{B}. Suppose now that $R$ is a $\mathbb{Q}$-algebra and $(A_{S},\varphi \otimes 1)$ is gauge formal. The same induction can be performed higher-up and proves that $K^n_{\varphi_t}$ are zero for all $n \geqslant 1$. Thus, the Kaledin class $K_{\varphi_t}$ is zero by Proposition \ref{toutes nulles 2} and the $(A, \varphi)$ is also gauge formal by Theorem \ref{B}. 
\end{proof}

\noindent For example, suppose that $(R, \mathfrak{m})$ is a Noetherian local ring and let $\mathfrak{I} \subset \mathfrak{m}$ be an ideal. We denote $\hat{R}$ the completion of $R$ with respect to $\mathfrak{I}$. 
Then the ring morphism $R \to \hat{R}$ is faithfully flat and the result follows from Theorem \ref{formality descent}. Let $p$ be a prime number. Let us consider the commutative ring $\mathbb{Z}_{p}$ of $p$-adic integers. It is the completion of the local ring \[ R = \mathbb {Z} _{(p)}=\{{\tfrac {n}{d}}\mid n,d\in \mathbb {Z} ,\,d\not \in p\mathbb {Z} \},\] with respect to $\mathfrak{m} = \mathfrak{I} = (p)$. In order to establish for the first time formality results with coefficients in $\mathbb {Z} _{(p)}$, we apply Theorem \ref{formality descent} as follows.

\begin{theorem}\label{new}
Let $p$ be a prime number. Let $X$ be a complement of a hyperplane arrangement over $\mathbb{C}$. 
	Let $\mathbb{Q}_{p} \hookrightarrow K$ be a finite extension. We denote by $q$ the order of the residue field of the ring of integers of $K$. Suppose that $X$ is defined over $K$, i.e. there exists an embedding $K \hookrightarrow \mathbb{C}$ and there exists $\mathcal{X}$ a complement of a hyperplane arrangement over $K$ such that \[\mathcal{X} \times _{K} \mathbb{C} \cong X \ .\]  Let $\ell$ be a prime different from $p$. The algebra of singular cochains $C(X_{\mathrm{an}}, \mathbb{Z}_{(\ell)})$ is gauge $(s-1)$-formal, where $X_{\mathrm{an}}$ is the complex analytic space underlying $X_{\mathbb{C}} = X \times _{K} \mathbb{C}$ and where $s$ is the order of $q$ in $\mathbb{F}_{\ell}^{\times}$~.
\end{theorem}

\begin{proof}
Drummond-Cole and Horel proved in \cite[Theorem 4.7]{DCH21} that the algebra of singular cochains $C(X_{\mathrm{an}}, \mathbb{Z}_{\ell})$ is gauge $(s-1)$-formal, see Point (4) of Example \ref{arrangements} for more details. The results then follows directly from Theorem \ref{formality descent}. 
\end{proof}

\subsection{Formality in families}\label{4:3.4}

Kaledin was motivated by establishing formality criteria, for which the earlier methods do not apply. The work of Kaledin was extended by Lunts in \cite{Lun07} who established Theorem \ref{families0} and Theorem \ref{families} in the context of any $A_{\infty}$-algebras over a ring of characteristic zero. Using the present Kaledin classes, we obtain a straightforward generalization of these results for any algebra over a colored operad or properad. We refer the reader to \cite{CE24b} for applications. 

\begin{theorem}\label{families0}
	Under Assumptions \ref{assump}, suppose that $R$ is an integral domain, that $H(A)$ is an $R$-module of finite presentation and that the following $R$-module is torsion free: \[H_{-1}\left(\tilde{\mathfrak{g}}^{\varphi_*}\right) \ .\] Let $\eta \in \mathrm{Spec} R$ be a generic point. The $\Cobar \C$-algebra $(A , \varphi)$ is gauge $n$-formal if and only if \[A_{\eta} = (A \otimes \kappa(\eta) , \varphi \otimes 1) \] is gauge $n$-formal. In particular, $(A , \varphi)$ is gauge $n$-formal if and only if $A_x$ is gauge $n$-formal, for all points $x \in \mathrm{Spec} R$. If $R$ is a $\mathbb{Q}$-algebra, the same result holds for gauge formality.
\end{theorem}

\begin{proof}
	The proof is the same as the one of Theorem \ref{formality descent}, replacing $S$ by the residue field $\kappa(\eta)$, the faithful flatness by the fact $H_{-1}\left(\tilde{\mathfrak{g}}^{\varphi_*}\right)$ is torsion free. 
\end{proof}

\begin{theorem}\label{families}
	Under Assumptions \ref{assump}, suppose that $R$ is a noetherian ring, that $H(A)$ is an $R$-module of finite presentation and that the following $R$-module is projective: \[H_{-1}\left(\tilde{\mathfrak{g}}^{\varphi_*}\right) \ .\] 	\begin{enumerate} 
		\item The following subset is closed under specialization: \[F\left(A\right)\coloneqq \lbrace x \in \mathrm{Spec} R \mid  A_x = (A, \varphi) \otimes_{R} \kappa(x) \mbox{ is gauge $n$-formal }\rbrace . \] 
		\item Let  $I \subset R$ be an ideal such that $\cap_k I^k=0$~. The $\Cobar \C$-algebra $(A , \varphi)$ is gauge $n$-formal if and only if \[A/I^k = A \otimes_R R/I^k  \] is $n$-gauge formal, for all $k \geqslant1$~. 
		\item Assume that $R$ has trivial radical. The $\Cobar \C$-algebra $(A , \varphi)$ is gauge $n$-formal if and only if $A_x$ is gauge $n$-formal, for all closed point $x \in \mathrm{Spec}R$.
	\end{enumerate}
	If $R$ is a $\mathbb{Q}$-algebra, the same result holds for gauge formality.
\end{theorem}

\begin{proof}
	Let us prove Point (1). If $F\left(A\right)$ is empty the result is immediate. Otherwise, let us fix $\eta \in F\left(A\right)$. We denote by $\bar{\eta} = \mathrm{Spec} \bar{R}$ its closure. The ring $\bar{R}$ is an integral domain and the following $\bar{R}$-module is torsion free  \[H_{-1} \left(\tilde{\mathfrak{g}}_{H(A_{\bar{R}})}^{ \varphi_* \otimes 1  }  \right) \cong H_{-1} \left( \tilde{\mathfrak{g}}^{ \varphi_*   }  \right) \otimes \bar{R}\ .  \]  By Theorem \ref{families0}, the $\Cobar \C$-algebra $(A_{\bar{R}}, \varphi \otimes 1)$ is gauge $n$-formal, so does $A_x$ for all $ x \in \mathrm{Spec} \bar{R}$. 
	
	\noindent The proof of Point (2) is similar to the one of Theorem \ref{formality descent}, replacing $S$ by $R/I^k$  and the faithful flatness by the facts that $\cap_k I^k=0$ and $ H_{-1}\left(\tilde{\mathfrak{g}}^{\varphi_*}\right)$ is projective. For Point (3), let $\mathfrak{m} \subset R$ be a maximal ideal. The proof is again the same as the one of Theorem \ref{formality descent}, replacing $S$ by $R/\mathfrak{m}$ and the faithful flatness by the facts that the radical of $\mathfrak{m}$ is trivial and the projectivity of $ H_{-1}\left(\tilde{\mathfrak{g}}^{\varphi_*}\right)$.
\end{proof}

\subsection{Intrinsic formality}\label{4:2.4}

\begin{definition}[Intrinsic formality]
	Let $H$ be a graded $R$-module and let $(H, \varphi_*)$ be a $\Omega \C$-algebra structure concentrated in weight $1$. It is said \emph{intrinsically formal} if for any $\Cobar \C$-algebra structure $(H,\psi)$ such that $\psi^{(1)} = \varphi_*$, there exists an $\infty$-isotopy \begin{center}
		\begin{tikzcd}[column sep=normal]
			(H, \psi) \ar[r,squiggly,"\sim"] 
			& (H, \varphi_*) \ .
		\end{tikzcd} 
	\end{center} 
\end{definition}

\begin{theorem}\label{intrinsec} 
Under Assumptions \ref{assump}, let $H$ be a graded $R$-module and let $(H, \varphi_*)$ be a $\Omega \C$-algebra structure concentrated in weight $1$. Let $(H,\psi)$ be any $\Cobar \C$-algebra such that $\psi^{(1)} = \varphi_*$.

\begin{enumerate}
	\item If $H_{-1}\left(\mathfrak{g}_H^{\varphi_*}/\mathcal{F}^n\mathfrak{g}_H\right)=0$, then $(H,\psi)$ is gauge $n$-formal. 
	
	\item If $R$ is a $\mathbb{Q}$-algebra and if $H_{-1}(\mathfrak{g}^{\varphi_*}_H)=0$, then $(H,\varphi_*)$ is intrinsically formal.  
\end{enumerate}

\end{theorem}

\begin{proof}
	Point (2) is an immediate corollary of Point (2) of Theorem \ref{A}. The proof of Point (1) makes use of the Kaledin classes and Proposition \ref{Kaledin} in a way similar to the proofs of Theorem \ref{intrinsec0}. 
\end{proof}

\begin{example}[Tamarkin's proof of Kontsevich formality theorem]
	Theorem \ref{intrinsec} generalizes Theorem~4.1.3 of \cite{Hinich03bis}, established for a characteristic zero field. In his paper, Hinich presents Tamarkin's proof of Kontsevich formality theorem that can be stated as follows. For every polynomial algebra $A$ over a field of characteristic zero, the shifted Hochschild cochain complex $C(A; A)[1]$ is formal as a dg Lie algebra. \smallskip
	
	\noindent Recall that the Hochschild cohomology $H \coloneqq HH^{\bullet}(A)$ admits a natural Gerstenhaber algebra structure denoted $ \varphi_*$, after \cite{Gerst}. Let us denote by $Gerst$ the Koszul operad encoding the Gerstenhaber algebras. 
	Tamarkin's proof relies on the following two facts, 
	\begin{itemize}
		\item[$\centerdot$] the existence of a $Gerst_{\infty}$-algebra structure on the Hochschild cochain complex of an associative algebra, which induces the $Gerst$-algebra structure on $HH^{\bullet}(A)$~, see \cite{Tamarkin98}; 
		\item[$\centerdot$] the intrinsic formality as a $Gerst$-algebra of the Hochschild cohomology $HH^{\bullet}(A)$~. 	
	\end{itemize}
	Using exactly the same strategy as in \cite[Section~5.4]{Hinich03bis}, the last point we can proved by checking that $H_{-1}(\mathfrak{g}^{\varphi_*})=0$ where \[\mathfrak{g} = \Hom_{\mathbb{S}}\left(\overline{Gerst}^{\antishriek}, \End_H\right) \ .\] 
\end{example}

\begin{example}[The intrinsic formality of $E_n$-operads]
	In \cite{FW20}, Fresse and Willwacher establish an intrinsic formality result for $E_n$-operads. More precisely, let $Pois_n^c$ be the dual cooperad of the $n$-Poisson operad $Pois_n$.  When $n \geqslant 3$, the authors establish that a Hopf cooperad $\mathcal{K}$ is related to $Pois_n^c$ by a zig-zag of quasi-isomorphisms of Hopf cooperads as soon as there is an isomorphism \[H(\mathcal{K}) \cong Pois_n^c \ .\] When $4$ divides $n$, an extra assumption that $\mathcal{K}$ is equipped with an involutive isomorphism mimicking the action of a hyperplane reflection is also required. Using Bousfield obstruction theory, they prove that the result boils down to the vanishing of the cohomology group of degree $1$ of a certain deformation complex. This deformation complex is very likely to be isomorphic to the twisted convolution dg Lie algebra $\mathfrak{g}^{\varphi_*}$ where, \[\mathfrak{g} \coloneqq \mathrm{Hom}_{\mathbb{S}}(\overline{\mathcal{O}}^{\antishriek},\mathrm{End}_{Pois_n}) \ , \] where $\mathcal{O}$ is the colored properad encoding Hopf cooperad. One would conclude with a colored properadic version of Theorem \ref{intrinsec}. 
\end{example}

\subsection{Chain lifts of homology automorphisms} \label{critere}

Motivated by Weil’s conjectures insights, there has been a long tradition of using purity in order to prove formality. This goes back to the introduction of \cite{DGMS75} and the following observation. Let $\alpha$ be a unit of infinite order in $R$ and let $\sigma_{\alpha}$ be the grading automorphism defined as the multiplication by $\alpha^k$ in homological degree $k$. If the grading automorphism admits a chain level lift, this prevents the existence of non-trivial higher structures and thus implies formality, see \cite{Sul77}, \cite{GNPR05}, \cite{Petersen14}, \cite{Dup16}, \cite{CH20}, \cite{DCH21} and \cite{CH22}. In this section, we use Kaledin classes to generalize this result to properads and answer the following questions:
\begin{itemize}
	\item[$\centerdot$]  Is grading automorphism the only homology automorphism satisfying this property?
	\item[$\centerdot$]  Under which condition on an homology automorphism, the existence of a lift implies formality? 
\end{itemize}

\begin{definition}
	Under Assumptions \ref{assump}, let $u \in \mathrm{Aut}\left(H(A), \varphi_* \right)$ be a strict automorphism of the homology. We see it as an element of $\mathfrak{a}_{H(A)}$ through the map \[u : \mathcal{I} \to \End_{H(A)} \ , \quad 1 \mapsto u  \ .\] 
	\begin{enumerate}
		\item We say that it admits a \emph{chain level lift} if there exists an $\infty$-quasi-isomorphism $v$ of $(A,\varphi)$ such that $H\left(v^{(0)}\right) = u$.
		\item 	 We denote by $\mathrm{Ad}_u$ the dg Lie algebra automorphism of $\mathfrak{g}^{\varphi_*}$ defined by \[\mathrm{Ad}_u(x) = (u \circ x) \circledcirc u^{-1} \quad \left(\mbox{resp.} \; \; \mathrm{Ad}_{u}(x) = \left(u \rhd x \right) \lhd u^{-1} \; \right) \] in the operadic case (resp. properadic case).
	\end{enumerate}
\end{definition}

\begin{theorem}\label{C}
	Under Assumptions \ref{assump}, the following propositions hold true. Let $u \in \mathrm{Aut}\left(H(A), \varphi_* \right)$ be an automorphism that admits a chain level lift. 
	
	\begin{enumerate}
		\item  If $\mathrm{Ad}_u - \mathrm{id} $ is invertible modulo $\mathcal{F}^{n+2} \mathfrak{g}$, the $\Cobar \C$-algebra $(A, \varphi)$ is gauge $n$-formal.
		
		\item If $R$ is a $\mathbb{Q}$-algebra and $\mathrm{Ad}_u - \mathrm{id} $ is invertible, then
		\begin{itemize}
			\item[$\centerdot$] the $\Cobar \C$-algebra structure $(A, \varphi)$ is gauge formal; 
			\item[$\centerdot$] every automorphism of $\left(H(A), \varphi_* \right)$ admits a chain level lift.
		\end{itemize}
	\end{enumerate}
\end{theorem}

\noindent The proof of Theorem \ref{C} relies on the following lemma. 

\begin{lemma}\label{induction}
Under Assumptions \ref{assump}, let $(H(A), \phi) $ be an $\Cobar \C$-algebra structure such that \[\phi \equiv \varphi_* \pmod{\mathcal{F}^{k+1} \mathfrak{g}} \ ,\] where $k \leqslant n$ is an integer. Let $s$ be an $\infty$-automorphism of $(H(A),\phi)$ such that \[s^{(0)} = u \quad  \mbox{and} \quad s^{(i)} = 0 \quad   \mbox{for all} \ 1 \leqslant i \leqslant k-1 \ .\] Let $u \in \mathrm{Aut}\left(H(A), \varphi_* \right) $ that admits a chain level lift. If $\mathrm{Ad}_u - \mathrm{id} $ is invertible modulo $\mathcal{F}^{k+2} \mathfrak{g}$,
\begin{enumerate}
	\item there exists an $\infty$-isotopy $\phi \rightsquigarrow \psi $ where $\psi \in \mathrm{MC}(\mathfrak{g})$ is $\Cobar \C$-algebra structure such that \[\psi \equiv \varphi_* \pmod{\mathcal{F}^{k+2} \mathfrak{g}} \ , \]
	\item there exists an $\infty$-automorphism $\tilde{s}$ of $(H, \psi)$ such that \[\tilde{s}^{(0)} = u \quad  \mbox{and} \quad \tilde{s}^{(i)} = 0 \quad   \mbox{for all} \ 1 \leqslant i \leqslant k \ .\]  
\end{enumerate}
\end{lemma}

\begin{proof}
Let us prove the result in the operadic case. By Lemma \ref{key lemma}, we have \[ \left(\mathrm{Ad}_{u} - \mathrm{id} \right)\left(\phi^{(k+1)} \right) = d^{\varphi_*}\left(s^{(k)} \circledcirc u^{-1}\right)\ ,\] since $s$ is an $\infty$-automorphism of $\phi$.  Since $\mathrm{Ad}_{u} - \mathrm{id}$ is invertible modulo $\mathcal{F}^{k+2} \mathfrak{g}$, we have \[\phi^{(k+1)} \equiv d^{\varphi_*} \left( \left(\mathrm{Ad}_u - \mathrm{id} \right)^{-1} \left(s^{(k)} \circledcirc u^{-1} \right) \right) \ , \pmod{ \mathcal{F}^{k+2} \mathfrak{g}}  \]	and $K^{k}_{\phi} = \left[k \phi^{(k+1)} \hbar^{k-1} \right] = 0$~. By Proposition \ref{Kaledin}, there exist $\psi \in \mathrm{MC}(\mathfrak{g}_{H})$ of the desired form and an $\infty$-isotopy \[f : (H, \phi) \rightsquigarrow  (H, \psi) \ . \] More precisely, we can take $f$ of the form $f = 1 + f^{(k)}$ such that \[\left(\mathrm{Ad}_u - \mathrm{id} \right) \left(f^{(k)} \right) = s^{(k)} \circledcirc u^{-1} \ ,\]  by the proof of Proposition \ref{Kaledin}. Finally, let us set $\tilde{s} \coloneqq f \circledcirc s \circledcirc f^{-1} ~. $ The inverse $f^{-1}$ can be explicitly computed using \cite[Theorem~ 10.4.2]{LodayVallette12}. It is an $\infty$-isotopy of the form \[f^{-1} = 1 	- 	 f^{(k)} + \cdots  \] This implies that $\tilde{s}^{(0)} = u$ and $\tilde{s}^{(i)} = 0$ for $1 \leqslant i \leqslant k-1$. Furthermore, we have \[\tilde{s}^{(k)}  = s^{(k)} + f^{(k)} \circledcirc u - u \star f^{(k)}  =  \left( s^{(k)}\circledcirc u^{-1} -  \left(\mathrm{Ad}_u - \mathrm{id} \right) \left( f^{(k)} \right) \right) \circledcirc u = 0 \ .  \] 
The proof is similar in the properadic case, using Lemma \ref{key lemma 2} and the fact that  \[ \left(\mathrm{Ad}_{u} - \mathrm{id} \right)\left(\phi^{(k+1)} \right) = d^{\varphi_*}\left(s^{(k)} \lhd u^{-1}\right)\ .\qedhere\] 
\end{proof}

\begin{proof}[Proof of Theorem \ref{C}]
The transferred structure comes equipped with $\infty$-quasi-isomorphisms
	\[\hbox{
		\begin{tikzpicture}
			
			\def\upshift{0.075}
			\def\downshift{0.075}
			\pgfmathsetmacro{\midshift}{0.005}
			
			\node[left] (x) at (0, 0) {$(A, d,\varphi)$};
			\node[right=1.5 cm of x] (y) {$(H(A),0,\varphi_t) \ .$};
			
			\draw[->] ($(x.east) + (0.1, \upshift)$) -- node[above]{\mbox{\tiny{$p_{\infty}$}}} ($(y.west) + (-0.1, \upshift)$);
			\draw[->] ($(y.west) + (-0.1, -\downshift)$) -- node[below]{\mbox{\tiny{$i_{\infty}$}}} ($(x.east) + (0.1, -\downshift)$);
			
	\end{tikzpicture}}
	\] The $\infty$-quasi-isomorphism $p_{\infty}$ extends the projection $p : A  \to H(A)$ and $i_{\infty}$ extends the inclusion $H(A) \to A$~. Consider the $\infty$-automorphism \[s \coloneqq  p_{\infty} \circledcirc v \circledcirc i_{\infty}  \ . \] By definition of a contraction, we have $s^{(0)}= p \circ v^{(0)} \circ i = u\ .$  Let us prove the point (1). Starting from $(H(A), \varphi_t)$ and the $\infty$-automorphism $s$~, one can apply Lemma \ref{induction} recursively for $k \leqslant n$. We obtain an $\Cobar \C$-algebra structure $\psi \in \mathrm{MC}(\mathfrak{g})$ such that 
	\[\psi^{(1)}  = \varphi_* \quad \mbox{and} \quad \psi^{(i)} = 0 \quad \mbox{for} \quad 2 \leqslant i \leqslant n +1 \ , \] and an $\infty$-isotopy $f : (H,\varphi_t) \rightsquigarrow  (H, \psi)$.  Thus, the $\Cobar \C$-algebra $(A, \varphi)$ is gauge $n$-formal. 
	
	\noindent Let us prove the point (2). As before, starting from $(H(A), \varphi_t)$ and the $\infty$-automorphism $s = p_{\infty} \circledcirc v \circledcirc i_{\infty} $~, one can apply Lemma \ref{induction} recursively. For all $k$, we obtain an $\Cobar \C$-algebra structure $\psi_n \in \mathrm{MC}(\mathfrak{g})$ such that 
	\[\psi_n^{(1)}  = \varphi_* \quad \mbox{and} \quad \psi^{(k)}_n = 0 \quad \mbox{for} \quad 2 \leqslant k \leqslant n +1 \ , \] and an $\infty$-isotopy $f_n : (H,\varphi_t) \rightsquigarrow  (H, \psi_n)$. In the operadic case, all the truncations $K^{n}_{\varphi_t}$ are zero by Proposition \ref{Kaledin}. This implies that $K_{\varphi_t} = 0$ by Proposition \ref{toutes nulles 2} and $(A,\varphi)$ is formal by Theorem \ref{B}. In the properadic case, all the truncations $K^{n}_{\varphi_t}$ are zero by Theorem \ref{Kaledin properadic}. This implies that $K_{\varphi_t} = 0$ by Proposition \ref{toutes nulles} and $(A,\varphi)$ is gauge formal by Theorem \ref{Kaledin properadic}. Let $f : \varphi_t \rightsquigarrow \varphi_*$ be an $\infty$-isotopy. For all $\tilde{\omega} \in \mathrm{Aut}\left(H(A), \varphi_* \right)$, we have $H\left(\omega^{(0)}\right) = \tilde{\omega}$ with \[\omega \coloneqq i_{\infty} \circledcirc f^{-1} \circledcirc u \circledcirc f \circledcirc p_{\infty} \ . \qedhere\] 
\end{proof}

\begin{corollary}\label{grading}
	Under Assumptions \ref{assump}, suppose that the weight-grading of $\C$ is given by the homological degree. Let $\alpha$ be a unit in $R$ and let $\vartheta$ be a non-zero rational number. Let $\sigma_{(\alpha,\vartheta)}$ be the automorphism of $(H(A), \varphi_*)$ defined by \[\sigma_{(\alpha,\vartheta)}(x) = \alpha^{\vartheta k} x \quad \mbox{for all } x \in H_k(A) \ .\] Suppose that $(A, \varphi)$ is $\vartheta$-pure, i.e. that $\sigma_{(\alpha,\vartheta)}$ admits a chain level lift.  
\begin{enumerate}
		\item If $\alpha^{\vartheta k} -1$ is a unit in $R$ for all $k \leqslant n + 1$, then $(A,\varphi)$ is gauge $n$-formal;
		\item If $R$ is a $\QQ$-algebra and if $\alpha^{\vartheta k} -1$ is a unit in $R$ for all $k$, then $(A,\varphi)$ is gauge formal. 
	\end{enumerate}
\end{corollary}

\begin{proof}
For all $\psi$ of degree $k$ in $\mathfrak{g}^{\varphi_*}$, we have $\mathrm{Ad}_{\sigma_{(\alpha,\vartheta)}}(\psi) = \alpha^{ \vartheta k } \psi \ .$
\begin{itemize}
	\item[$\centerdot$]  If $\alpha^{\vartheta k} -1$ is a unit in $R$ for all $k \leqslant n$, then $\mathrm{Ad}_{\sigma_{(\alpha,\vartheta)}} - \mathrm{id} $ is invertible modulo $\mathcal{F}^{n+2} \mathfrak{g}$. 
	\item[$\centerdot$] 
	If $\alpha^{\vartheta k}  -1$ is a unit in $R$ for all $k$, then $\mathrm{Ad}_{\sigma_{(\alpha,\vartheta)}} - \mathrm{id} $ is invertible.
\end{itemize}
The result now follows from Theorem \ref{C}.  
\end{proof}

\begin{examples}\label{arrangements} Let us present examples where the formality is established using $1$-purity, i.e. the existence of a chain level lift of the grading automorphism $\sigma_{(\alpha,1)}$. 
	\begin{enumerate}
		\item \textbf{Little disks operad}. The above type of result is used by Petersen in \cite{Petersen14} to give a proof of the formality of the little $2$-disks operad. He use the action of the Grothendieck-Teichmüller group at the chain level to exhibit a lift of the grading automorphism . Similar arguments where used in \cite{BdBH21} and \cite{CH22} to prove the formality of higher dimensional little disks operads and with torsion coefficients respectively. \medskip
		
		\item \textbf{Formality of cochains on BG}. In \cite{BG23}, Benson and Greenlees prove that for a compact Lie group $G$ with maximal torus $T$, if $|N_G(T)/T|$ is invertible in the field $k$ then the algebra of cochains $C^*(BG;k)$ is formal as an $A_\infty$-algebra. The spirit of their proof with the approximations of \cite[Section~3]{BG23} boils down to finding a lift of a grading automorphism and can be rephrased in terms of Corollary \ref{grading}. \medskip
		
		\item \textbf{Modular Koszul duality}. In \cite{RSW14}, Riche, Soergel, and Williamson prove an analogue of Koszul duality for the category $\mathcal{O}$ of a reductive group $G$ in certain positive characteristic. Their proof is base on a formality result of the dg associative algebra of extensions of parity sheaves on the flag variety, which is established by constructing a lift of the grading automorphism. \medskip

		\item \textbf{Complement of hyperplane arrangements}. Let $p$ and $\ell$ be two different prime numbers. Let $X$ be a complement of a hyperplane arrangement over $\mathbb{C}$. Let $\mathbb{Q}_{p} \hookrightarrow K$ be a finite extension. We denote by $q$ order of the residue field of the ring of integers of $K$. Suppose that $X$ is defined over $K$, i.e. that there exists an embedding $K \hookrightarrow \mathbb{C}$ and that there exists $\mathcal{X}$ a complement of a hyperplane arrangement over $K$ such that $\mathcal{X} \times _{K} \mathbb{C} \cong X \ .$ Drummond-Cole and Horel prove in \cite{DCH21} that the algebra  \[C^*_{\mathrm{sing}}(X_{\mathrm{an}}, \mathbb{Z}_{\ell})\] is $(s-1)$-formal, where $X_{\mathrm{an}}$ denotes the complex analytic space underlying $X_{\mathbb{C}} = X \times _{K} \mathbb{C}$ and where $s$ is the order of $q$ in $\mathbb{F}_{\ell}^{\times}$~. The idea is the following one. There exists a zig-zag of quasi-isomorphisms of dg associative algebras
		\[C^*_{\mathrm{sing}}(X_{\mathrm{an}}, \mathbb{Z}_{\ell}) \xleftarrow{\sim} C^*_{\mathrm{\acute{e}t}}(\mathcal{X}_{\mathbb{C}}, \mathbb{Z}_{\ell}) \xrightarrow{\sim} C^*_{\mathrm{\acute{e}t}}(\mathcal{X}_{\overline{K}}, \mathbb{Z}_{\ell}) \ ,\] see the proof of Theorem \ref{smooth proj} for more details. It was proved in \cite[Theorem 1']{Kim94} that the action of a Frobenius lift on $H_{\mathrm{\acute{e}t}}^*(\mathcal{X}_{\overline{K}}, \mathbb{Z}_{\ell})$ is given by $\sigma_{(q,1)}$. Thus the grading automorphism admits a chain level lift and we get the desired formality result. This result is really specific to complement of a hyperplane arrangements: in general, a Frobenius lift on the étale cohomology of a $K$-scheme is not a grading automorphism. Nonetheless, using the work of Deligne in the context of Weil conjectures, we prove bellow in Theorem \ref{smooth proj} that the action of a Frobenius lift on the étale cohomology of any smooth and proper $K$-scheme satisfies Theorem \ref{C}.   
	\end{enumerate}
\end{examples}

\begin{examples} Let us present formality examples relying on $\vartheta$-purity with $\vartheta \neq 1$. 
	\begin{enumerate}
			\item[(5)] \textbf{Gravity operad} In \cite{DH18}, Dupont and Horel prove the formality
	of the gravity operad of Getzler, which is an operad structure on the collection  \[ \lbrace H_{*} (\mathcal{M}_{0, n+1}) \rbrace_{n \in \mathbb{N}} \ ,\] where $\mathcal{M}_{0, n+1}$ denotes the moduli spaces of stable algebraic curves of genus zero with $n+1$ marked points. To do so, they exhibit a model $\mathcal{G}rav^{W'}$ which comes with an action of the Grothendieck-Teichmüller group. By exploiting this action, one can prove that $\mathcal{G}rav^{W'}$ is $2$-pure, i.e. that  $\sigma_{(\alpha,2)}$ admits a chain level lift, see \cite[Section~7.4]{CH20}. Thus it is formal, so does the gravity operad.  \medskip
	
			\item[(6)] \textbf{Good arrangement of codimension $c$ subspaces} Let $p$ be a prime number. Let $\mathbb{Q}_{p} \hookrightarrow K$ be a finite extension.  We denote by $q$ order of the residue field of the ring of integers of $K$. Let $V$ be a $d$-dimensional $K$-vector space and let $\{W_i\}_{i \in I}$ be a good arrangement of codimension $c$ subspaces, i.e. a family of subspaces of $V$ such that for all $i \in I$, the subspace $W_i$ is of codimension $c$ and the subspaces of $W_i$ \[ \lbrace  W_i \cap W_j\rbrace_{i \neq j}\]  form a good arrangement of codimension $c$ subspaces. By \cite{BE97}, the étale cochains \[ C^*_{\mathrm{\acute{e}t}}\left(V - \cup_i W_i, \mathbb{F}_{\ell} \right)\] is $(c /(2c-1))$-pure and thus gauge $n$-formal by Corollary \ref{grading} with \[n \coloneqq \lfloor (s - 1 ) (2c - 1) / c\rfloor \ , \] where $s$ is the order of $q$ in $\mathbb{F}_{\ell}^{\times}$~. In \cite{CH22}, Horel and Cirici establish the $n$-formality of this algebra in the case $c \geqslant 1$. 
	\end{enumerate}
\end{examples}

\noindent The following corollary of Theorem \ref{C} offers a more manageable criterion over a field.

\begin{corollary}\label{condition}
Under Assumptions \ref{assump}, suppose that $R$ is a field and that for all $k \geqslant 0$, the vector space $H_k(A)$ is finite dimensional. Suppose that the weight-grading on $\C$ is the homological degree. Suppose that there exists \[u \in \mathrm{Aut}\left(H(A),\varphi_*\right)\] that admits a chain level lift. For all $k$, we denote by $u_k$ the restriction of $u$ to $H_k(A)$ and $\mathrm{Spec}(u_k)$ its set of eigenvalues. If we have 
\begin{equation}\label{incompatibilité}
\mathrm{Spec}(u_{\alpha_1} \otimes \dots \otimes u_{\alpha_p}) \cap \mathrm{Spec}(u_{\beta_1} \otimes \dots \otimes u_{\beta_q}) = \varnothing
\end{equation}
for all $k$ (resp. $k \leqslant n + 1$), for all $p$-tuples $(\alpha_1, \dots, \alpha_p)$ and all $q$-tuples $(\beta_1, \dots, \beta_q)$ such that \[k = \alpha_1 + \cdots + \alpha_p - \beta_1 - \cdots - \beta_q \ , \]  then $(A,\varphi)$ is gauge formal (resp gauge $n$-formal).  
\end{corollary}

\begin{proof}
The eigenvalues of the restrictions of $\mathrm{Ad}_u$ to  degree $k$ elements are given by quotients \[ \frac{\lambda_{\alpha_1} \dots   \lambda_{\alpha_p}}{\lambda_{\beta_1} \dots   \lambda_{\beta_q}} \ , \quad   \lambda_{i} \in \mathrm{Spec}(u_{i})\] where $k = \alpha_1 + \cdots + \alpha_p - \beta_1 - \cdots - \beta_q $. By hypothesis (\ref{incompatibilité}), none of these quotients is equal to 1. Thus $\mathrm{Ad}_u - \mathrm{id}$ is invertible and the result follows from Theorem \ref{C}. The case of gauge $n$-formality if similar. 
\end{proof}

\noindent In \cite{Deligne71}, Deligne proved the formality with coefficients in $\mathbb{Q}_{\ell}$ of smooth proper varieties defined over a finite field, using the weights of the Frobenius action in the $\ell$-adic
cohomology and his solution of the Riemann hypothesis. This result arises from Kaledin classes theory using Corollary \ref{condition} and admits the following generalization:

\begin{theorem}\label{smooth proj} Let $X$ be a smooth and proper scheme over $\mathbb{C}$ and let $X_{\mathrm{an}}$ denotes the associated complex analytic space.	The dg associative algebra of singular cochains $C^{*}_{\mathrm{sing}}(X_{\mathrm{an}}, \mathbb{Q})$ is formal. 
\end{theorem}

\begin{proof}
Let $\ell$ be a fixed prime number. Let $R \subset \mathbb{C}$ be the ring of finite type generated by the coefficients of equations defining $X$. By inverting elements $R$, there exists a smooth and proper model $X_R$ over $R$ such that $X = X_R \otimes_{R} \mathbb{C}$ and $R$ is a local ring. Let $\mathfrak{m}$ be the maximal ideal of $R$. The field $R/\mathfrak{m}$ is of the form $\mathbb{F}_{q}$ where $q = p^r$. Without loss of generality, we can suppose that $p$ is different from $\ell$, even if it means also inverting $\ell$ in $R$. There exists a finite extension $K$ of $\mathbb{Q}_p$ and an embedding $R \hookrightarrow  \mathcal{O}_K$, such that \[\mathfrak{m}_{\mathcal{O}_K} \cap R = \mathfrak{m} \ .\] Let us consider the smooth and proper model $\mathcal{X} \coloneqq X_R \otimes_{R} \mathcal{O}_K$. By construction,  $X_R$ can be seen as a $K$-scheme of \emph{good reduction} i.e. for which there exists a smooth and proper model $\mathcal{X}$ over ${\mathcal{O}_K}$. There exists a zig-zag of quasi-isomorphisms of dg associative algebras \[C_{\mathrm{sing}}^*(X_{\mathrm{an}}, \mathbb{Q}_{\ell}) \overset{\sim}{\longleftarrow} C_{\mathrm{\acute{e}t}}^*(X_{\mathbb{C}}, \mathbb{Q}_{\ell}) \overset{\sim}{\longrightarrow} C_{\mathrm{\acute{e}t}}^*(X_{\overline{K}},\mathbb{Q}_{\ell}) \ ,\] where the left quasi-isomorphism is given by Artin's Theorem \cite[XI, Theorem 4.4, (iii)]{SGA4} and the right one follows from invariance of étale cochains by extension of algebraically closed fields \cite[VIII, Corollary 1.6]{SGA4}. Since $K$ is a $p$-adic field with residue field $\mathbb{F}_q$, there exists a surjective map \[\mathrm{Gal}(\overline{K}/ K) \twoheadrightarrow \mathrm{Gal}(\overline{\mathbb{F}}_q/\mathbb{F}_q) \ . \] The Galois group $\mathrm{Gal}(\overline{\mathbb{F}_q}/\mathbb{F}_q) $ is isomorphic to $\hat{\mathbb{Z}}$ generated by the Frobenius. Let us consider the smooth and proper $\mathbb{F}_q$-scheme given by $Y \coloneqq \mathcal{X} \otimes_{\mathcal{O}_K} \mathbb{F}_q$. By smooth and proper base change \cite[XVI, Corollary 2.2]{SGA4}, there exists a quasi-isomorphism \[C_{\mathrm{\acute{e}t}}^*(Y_{\overline{\mathbb{F}}_q}, \mathbb{Q}_{\ell}) \overset{\sim}{\longrightarrow} C_{\mathrm{\acute{e}t}}^*(X_{\overline{K}}, \mathbb{Q}_{\ell}) \] compatible with the Galois group actions on both sides. Let $u_k$ be the endomorphism of \[H_{\mathrm{\acute{e}t}}^k (Y_{\overline{\mathbb{F}}_q} \mathbb{Q}_{\ell})\] induced by the Frobenius. By Deligne's theorem \cite{Deligne74bis}, every eigenvalue $x$ of $u_k$ is a Weil number of weight $k$ i.e. then for any embedding $\iota : \overline{\mathbb{Q}}_{\ell} \hookrightarrow \mathbb{C}$~, we have $|\iota(x)| = q^{k/2} \ . $ In particular, an eigenvalue $x$ of $u_{\alpha_1} \otimes \dots \otimes u_{\alpha_p}$ satisfies $|\iota(x)| = q^{k/2}$ where $k = \alpha_1 + \cdots + \alpha_p$. Thus, the Frobenius satisfies the incompatibility condition (\ref{incompatibilité}) and it follows from Corollary \ref{condition} that $C_{\mathrm{\acute{e}t}}^*(Y_{\overline{\mathbb{F}}_q}, \mathbb{Q}_{\ell})$ is formal and so is $C_{\mathrm{sing}}^*(X_{\mathrm{an}}, \mathbb{Q}_{\ell})$. The result with coefficients in $\mathbb{Q}$ follows from Theorem \ref{formality descent}, using the fact that $X_{\mathrm{an}}$ has finite dimensional cohomology, see e.g. \cite[Corollary~5.25]{Voi02}. 
\end{proof}

\noindent The previous results extends to prove the formality of algebraic structures in smooth and proper schemes. This leads to the following theorem.

\begin{theorem}\label{smooth proj2}
Let $\mathbb{V}$ be a groupoid and let $\P$ be a $\mathbb{V}$-colored operad in sets. Let $p$ be a prime number. Let $K$ be a finite extension of $\mathbb{Q}_p$ and let $K \hookrightarrow \mathbb{C}$ be an embedding. Let $X$ be a $\P$-algebra in the category of smooth and proper schemes over $K$ of good reduction, i.e. for which there exists a smooth and proper model $\mathcal{X}$ over the ring of integers $\mathcal{O}_K$. The dg $\P$-algebra of singular chains $C_*^{\mathrm{sing}}(X_{\mathrm{an}}, \mathbb{Q})$ is formal.		
\end{theorem}

\begin{proof}
The proof is similar to the one of Theorem \ref{smooth proj}, using the same strategy than \cite[Section~3]{CH22}. The functor $C_{\mathrm{\acute{e}t}}^*(-, \mathbb{Q}_{\ell})$ descends to a symmetric monoidal $\infty$-functor
$\mathbf{C_{\mathrm{\acute{e}t}}^*}(-, \mathbb{Q}_{\ell})$ from the nerve $\mathbf{N}(\mathrm{Sch}_{K})$ to the $\infty$-category of chain complexes. Since finite type schemes have finitely generated étale cohomology, we can compose with the duality symmetric monoidal $\infty$-functor. This leads to a covariant
symmetric monoidal $\infty$-functor of étale chains $\mathbf{C^{\mathrm{\acute{e}t}}_*}(-, \mathbb{Q}_{\ell})$, which is equivalent to $\mathbf{C^{\mathrm{sing}}_*}(-, \mathbb{Q}_{\ell})$
as lax monoidal functors $\infty$-functor. The same arguments than before now apply to prove that for any $\P$-algebra $X$ in the category of smooth and proper schemes over $K$ of good reduction. Thus, the dg $\P$-algebra in the $\infty$-category of chain complexes $\mathbf{C^{\mathrm{sing}}_*}(X_{\mathrm{an}}, \mathbb{Q}_{\ell})$ is formal. By \cite[Corollary~2.4]{CH20}, this implies that $C^{\mathrm{sing}}_*(X_{\mathrm{an}}, \mathbb{Q}_{\ell})$ is formal as a dg $\P$-algebra, so does $C^{\mathrm{sing}}_*(X_{\mathrm{an}}, \mathbb{Q})$ by Theorem \ref{formality descent}. 
\end{proof}

\begin{examples}\textbf{Moduli spaces of stable algebraic curves}.  Let us consider the cyclic operad $\overline{\mathcal{M}}$ of moduli spaces of stable algebraic curves of genus $0$ with marked points. This is a cyclic operad in smooth and proper $\mathbb{Z}$-schemes. By Theorem \ref{smooth proj2}, the cyclic operad \[C_*^{\mathrm{sing}}(\overline{\mathcal{M}};\mathbb{Q})\] is formal, which allows us to recover \cite[Corollary 7.2.1]{GNPR05}. 
\end{examples}

\medskip

\bibliographystyle{alpha}
\bibliography{bib}

\end{document}